\documentclass{amsart}
\usepackage{amsmath}
\usepackage{framed,comment,enumerate}
\usepackage[pdftex]{graphicx}
\usepackage{epstopdf}
\usepackage{xcolor, framed}
\usepackage{color}

\def\R {\mathbb{R}}

\def\eps{\varepsilon}

\def\PosS{\{u>0\}}
\def\ConS{\{u=0\}}
\def\PPosS{\partial\{u>0\}}
\def\OEta{\Omega_\eta}

\def\CO{\mathcal{U}}
\def\BO{\mathcal{O}}
\def\En{\mathcal{E}}
\def\EnG{\mathcal{E}_\gamma}
\def\DStar{d^*_\gamma}

\def\URad{u_{\mathrm{rad}}}
\def\UAS{u_{\mathrm{as}}}

\def\Diri{|\nabla u|^2}

\def\MGa{\mathcal{M}_\gamma}
\def\MPGa{\mathcal{M}^+_\gamma}
\def\MMGa{\mathcal{M}^-_\gamma}

\def\SGa{\mathcal{S}_\gamma}
\def\SPGa{\mathcal{S}^+_\gamma}
\def\SMGa{\mathcal{S}^-_\gamma}

\def\DTheta{\frac{d}{d\theta}}

\def\PR{\frac{\partial}{\partial r}}

\def\BU{\bar{u}}
\def\BV{\bar{v}}

\def\LURad{L_{\URad}}
\def\CRad{c_{\mathrm{rad}}}

\def\LP{L_p}

\def\LU{L_u}
\def\LBU{L_{\bar{u}}}

\def\RIn{R_{\mathrm{in}}}
\def\UIn{U_{\mathrm{in}}}

\def\ROut{R_{\mathrm{out}}}
\def\UOut{U_{\mathrm{out}}}

\def\hem{\hspace{0.5em}}
\def\vem{\vspace{0.6em}}

\newtheorem{thm}{Theorem}[section]
\newtheorem{prop}[thm]{Proposition}

\newtheorem{cor}[thm]{Corollary}
\newtheorem{lem}[thm]{Lemma}
\theoremstyle{definition}
\newtheorem{defi}[thm]{Definition}
\numberwithin{equation}{section}

\theoremstyle{remark}
\newtheorem{rem}[thm]{Remark}

\title[Cones in the Alt-Phillips problem]{Stable and minimizing cones in the Alt-Phillips problem} 

\author{Ovidiu Savin}
\address{Department of Mathematics,	Columbia University, New York, USA}
\email{savin@math.columbia.edu}

\author{Hui Yu}
\address{Department of Mathematics,	National University of Singapore, Singapore}
\email{huiyu@nus.edu.sg}

\begin{document}

\begin{abstract}
We study homogeneous solutions to the Alt-Phillips problem when the exponent $\gamma$ is close to 1. 

In dimension $d\ge3$, we show that the radial cone is minimizing when $\gamma$ is close to 1.

In dimension $d \ge 4$, we construct an axially symmetric cone whose contact set has with positive density. We show that it is a global minimizer. It is analogous to the De Silva-Jerison \cite{DJ} cone for the Alt-Caffarelli functional which corresponds to exponent $\gamma=0$. 
The cone we construct bifurcates from another minimizing cone whose contact set has zero density,  obtained as the trivial extension of the radial solution. This second cone is analogous to a quadratic polynomial solution in the classical obstacle problem which corresponds to exponent $\gamma=1$. In particular our results show that, when $\gamma<1$ is sufficiently close to 1, there are axis symmetric cones that exhibit the properties of both end point cases $\gamma=0$ and $\gamma=1$.


\end{abstract}
\maketitle
\tableofcontents

\section{Introduction}
The one-phase Alt-Phillips problem studies non-negative minimizers of the \textit{Alt-Phillips functional}\footnote{For brevity, we refer to this functional as the \textit{$\gamma$-Alt-Phillips functional}.}
\begin{equation}
\label{EqnAP}
\En_\gamma(u;\Omega):= \int_\Omega \frac \Diri 2+u^\gamma\chi_{\PosS}, 
\end{equation} 
where $\Omega$ is a domain in the Euclidean space. This functional leads to interesting free boundary problems when the exponent $\gamma$ is in $(-2,2)$. 

Since its introduction by Phillips \cite{P}, this problem has received intense attention in the past few decades. 
Apart from its applications in porous catalysts \cite{A} and population dynamics \cite{GM}, the Alt-Phillips problem embeds two of the most well-studied elliptic free boundary problems as special cases, namely, the \textit{Alt-Caffarelli problem} when $\gamma=0$ \cite{V} and the \textit{classical obstacle problem} when $\gamma=1$ \cite{PSU}. It also provides the paradigm case for the study of general semilinear and quasilinear problems \cite{AS, DS1, FeRo, ReRo}. Recently, it was shown that minimizers of \eqref{EqnAP} approximate minimal surfaces as $\gamma\to-2$ \cite{DS3, DS4}.

In this work we focus on the study of axially symmetric cones when the exponent $\gamma$ is sufficiently close to $1$. The obstacle problem case $\gamma=1$ is quite special since the classification of cones is available in all dimensions \cite{C, PSU}. Beside the least energy solution $\frac 12 (x_n^+)^2$ and its rotations, there is the continuous family of quadratic polynomials $\frac 12 x^T A x$ with $A \ge 0$, $tr A=1$. 

It is natural to investigate the behavior of this large connected family of cones as we perturb $\gamma$ away from 1. Recently, in \cite{SY} we showed that only the radial quadratic polynomials and their trivial extensions  can appear as limits of cones as $\gamma \to 1$, suggesting a much more rigid family of conical solutions when $\gamma \ne 1$.

\vem

Under reasonable assumptions, the existence of a minimizer $u$ of \eqref{EqnAP} follows by the direct method. More subtle is the regularity of the \textit{free boundary} $\PPosS$, which separates the \textit{positive set} $\PosS$ from the \textit{contact set} $\ConS$.  Depending on the exponent $\gamma$, the free boundary exhibits different behaviors. In this work, \textit{we focus on the range $\gamma\in[0,2)$}. 

 Following either the classical approach in Alt-Caffarelli \cite{AC} and Alt-Phillips \cite{AP} or the modern approach in De Silva \cite{D} and De Silva-Savin \cite{DS1}, we know that the free boundary is locally a $C^{1,\alpha}$-hypersurface\footnote{This initial regularity can be bootstrapped to $C^\infty$. See  \cite{DS5, ReRo}.} under a flatness assumption.  This assumption can be verified by the blow-up analysis at a free boundary point. The blow-up limit is a homogeneous minimizer, known as a \textit{minimizing cone} \cite{W}. The flatness assumption is satisfied at the point if and only if the blow-up limit is a \textit{flat cone}\footnote{The flat cone is given by $U(x)=u(x\cdot e)$, where $e$ is a unit vector and $u$ is the global minimizer on $\R$ which vanishes on $(-\infty,0]$.}.
In particular,  the flatness assumption is always satisfied in $\R^d$ when all minimizing cones are flat in this dimension. 

For each exponent $\gamma\in[0,2)$, a dimension reduction argument gives a \textit{critical dimension} $d^*_\gamma$, defined as 
\begin{align}
\label{EqnCriticalDimension}
d^*_\gamma:&=\max\{d: \text{ Minimizing cones are flat in }\R^d\}\\
&=\min\{d:\text{There exists a non-flat minimizing cone in }\R^{d+1}\}.\nonumber
\end{align}
Non-flat minimizing cones are often referred to as \textit{singular minimizing cones}. As a consequence, the free boundary of a minimizer is smooth in $\R^d$ if $d\le\DStar$. In general dimensions, this smoothness holds outside a  \textit{singular set} of dimension at most $(d-\DStar-1)$ \cite{W, FeY}. This critical dimension $\DStar$ also gives information on Bernstein-type results  \cite{EFeY} and on generic regularity of free boundaries \cite{FeY, FeG}.

\vem

To know the value of  $\DStar$, we need to rule out singular minimizing cones in low dimensions and construct them in high dimensions. For $\gamma\ge1,$ the convexity of $\EnG$ in  \eqref{EqnAP} allows simple constructions of singular minimizing cones on the real line. This gives $\DStar=1$ for $\gamma\ge1$. For $\gamma\in[0,1)$, we do not have complete information in either of these two directions. 

For the Alt-Caffarelli problem ($\gamma=0$),  we know that minimizing cones are flat in $\R^d$ for $d\le 4$ by Caffarelli-Jerison-Kenig \cite{CJK} and Jerison-Savin \cite{JS}. Their proofs also work for stable cones. This gives the lower bound $d^*_0\ge 4.$ Assuming axial symmetry, singular stable cones can be ruled out up to dimension 6 \cite{CJK, JS, FeRo}.

In terms of the construction of singular cones, still for the case $\gamma=0$, Hong built homogeneous solutions to the Euler-Lagrange equation \cite{H}. However,  most of these are not stable or minimizing. Currently, the only known singular minimizing cone is the \textit{De Silva-Jerison cone}  in dimension $7$ \cite{DJ}. This gives the upper bound  $d^*_0\le 6$. 

The exact value of  $d^*_0$ remains  an open question. 

For general $\gamma\in(0,1)$, almost nothing is known, except that there is no singular minimizing cone in the plane \cite{AP}. This leads to a lower bound $\DStar\ge 2$.  Assuming axial symmetry, singular stable cones can be ruled out in low dimensions  when $\gamma$ is close to $0$ \cite{KSP}. On the other hand, there is no construction of stable or minimizing singular cones. No upper bound is known on $\DStar$ for any $\gamma\in(0,1)$.

\vem

Our first result determines the value of this critical dimension (see \eqref{EqnCriticalDimension}) as 
$$
\DStar=2 \hem\text{ for }\gamma\in(0,1) \text{ close to }1.
$$

To be precise, we have
\begin{thm}
\label{ThmMainRadial}
For $d\ge3$, there are dimensional constants $0<\gamma_d^1\le\gamma_d^2<1$ such that for $\EnG$ in  \eqref{EqnAP} in $\R^d$, we have:
\begin{enumerate}
\item{If $\gamma\in(0,\gamma_d^1)$, then the radial cone $u_{\mathrm{rad}}$ is unstable;}
\item{If $\gamma\in[\gamma_d^1,2)$, then $u_{\mathrm{rad}}$ is stable and  one-sided minimizing from above;}
\item{If $\gamma\in(\gamma_d^2,2)$, then  $u_{\mathrm{rad}}$ is minimizing.}
\end{enumerate}
\end{thm} 
The radial cone $\URad$ is given in \eqref{EqnURad}. For notions of minimality and one-sided minimality, see Definition \ref{DefMinimality}. For stability, see \eqref{EqnStabURad}. 

\begin{rem}
For more details on the stability of $\URad$, see Proposition \ref{PropStabOfURad}, where the value of $\gamma_d^1$ is given. Combining statements (1) and (2) in Theorem \ref{ThmMainRadial}, we see that this value  is sharp.

For more details on minimality and one-sided minimality, see Proposition \ref{PropMinimalityOfURad}, which gives an upper bound of $\gamma_d^2$. We do not know the sharp range of $\gamma$ for which $\URad$ is minimizing. 
\end{rem} 

\vem

Geometric properties of singular minimizing cones give information about the free boundary of a general minimizer. It is natural to first study cones with axial symmetry as they represent the most typical behavior as $\gamma$ is close to $1$ \cite{SY}. We are interested in the density of their contact sets. For a general minimizer $u$, this gives information on the  behavior when two regular patches of $\ConS$ collide with each other. 

For $\gamma\in [1,2)$, convexity implies that  the contact set of a singular cone has zero measure \cite{B,WY}. As a result, a cusp forms when two regular patches of $\ConS$ collide. On the other hand, for $\gamma\in(-2,0]$, we have density lower bounds on the contact set \cite{AC,DS3}. Two patches collide and  approximate a cone with non-empty interior. 

It is unclear what is the behavior when $\gamma\in(0,1)$.

For the Alt-Phillips problem when $\gamma$ is close to $1$,  the radial cone $\URad$ from Theorem \ref{ThmMainRadial}, together with its extensions to higher dimensions, are  minimizing cones whose contact sets have zero measure. This suggests cusp-like behavior when  patches of the contact set collide. 

However, this is not the only possibility. Our second result, to which most of the current work is devoted, gives the construction of singular minimizing cones  whose contact sets have positive measure. 
To the knowledge of the authors, this is the first free boundary problem where cusp-like and cone-like behaviors are both allowed. 

\begin{thm}
\label{ThmMainAS}
For $d\ge4$, there is a dimensional constant $0<\gamma_d^3<1$ such that for $\EnG$ in  \eqref{EqnAP} in $\R^{d}$ with $$\gamma_d^3<\gamma<1,$$ there is a singular minimizing cone $\UAS$ that is axially symmetric, and whose contact set $\{\UAS=0\}$ satisfies
$$
|\{\UAS=0\}\cap B_1|\sim_d (1-\gamma)^{\frac{d-1}{d-3}}.
$$ 
\end{thm} 
\begin{rem}
\label{RemSimAndO}
We use the \textit{big-O} notation. For two quantities $Q_1$ and $Q_2$, we write
$$
Q_1\le \BO_{c_1,c_2,\dots,c_m}(Q_2)
$$
if $Q_1\le C Q_2$ for some constant $C$ depending only on $c_1,c_2,\dots, c_m.$

When $Q_1\le C Q_2$ for a constant $C$ depending only on the dimension $d$, we simply write $Q_1\le\BO(Q_2)$.

We use the \textit{similarity} notation. We write 
$$
Q_1\sim_{c_1,c_2,\dots,c_m}Q_2
$$
if 
$$Q_1\le \BO_{c_1,c_2,\dots,c_m}(Q_2) \text{ and }Q_2\le\BO_{c_1,c_2,\dots,c_m}(Q_1).$$
\end{rem} 
\begin{rem}
The cone $\UAS$ is constructed explicitly, which gives an upper bound for the value of $\gamma_d^3$. It will be interesting to know the sharp range of $\gamma$ for which this construction is possible.
\end{rem} 

\vem 

For an axially symmetric cone, the Euler-Lagrange equation for $\EnG$ \eqref{EqnAP} reduces to an ordinary differential equation.  To construct the cone $\UAS$ in Theorem \ref{ThmMainAS}, we study this ODE in two regions, one near the equator of the unit sphere and one near the pole. Detailed expansions of solutions in these regions allow us to match the data at one point. This gives the solution $\UAS$. 

The minimality of  the cones  in Theorem \ref{ThmMainRadial} and  Theorem \ref{ThmMainAS} is more subtle. To this end, we construct foliations near the cones.

Around a minimizer, it is often possible to build foliations of minimizers. For minimal surfaces, this was done in the classical work of Hardt-Simon \cite{HS}. For the Alt-Caffarelli problem ($\gamma=0$), this was recently done by De Silva-Jerison-Shahgholian \cite{DJS} and Edelen-Spolaor-Velichkov \cite{EdSV}. 

Without assuming minimality, however, the construction of explicit families of sub and supersolutions on each side of the cone is quite delicate. In the context of minimal surfaces, this was achieved by Bombieri-De Giorgi-Giusti in the isotropic setting \cite{BDGG} and by Mooney-Yang in the anisotropic setting \cite{MY}. In the context of free boundary problems, the only known construction is for the Alt-Caffarelli problem in $\R^7$ \cite{DJ}. We remark there is a key step in showing the minimality of the De Silva-Jerison cone, that requires computer assistance. 

For our construction, we take advantage of the smallness of $|\gamma-1|$. This simplifies the situation since  the axially symmetric cone $\UAS$ becomes almost cylindrical at large scales. With these, we are able to give explicit foliations around both $\URad$ from Theorem \ref{ThmMainRadial} and $\UAS$ from Theorem \ref{ThmMainAS}.

\vem

This paper is \textbf{structured as follows}: 

In Section \ref{SectionPreliminaries}, we collect some preliminaries on the Alt-Phillips energy \eqref{EqnAP}. We briefly discuss the relation between minimality and the existence of foliations. We also recall some properties of the obstacle problem. 

 The remaining of the work is divided into two parts.  
 
 In the \textbf{first part}, we focus on the radial cone $\URad$ in Theorem \ref{ThmMainRadial}, not only for its own interest, but also to introduce some ideas that are useful for the treatment of axially symmetric cones. In Section \ref{SectionStabilityOfURad}, we study the stability of $\URad$. In Section \ref{SectionMinimalityOfURad}, we study its minimality. 

The \textbf{second part} focuses on the axially symmetric cone $\UAS$ in Theorem \ref{ThmMainAS}.  In Section \ref{SectionCAS}, we construct the cone $\UAS$. The remaining part is devoted to the construction of a foliation around $\UAS$. In Section \ref{SectionLowerFoliationAS}, we construct the lower leaves in this foliation. In Sections \ref{SectionUFAS}, we construct the upper leaves.

Several technical tools are left to the appendices.

\section{Preliminaries}
\label{SectionPreliminaries}
In this section, we collect some information about the Alt-Phillips functional. We also discuss briefly the calibration argument and some properties of solutions to the obstacle problem.

\subsection{The Alt-Phillips functional}
The $\gamma$-Alt-Phillips functional $\EnG$ is given in \eqref{EqnAP}. We work with exponents $\gamma\in(0,2)$. For each exponent, the corresponding \textit{scaling parameter} is  
\begin{equation}
\label{EqnScalingParameter}
\beta:=\frac{2}{2-\gamma}\in (1,+\infty),
\end{equation}
in relation to the scaling symmetry of the functional.
To be precise, for $r>0$, if we define 
\begin{equation}
\label{EqnRescaling}
u_r(x):=u(rx)/r^\beta,
\end{equation}
then the $\gamma$-Alt-Phillips functional in $\R^d$ satisfies
\begin{equation}
\label{EqnScalingSymmetryOfEnergy}
\EnG(u_r;B_1)=r^{2-2\beta-d}\cdot\EnG(u;B_r).
\end{equation} 

As a consequence, cones of $\gamma$-Alt-Phillips functional are $\beta$-homogeneous.
When $\gamma$ is close to $1$, the scaling parameter $\beta$ is close to $2$ with
$$
\beta-2=\frac{2}{2-\gamma}(\gamma-1).
$$

\vem
We need various notions of minimality:
\begin{defi}
\label{DefMinimality}
For the $\gamma$-Alt-Phillips functional \eqref{EqnAP}, we say that $u$ is a \textit{minimizer} in a domain $\Omega$, and write
$$
u\in\mathcal{M}_\gamma(\Omega),
$$
if 
\begin{equation}
\label{EqnLocalComparisonInMinimality}
\En_\gamma(u;\Omega)\le \En_\gamma(v;\Omega)
\end{equation}
whenever 
$$
v=u \text{ on }\partial\Omega.
$$

We say that $u$ is a \textit{one-sided minimizer from above} in $\Omega$, and write
$$
u\in\mathcal{M}^+_\gamma(\Omega),
$$ 
if the comparison \eqref{EqnLocalComparisonInMinimality} holds whenever 
$$
v=u \text{ on }\partial\Omega, \text{ and }\hem v\ge u \text{ in }\Omega.
$$ 

\textit{One-sided minimizers from below}  in $\Omega$  is defined in the symmetric way, and form the class $\mathcal{M}^-_\gamma(\Omega).$

\vem

We say that $u$ is a \textit{global minimizer/ global one-sided minimizer} in $\R^d$ if the corresponding property holds in 
$B_R \text{ for all }R>0.$ 

The class of global minimizers is denoted by $\mathcal{M}_\gamma(\R^d)$. The class of global one-sided minimizers is denoted by $\mathcal{M}^{\pm}_\gamma(\R^d).$
\end{defi}

 For a minimizer, it is straightforward to compute the Euler-Lagrange equation, understood in the viscosity sense \cite{DS1}.
\begin{defi}
\label{DefSolution}
For $\gamma\in(0,2)$ and $u\in C(\Omega)$ with $u\ge0$, we say that $u$ is a \textit{solution to the $\gamma$-Alt-Phillips problem} in $\Omega$, and write
$$
u\in\mathcal{S}_\gamma(\Omega)
$$
 if it satisfies $$
\Delta u=\gamma u^{\gamma-1}\hem \text{ in }\PosS\cap\Omega,
\hem\text{ and }\hem
|\nabla u|^2/u^\gamma=2\hem \text{ along }\PPosS\cap\Omega.
$$

We say that  $u$ is a \textit{subsolution} in $\Omega$, and write
$$
u\in\mathcal{S}^+_\gamma(\Omega)
$$
if it satisfies
$$
\Delta u\ge\gamma u^{\gamma-1} \hem\text{ in }\PosS\cap\Omega, \hem \text{ and }\hem|\nabla u|^2/u^\gamma\ge2\hem \text{ along }\PPosS\cap\Omega.
$$

A \textit{supersolution} $u$ is defined in the symmetric way. The class of supersolutions in $\Omega$ is denoted by $\mathcal{S}^-_\gamma(\Omega)$,

\end{defi}

The scaling symmetry in \eqref{EqnScalingSymmetryOfEnergy} leads to the same symmetry of solutions. For the rescaled function $u_r$ in \eqref{EqnRescaling}, we have
\begin{prop}
\label{PropSymmetryOfEquation}
 If $u\in\mathcal{S}(B_r)$, then $u_r\in\mathcal{S}(B_1)$.

Similar properties hold for $\mathcal{S}^\pm_\gamma(B_r).$
\end{prop} 

The $\gamma$-Alt-Phillips functional is convex when $\gamma\in[1,2)$. As a result, solutions are minimizers.
\begin{prop}
\label{PropSolutionsAreMinimizers}
Suppose that $u\in\SGa(B_1)$ for $\gamma\in[1,2)$, then $u\in\MGa(B_1)$.
\end{prop} 

Another feature of the Alt-Phillips problem when $\gamma\ge 1$ is the convexity of global solutions:
\begin{prop}[\cite{B, C, PSU, WY}]
\label{PropConvexityOfSolutions}
Suppose that $u\in\SGa(\R^d)$ for $\gamma\in[1,2)$ with $0\in\partial\PosS$, then 
$$
D^2u\ge0 \text{ in }\R^d.
$$

If $u$ is a minimizing cone for $\gamma\in[1,2)$, then either $u$ is the flat cone or $\{u=0\}$ has zero measure.
\end{prop} 
Recall that a flat cone is, up to a rotation, of the form $c_\gamma[(x_1)_+]^\beta$.
\subsection{Foliation and minimality} 
\label{SubsectionFoliationsAndMinimality}
We discuss the relation between the minimality of a solution and the existence of a foliation around it. Since the argument is classical, the proofs are only sketched. For more details of this argument in various contexts, interested readers may consult \cite{AAC, BDGG, CEF, EFeY, MY}.

In the following, we use notations $\SGa,\SPGa$ and $\SMGa$ from Definition \ref{DefSolution}.
\begin{defi}
\label{DefUpperFoliation}
Suppose that $u\in\SGa(\R^d)$.

An \textit{upper foliation} of $u$  is a family of functions $\{\Psi_t\}_{t\in(0,+\infty)}$ satisfying the following properties:
\begin{enumerate}
\item{For each $t\in(0,+\infty)$, we have 
$$
\Psi_t\ge u \text{ in }\R^d, \text{ and }\hem \Psi_t>u \text{ in }\overline{\PosS};
$$}
\item{For each $t\in(0,+\infty)$, $\Psi_t\in\SMGa(\R^d)$;}
\item{The map $(t,x)\mapsto\Psi_t(x)$ is continuous;}
\item{As $t\to 0$, we have 
$$
\Psi_t\to u\hem\text{locally uniformly on }\PosS;
$$ and}
\item{As $t\to+\infty$, we have 
$$
\Psi_t\to+\infty \hem\text{locally uniformly.}
$$}
\end{enumerate}

Each function $\Psi_t$ is called an \textit{upper leaf}.
\end{defi}

The existence of an upper foliation gives one-sided minimality from above.
\begin{prop}
\label{PropMinimalityFromAbove}
Suppose that $u\in\SGa(\R^d)$ has an upper foliation  $\{\Psi_t\}_{t\in(0,+\infty)}$ and satisfies
$
\En_\gamma(u;B_1)<+\infty.
$

Then $u\in\MPGa(B_1)$.
\end{prop} 
Recall the class  $\MPGa$ from  Definition \ref{DefMinimality}. 

\begin{proof}
By the direct method, we can find $w$, a minimizer of $\EnG(\cdot)$ in the family
$
\{v:\hem v=u \hem\text{along }\partial B_1, \text{ and }\hem v\ge u \text{ in } B_1\}.
$
We have $w\in\SGa(\{w>u\})$. Define
$
\tilde{w}:=\max\{u,w\},
$
then $\tilde{w}\in\SPGa(B_1)$.

It suffices to show that 
$
\tilde{w}=u \text{ in }B_1.
$

\vem

Suppose not, then the set $\{\tilde w>u\}$ is non-empty and contained in $\overline{B_1}$. By the definition of an upper foliation, we can find $t_0\in(0,+\infty)$ such that 
$$
\Psi_{t_0}\ge\tilde w \text{ in }B_1, \text{ and }\hem\Psi_{t_0}(x_0)=\tilde w(x_0) \text{ at some }x_0\in\overline{\{\tilde w>0\}}\cap B_1.
$$
This is a contradiction as a subsolution cannot be touched from above by a supersolution \cite{DS1}.
\end{proof} 

Symmetric to Definition \ref{DefUpperFoliation}, we have
\begin{defi}
\label{DefLowerFoliation}
Suppose that $u\in\SGa(\R^d)$.

An \textit{lower foliation} of $u$ is a family of functions $\{\Phi_t\}_{t\in(0,+\infty)}$ satisfying the following properties:
\begin{enumerate}
\item{For each $t\in(0,+\infty)$, we have 
$$
0\le\Phi_t\le u \text{ in }\R^d, \text{ and }\hem \Phi_t<u \text{ in }\overline{\{\Phi_t>0\}};
$$}
\item{For each $t\in(0,+\infty)$, $\Phi_t\in\SPGa(\R^d)$;}
\item{The map $(t,x)\mapsto\Phi_t(x)$ is continuous;}
\item{The map $t\mapsto \overline{\{\Phi_t>0\}}$ is locally continuous in the Hausdorff distance;}
\item{As $t\to 0$, we have 
$$
\Phi_t\to u \hem\text{locally uniformly in $\PosS$};
$$ and}
\item{Given a compact set $K$, there is $t_K$ such that 
$$
\Phi_t=0 \hem\text{ on $K$ for all }t>t_K.
$$}
\end{enumerate}

Each function $\Phi_t$ is called a \textit{lower leaf}.
\end{defi}

Similar to Proposition \ref{PropMinimalityFromAbove}, we have
\begin{prop}
\label{PropMinimalityFromBelow}
Suppose that $u\in\SGa(\R^d)$ has a lower foliation  and satisfies
$
\En_\gamma(u; B_1)<+\infty.
$

Then $u\in\MMGa(B_1)$.
\end{prop} 

Combining an upper foliation and a lower foliation, we have a foliation.
\begin{defi}
\label{DefFoliation}
Suppose that $u\in\SGa(\R^d)$. 

A \textit{foliation }of $u$  is a family of functions $\{\Psi_t\}_{t\in(0,+\infty)}\cup\{\Phi_t\}_{t\in(0,+\infty)}$, where $\{\Psi_t\}_{t\in(0,+\infty)}$ is an upper foliation and $\{\Phi_t\}_{t\in(0,+\infty)}$ is a lower foliation. 
\end{defi}

As a consequence of Proposition \ref{PropMinimalityFromAbove} and Proposition \ref{PropMinimalityFromBelow}, we have
\begin{thm}
\label{ThmMinimality}
Suppose that $u\in\SGa(\R^d)$ has a foliation and satisfies
$
\En_\gamma(u; B_1)<+\infty.
$

Then $u\in\MGa(B_1)$.
\end{thm}

\subsection{Solutions to the obstacle problem}
For a domain $\Omega\subset\R^d$, the \textit{classical obstacle problem} studies solutions to the following
\begin{equation}
\label{EqnClassicalObstacleProblem}
\Delta u=\chi_{\PosS},\text{ and }\hem u\ge0 \text{ in }\Omega.
\end{equation} 
For our problem \eqref{EqnAP}, this corresponds to the case when $\gamma=1$ and the scaling parameter $\beta=2$ (see \eqref{EqnScalingParameter}).

We are particularly interested in radial solutions. For a solution of the form $P=P(|x|)$, we have
$$
P''+\frac{d-1}{r}P'=\chi_{\{P>0\}} \text{ on }(0,+\infty), \hem\text{ and }P\ge0 \text{ on }[0,+\infty).
$$
For $\lambda>0$, let $P_\lambda$ denote the solution whose contact set is
$$
\{P_\lambda=0\}=\overline{B}_\lambda,
$$ 
then $P_\lambda$ is given by
\begin{equation}
\label{EqnPLambda}
P_\lambda=\begin{cases}
\frac{1}{2d}|x|^2-\frac{\lambda^2}{2(d-2)}+\frac{\lambda^d}{d(d-2)}|x|^{-(d-2)} &\text{ for }|x|\ge\lambda,\\
0 &\text{ for }|x|<\lambda.
\end{cases}
\end{equation}

\section{Stability  of the radial cone}
\label{SectionStabilityOfURad}
In $\R^d$, the radial cone for the Alt-Phillips problem is 
\begin{equation}
\label{EqnURad}
\URad(x):=c_{\mathrm{rad}}|x|^\beta,
\end{equation} 
with the scaling parameter $\beta$ from \eqref{EqnScalingParameter}. The coefficient satisfies
$$
c_{\mathrm{rad}}^{2-\gamma}=\frac{\gamma}{\beta(d+\beta-2)},
$$
so that $\URad\in\SGa(\R^d)$  as in Definition \ref{DefSolution}.
\vem

In this section, we address the stability of $\URad$,  in terms of the second-order expansion of the energy\footnote{Our expansion  relies on the positivity of $\URad$  in  $\R^d\backslash\{0\}$. For the stability condition on general solutions, see the work by Karakhanyan and Sanz-Perela \cite{KSP}.} in \eqref{EqnAP}. 

To be precise, we say that $\URad$ is \textit{stable} if we have
\begin{equation}
\label{EqnStabURad}
\lim_{t\to0}[\EnG(\URad+t\varphi)-\EnG(\URad)]/t^2\ge0
\end{equation}
 for all $\varphi\in C^\infty_c(\R^d\backslash\{0\})$.
 
Since $\URad\in\SGa(\R^d)$, for $\varphi\in C^\infty_c(\R^d\backslash\{0\})$ and $t$ small, we have 
 \begin{equation*}
\EnG(\URad+t\varphi)-\EnG(\URad)=\frac{t^2}{2}[\int|\nabla\varphi|^2+\gamma(\gamma-1)\URad^{\gamma-2}\varphi^2]+\BO(t^3).
\end{equation*}
 As a result, the radial cone $\URad$ is stable if and only if, for all $\varphi\in C^\infty_c(\R^d\backslash\{0\})$,   
\begin{equation}
\label{EqnStableInEnergyRadial}
 \int\varphi\cdot \LURad\varphi=-\int|\nabla\varphi|^2+\gamma(\gamma-1)\URad^{\gamma-2}\varphi^2\le0
\end{equation} 
  where $\LURad$ is the \textit{linearized operator around} $\URad$:
 \begin{equation}
\label{EqnLinearizedOperatorAroundURad}
L_{\URad}\varphi:=\Delta \varphi-\gamma(\gamma-1)\URad^{\gamma-2}\varphi=\Delta \varphi-(\gamma-1)\beta(d+\beta-2)\varphi/|x|^2.
\end{equation}

The \textit{linearized equation} reads
\begin{equation}
\label{EqnLinearizedEquationRadial}
\LURad\varphi=0\hem \text{ in }\R^d\backslash\{0\}.
\end{equation}
Restricting to radial functions $\varphi=\varphi(r)$, this equation reduces to an ODE
\begin{equation}
\label{EqnRadialODE}
\varphi''+\frac{d-1}{r}\varphi'=(\gamma-1)\beta(d+\beta-2)\frac{\varphi}{r^2}\hem \text{ on }(0,+\infty).
\end{equation}   
For functions of the form $\varphi(r)=r^\alpha$, this imposes a condition on the exponent $\alpha$:
\begin{equation}
\label{EqnThePolynomial}
\alpha^2+(d-2)\alpha+(1-\gamma)\beta(d+\beta-2)=0.
\end{equation} 

 \vem
 
Following Jerison-Savin \cite{JS},  the stability of $\URad$ depends on the discriminant of \eqref{EqnThePolynomial}, namely, 
\begin{equation}
\label{EqnDiscriminant}
\Delta(d,\gamma):=(d-2)^2-4\beta(d+\beta-2)(1-\gamma).
\end{equation}
With this,  the \textbf{main result} in this section is
\begin{prop}
\label{PropStabOfURad}Suppose that $d\ge 3$ and $\gamma\in(0,2)$.

If $\Delta(d,\gamma)<0$, then the radial cone $\URad$ is unstable.

If $\Delta(d,\gamma)\ge0$, then  $\URad$ is stable. 
\end{prop} 


\vem

We prove the first part of Proposition \ref{PropStabOfURad}:
\begin{proof}[Proof of instability when $\Delta<0$] Under this assumption, the linearized equation \eqref{EqnRadialODE} has a solution
$$
\varphi(r)=r^{-\frac{d-2}{2}}\cos(\frac12\sqrt{|\Delta|}\ln r).
$$
Its oscillatory nature  gives two radii $0<r_1<r_2$ such that 
$$
\varphi=-1\hem \text{ at }r_1 \text{ and }r_2, \text{ and }\hem \varphi>-1 \hem\text{ on }(r_1,r_2).
$$
Define a test function $\tilde\varphi$ as
$$
\tilde\varphi(x):=\varphi(|x|)+1>0 \text{ if }r_1<|x|<r_2, \text{ and }\hem\tilde\varphi=0 \text{ otherwise}.
$$
Then $\tilde\varphi$ satisfies
\begin{equation*}
\LURad\tilde\varphi=\gamma(1-\gamma)\URad^{\gamma-2} \text{ in }B_{r_2}\backslash \overline{B_{r_1}},
\end{equation*} 
Since $\gamma\in(0,1)$ when $\Delta(d,\gamma)<0$, this implies
$$
\tilde\varphi\cdot\LURad(\tilde\varphi)=\gamma(1-\gamma)\URad^{\gamma-2}\tilde\varphi\ge 0 \textit{ in }\R^d
$$
with strict inequality in $B_{r_2}\backslash\overline{B_{r_1}}$.

This gives the instability of  $\URad$ according to \eqref{EqnStableInEnergyRadial}.
\end{proof}

\vem

We now turn to the case when the discriminant in \eqref{EqnDiscriminant} satisfies
\begin{equation}
\label{EqnPositiveDiscriminant}
\Delta(d,\gamma)\ge 0,
\end{equation}
and show the second part of Proposition \ref{PropStabOfURad}.  With Proposition \ref{PropSolutionsAreMinimizers}, it suffices to consider $\gamma\in(0,1)$.

With \eqref{EqnPositiveDiscriminant}, the polynomial \eqref{EqnThePolynomial} has a real root
\begin{equation}
\label{EqnTheConstantAlpha}
\alpha=-\frac{d-2}{2}+\frac{\sqrt{\Delta}}{2}.
\end{equation}
As a result, the linearized equation \eqref{EqnRadialODE} has a positive solution 
\begin{equation}
\label{EqnAPositiveSolutionRadial}
g(r):=r^{\alpha}=r^{-\frac{d-2}{2}+\frac{\sqrt{\Delta}}{2}}.
\end{equation}

This leads to the maximum principle of the linearized operator $\LURad$ in   \eqref{EqnLinearizedOperatorAroundURad}:
\begin{lem}
\label{LemMaximumPrincipleRadial}
Suppose that $d\ge3$ and $\gamma\in(0,1)$ satisfy 
$
\Delta(d,\gamma)\ge0.
$

Let $w$ be a subsolution to the linearized equation \eqref{EqnLinearizedEquationRadial} in $B_{2}\backslash\overline{B_1}$, namely,
$$
\LURad(w)\ge0 \hem \text{ in }B_2\backslash\overline{B_1}.
$$
If $w\le 0$ on $\partial B_2\cup\partial B_1$, then 
$$
w\le 0 \text{ in }B_2\backslash \overline{B_1}.
$$
\end{lem} 
\begin{proof}
Suppose not, then the ratio $G:=w/g$ has a positive maximum at some point $p$ in $B_2\backslash\overline{B_1}$. Here $g$ is the function from \eqref{EqnAPositiveSolutionRadial}. That is,
$$
0<M:=\max_{\overline{B_2\backslash B_1}}G=G(p) \text{ for some }p\in B_2\backslash\overline{B_1}.
$$

\vem

Define $W:=w-Mg$, then we have
$$
W\le 0\hem \text{ in }B_2\backslash \overline{B_1}, \text{ and }W(p)=0.
$$
Meanwhile, we have
$$
\Delta W=\Delta w-M\Delta g\ge \gamma(\gamma-1)\URad^{\gamma-2}(w-Mg)\ge0\hem \text{ in }B_2\backslash\overline{B_1}.
$$
For the last inequality, we used $\gamma<1$.

The strong maximum principle of sub-harmonic functions forces
$$
w-Mg=W\equiv0 \text{ in }B_2\backslash B_1.
$$
This contradicts $w\le 0$ on $\partial B_2\cup\partial B_1.$
\end{proof} 

The following corollary is a useful tool to rule out bounded subsolutions in the punctured disk.
\begin{cor}
\label{CorBoundedSubsolutions}
Suppose that $d\ge3$ and $\gamma\in(0,1)$ satisfy 
$
\Delta(d,\gamma)\ge0.
$

Let $w$ be a subsolution to the linearized equation \eqref{EqnLinearizedEquationRadial} in $B_{1}\backslash\{0\}$, namely,
$$
\LURad(w)\ge0 \hem \text{ in }B_1\backslash\{0\}.
$$
If $w$ satisfies
$$w\le 0\text{ on }\partial B_1, \text{ and }\sup_{B_1\backslash\{0\}}w<+\infty,$$
then 
$$
w\le 0 \text{ in }B_1\backslash\{0\}.
$$
\end{cor} 
\begin{proof}
Under our assumption $\gamma\in(0,1)$, the discriminant $\Delta(d,\gamma)$ in \eqref{EqnDiscriminant}
satisfies $\Delta(d,\gamma)<(d-2)^2$. As a consequence, the exponent $\alpha$ in \eqref{EqnTheConstantAlpha} is strictly negative. 

Given any $\eps>0$, this gives
$$
w-\eps g\le 0 \text{ on }\partial B_1\cup\partial B_{r_\eps},
$$
where $g$ is the radial solution in \eqref{EqnAPositiveSolutionRadial}, and $r_\eps$ is given by
$$
r_\eps:=(\sup w/\eps)^{\frac{1}{\alpha}}.
$$

As a result, we can apply Lemma \ref{LemMaximumPrincipleRadial} to $w-\eps g$ in $B_1\backslash\overline{B_{r_\eps}}$ to conclude 
$$
w\le \eps g \text{ in }B_1\backslash\overline{B_{r_\eps}}.
$$
Note that $r_\eps\to0$ as $\eps\to 0$, we get the desired conclusion. 
\end{proof}

With these preparations, we give the proof of the stability of $\URad$ if $\Delta\ge0$:
\begin{proof}[Proof of stability when $\Delta\ge0$]
Suppose not, then \eqref{EqnStableInEnergyRadial} gives a function $\varphi$, with support in $\overline{B_R\backslash B_r}$ for some $R>r>0$ such that 
\begin{equation}
\label{EqnContradictionToInStability}
\int|\nabla\varphi|^2+\gamma(\gamma-1)\URad^{\gamma-2}\varphi^2<0.
\end{equation}

Define a functional 
$$
F(v):=\int|\nabla v|^2+\gamma(\gamma-1)\URad^{\gamma-2}v^2
$$
over  the class 
$$
K:=\{v\in H_0^1(B_R\backslash\overline{B_r}): \|v\|_{\mathcal{L}^2}\le 1\}.
$$ 
The direct method gives a minimizer, $w$, in $K$. 

Without loss of generality, we can assume
\begin{equation}
\label{EqnPositivityOfAnImaginedW}
w\ge0 \text{ in }B_R\backslash B_r.
\end{equation}
With \eqref{EqnContradictionToInStability}, we have 
\begin{equation}
\label{EqnNegativeEnergyForW}
F(w)\le F(\varphi/\|\varphi\|_{\mathcal{L}^2})<0.
\end{equation} 
As a result, the Euler-Lagrange equation for $w$ implies
$$
\LURad(w)\ge0 \text{ in }B_R\backslash\overline{B_r}.
$$

Since $w$ is supported in $\overline{B_R\backslash B_r}$, Lemma \ref{LemMaximumPrincipleRadial} implies
$$
w\le 0\text{ in }B_R\backslash B_r.
$$
Together with \eqref{EqnPositivityOfAnImaginedW}, we have $w=0$ in $B_R\backslash B_r$, contradicting \eqref{EqnNegativeEnergyForW}.
\end{proof}

\section{Minimality of the radial cone}
\label{SectionMinimalityOfURad}
In this section, we address the minimality of the radial cone $\URad$ in \eqref{EqnURad}. With classes of minimizers and one-sided minimizers $\MGa$ and $\MPGa$ in  Definition \ref{DefMinimality}, the \textbf{main result} of this section reads
\begin{prop}
\label{PropMinimalityOfURad}
Suppose that $d\ge3$ and $\gamma\in(0,2)$. 

\begin{enumerate}
\item{If $\Delta(d,\gamma)\ge 0$, then  $\URad\in\MPGa(\R^d)$. }
\item{If $\gamma\in(1-\frac{(d-2)^2}{64d^2},2)$, then $\URad\in\MGa(\R^d)$. }
\end{enumerate}
\end{prop} 
Recall the discriminant $\Delta(d,\gamma)$ from \eqref{EqnDiscriminant}.  
\begin{rem}
For $\gamma\in[1,2)$, the minimality of $\URad$ follows from Proposition \ref{PropSolutionsAreMinimizers}. In this section, we consider only the range $$\gamma\in(0,1).$$

The lower bound for $\gamma$ in Statement (2) is sufficient for our purpose (See Lemma \ref{LemGluedSubsolutionRadial}). We do not know the sharp range of $\gamma$ for which the minimality of $u$ holds.
\end{rem}

With Theorem \ref{ThmMinimality}, we need a foliation around $\URad$.  In Subsection \ref{SubsectionUpperFoliationRadial}, we construct an upper foliation. In Subsection \ref{SubsectionLowerFoliationURad},  we build a lower foliation. At the beginning of each subsection, we give a \textbf{brief explanation of the strategy}. 
\subsection{Upper foliation of the radial cone}
\label{SubsectionUpperFoliationRadial}
In this subsection, we construct an upper foliation as in Definition \ref{DefUpperFoliation} for the radial cone $\URad$ in \eqref{EqnURad}. 
 Proposition \ref{PropMinimalityFromAbove} then implies the one-sided minimality of $\URad$  in Statement (1) of Proposition \ref{PropMinimalityOfURad}.
 
In particular, we need one supersolution to the Alt-Phillips problem (see Definition \ref{DefSolution}) that lies above $\URad$. In the radial setting, this problem reduces to the ODE in \eqref{EqnODEForUpperLeafRadial}. 

To show that its solution stays larger than $\URad$ (see \eqref{EqnDIsNonPositive}), we need the maximum principle  in Lemma \ref{LemMaximumPrincipleRadial}. This requires the assumption $\Delta(d,\gamma)\ge0$ in Statement (1) of Proposition \ref{PropMinimalityOfURad}. This is the same range of $\gamma$ for the stability of $\URad$ (see Proposition \ref{PropStabOfURad}).

\vem
Now we start the construction of the upper foliation. 
Leaves in this upper foliation arise as rescalings of   the solution to the following initial value problem:
\begin{equation}
\label{EqnODEForUpperLeafRadial}
v''+\frac{d-1}{r}v'=\gamma v^{\gamma-1} \text{ on }(0,+\infty),\hem\text{ and }
v(0)=1, \hem v'(0)=0.
\end{equation} 
Identifying $v$ with the function $x\in\R^d\mapsto v(|x|)$, 
we see that
\begin{equation}
\label{EqnPDEFromODERadial}
\Delta v=\gamma v^{\gamma-1} \text{ in }\R^d.
\end{equation}

The analysis of this  problem is summarized in the following:
\begin{prop}
\label{PropIVPRadial}
For $d\ge 3$ and $\gamma\in(0,1)$ with $\Delta(d,\gamma)\ge0$, there is a unique solution $v$ to \eqref{EqnODEForUpperLeafRadial} on $[0,+\infty)$. 

Moreover, this solution satisfies
$$
0<v-\URad\le 1,  \text{ and }v\ge 1 \text{ on }[0,+\infty).
$$
\end{prop} 

\begin{proof} 
For small $\eps>0$, we start with the  perturbed problem 
\begin{equation}
\label{EqnPerturbedIVPRadial}
v_\eps''+\frac{d-1}{r+\eps}v_\eps'=\gamma v_\eps^{\gamma-1} \text{ for }r>0,\hem\text{ and }
v_\eps(0)=1, \hem v_\eps'(0)=0.
\end{equation}
There is a smooth solution on $(0,R_\eps)$, with $R_\eps$ being finite only when
$$
\lim_{r\to R_\eps^-}v_\eps(r)=0.
$$

The perturbed equation is equivalent to 
$$
(r+\eps)^{1-d}[(r+\eps)^{d-1}v_\eps']'=\gamma v_\eps^{\gamma-1}.
$$
With $v_\eps(0)=1$, the right-hand side is initially positive. Hence $(r+\eps)^{d-1}v_\eps'$ is strictly increasing. With $v_\eps'(0)=0$, we have 
\begin{equation}
\label{EqnPerturbedSolutionMonotone}
v_\eps'>0\hem \text{ and }v_\eps>1 \text{ for }r>0.
\end{equation}
In particular, we have
$$
R_\eps=+\infty.
$$

With \eqref{EqnPerturbedSolutionMonotone},  it follows from  \eqref{EqnPerturbedIVPRadial}  that
$
v_\eps''<\gamma \text{ on }(0,+\infty).
$
 As a consequence, we have
$$
0<v_\eps'(r)<\gamma r, \text{ and }\hem 1<v_\eps(r)<1+\frac{1}{2}\gamma r^2 \hem\text{ for }r>0.
$$
This gives enough compactness of the family $\{v_\eps\}$. We can extract a subsequential limit, $v$, as $\eps\to 0$. This limit $v$ solves \eqref{EqnODEForUpperLeafRadial}, together with 
$$
v'\ge0, \text{ and }\hem v\ge 1 \text{ on }[0,+\infty).
$$

\vem

It remains to establish the bounds on the difference
$
D:=v-\URad.
$

We first show that $D>0$ on $(0,+\infty)$. 

Suppose not, with $D(0)>0$, we find $R<+\infty$ such that 
\begin{equation}
\label{EqnAssumptionOnDRadial}
D>0 \text{ on }[0,R),\text{ and } D(R)=0.
\end{equation}
Using the equations for $v$ and $\URad$ (see \eqref{EqnPDEFromODERadial} and Definition \ref{DefSolution}),  we have
$$
\Delta D=\gamma(v^{\gamma-1}-\URad^{\gamma-1})\ge\gamma(\gamma-1)\URad^{\gamma-2}(v-\URad) \hem\text{ in }B_R\backslash\{0\}.
$$
For the inequality, we used the convexity of $t\mapsto t^{\gamma-1}$ and our assumption that $v>\URad$ in $B_{R}$.

As a result, we have
$$
\LURad(D)\ge 0\text{ in }B_R\backslash\{0\}, \text{ and } D=0 \text{ on }\partial B_R,
$$
where $\LURad$ is the linearized operator  in \eqref{EqnLinearizedOperatorAroundURad}. Corollary \ref{CorBoundedSubsolutions} gives
\begin{equation}
\label{EqnDIsNonPositive}
D\le 0 \text{ in }B_R\backslash\{0\},
\end{equation}
contradicting our assumption \eqref{EqnAssumptionOnDRadial}. This is the only place where we need the assumption $\Delta(d,\gamma)\ge 0$. 

Therefore, we have
$$
v-\URad=D>0 \text{ on }[0,+\infty).
$$

\vem

We now show the upper bound $D\le 1$ on $[0,+\infty)$. 

Since $v$ and $\URad$ both solve the ODE in \eqref{EqnODEForUpperLeafRadial}, with $D=v-u>0$ and $\gamma\in(0,1)$,  we have 
$$
r^{1-d}[r^{d-1}D']'=D''+\frac{d-1}{r}D'=\gamma(v^{\gamma-1}-\URad^{\gamma-1})<0 \text{ on }(0,+\infty).
$$
With $D'(0)=0$, we have
$
D'< 0 \text{ on }(0,+\infty),
$
which implies
$$
v-\URad=D<D(0)=1 \text{ on }(0,+\infty).
$$
\end{proof} 

With this, we prove the first statement in Proposition \ref{PropMinimalityOfURad}:
\begin{proof}[Proof of one-sided minimality when $\Delta\ge0$]
Following the notation in Definition \ref{DefUpperFoliation}, we define
$$
\Psi_1(x):=v(|x|)
$$
with $v$  from Proposition \ref{PropIVPRadial}. With the notations from Definition \ref{DefSolution}, we have
$$
\Psi_1\in\SGa(\R^d).
$$
Properties of $v$ implies
$$
0<\Psi_1-\URad\le 1, \text{ and }\Psi_1\ge1 \text{ in }\R^d.
$$

For $t>0$, define
$$
\Psi_t(x):=t^\beta\Psi_1(x/t)
$$
with the scaling parameter $\beta$  from \eqref{EqnScalingParameter}. From this definition, the continuity of $(t,x)\mapsto\Psi_t(x)$ follows. Proposition \ref{PropSymmetryOfEquation} gives $\Psi_t\in\SGa(\R^d)$. 

Moreover, we have
$$
(\Psi_t-\URad)(x)=t^\beta[\Psi_1(x/t)-\URad(x/t)]\in(0,t^\beta)
$$
and
$$
\Psi_t\ge t^\beta \text{ in }\R^d.
$$

With these, we see that $\{\Psi_t\}$ gives an upper foliation of $\URad$ as in Definition \ref{DefUpperFoliation}. 
The conclusion follows from Proposition \ref{PropMinimalityFromAbove}.
\end{proof}

\subsection{Lower foliation of the radial cone}
\label{SubsectionLowerFoliationURad}
We turn to the minimality of the radial cone $\URad$ in  Statement (2) of Proposition \ref{PropMinimalityOfURad}. In Subsection \ref{SubsectionUpperFoliationRadial},  we have proved its one-sided minimality from above. It remains to prove the minimality from below.  With Proposition \ref{PropMinimalityFromBelow}, we need to construct a lower foliation of $\URad$ as in Definition \ref{DefLowerFoliation}.

The leaves in this lower foliation arise as  rescalings of the single leaf  given in \eqref{EqnLowerLeafRadialV2}. 
To construct this leaf, we  perturb $\URad$ by a supersolution to the linearized operator $\LURad$ in \eqref{EqnLinearizedOperatorAroundURad}. One particular supersolution is the function $v$  in \eqref{EquationVLinearizedEquationSuper}. Due to the singularity in $\LURad$ near the zero set of $\URad$, this perturbation $U:=\URad-v$ (see \eqref{EqnTheProfileAwayFromTheFreeBoundaryV2}) only gives meaningful information when $\URad$ and $v$ are comparable, that is, in the region $\OEta$ as in  \eqref{EqnOEtaRadial}. In this region, the desired subsolution property of $U$ follows from Lemma \ref{LemOuterRegionRadial}. 

In the complement of $\OEta$, we build the leaf by the explicit profile $V$ in \eqref{EqnTheProfileNearTheFreeBoundaryV2}. Its desired subsolution property is given in Lemma \ref{LemInnerRegionRadial}.

Gluing together $U$ and $V$ (see \eqref{EqnTheProfileAwayFromTheFreeBoundaryV2} and \eqref{EqnTheProfileNearTheFreeBoundaryV2}), we get a single leaf in the lower foliation in \eqref{EqnLowerLeafRadialV2}. To get the subsolution property along the interface of gluing, we need the derivatives of $U$ and $V$ to be ordered as in \eqref{EqnOrderingOfDerivativesOfUAndV}. This puts the constraint on $\gamma$ in Statement (2) of Proposition \ref{PropMinimalityOfURad}.

 \vem

To simplify our notations,  for a positive function $\varphi$, we denote the error in solving the Alt-Phillips problem (Definition \ref{DefSolution}) as
\begin{equation}
\label{EqnErrorInEquationRadial}
E(\varphi):=\Delta\varphi-\gamma\varphi^{\gamma-1}.
\end{equation}

\vem
For a small dimensional constant $\eta>0$ to be chosen, the \textit{region away from the free boundary} is defined as 
\begin{equation}
\label{EqnOEtaRadial}
\Omega_{\eta}:=\{\URad-v>\eta\URad\},
\end{equation}
where  $\URad$ denotes the radial cone in \eqref{EqnURad} and  
\begin{equation}
\label{EquationVLinearizedEquationSuper}
v(x):=|x|^{-\frac{d-2}{2}}.
\end{equation}
The \textit{profile away from the free boundary} is taken as 
\begin{equation}
\label{EqnTheProfileAwayFromTheFreeBoundaryV2}
U:=\URad-v.
\end{equation}
  
This function satisfies
\begin{lem}
\label{LemOuterRegionRadial}
For $\gamma\in(0,1)$ and $\eta>0$, we have
\begin{equation}
\label{EqnSomethingElementary}
E(U)\ge|x|^{-\frac{d}{2}-1}\cdot\{\frac{(d-2)^2}{4}-2d\eta^{\gamma-2}(1-\gamma)\}  \text{ in }\OEta.
\end{equation}
\end{lem} 

\begin{proof}
With $\URad\in\SGa(\R^d)$, the convexity of the function $t\mapsto t^{\gamma-1}$ and the definition of $\OEta$ in \eqref{EqnOEtaRadial}, we have
\begin{align*}
E(U)&=-\Delta v-\gamma[U^{\gamma-1}-\URad^{\gamma-1}]\\
&\ge -\Delta v+\gamma(\gamma-1)\eta^{\gamma-2}\URad^{\gamma-2}v\hem \text{ in }\Omega_\eta.
\end{align*} 
Using the particular forms of $\URad$, $\LURad$ and $v$, we conclude
\begin{align*}
E(U)&\ge |x|^{-\frac{d-2}{2}-2}\cdot\{\frac{(d-2)^2}{4}-(1-\gamma)\beta(d+\beta-2)\eta^{\gamma-2})\}\\
&\ge |x|^{-\frac{d-2}{2}-2}\cdot\{\frac{(d-2)^2}{4}-2d\eta^{\gamma-2}(1-\gamma)\} \hem \text{ in }\Omega_\eta.
\end{align*}
since the scaling parameter $\beta$ from \eqref{EqnScalingParameter} satisfies
$
\beta<2 \text{ for $\gamma\in(0,1)$.}
$
\end{proof} 

In particular, once $\eta>0$ is chosen, the profile $U$ gives  a subsolution in $\OEta$ when $\gamma$ is close to $1$.

\vem

We now turn to the \textit{region near the origin}. 
By definition, we have
\begin{equation}
\label{EqnRegionNearTheOriginV2}
\R^d\backslash\OEta=\overline{B_{R_\eta}} \text{ for some }0<R_\eta<+\infty,
\end{equation}
where $\OEta$ is given in \eqref{EqnOEtaRadial}.

In $B_{R_\eta}$, we build the profile of the form
\begin{equation}
\label{EqnTheProfileNearTheFreeBoundaryV2}
V(x):=A[(|x|-r_\eta)_+]^\beta,
\end{equation}
where the coefficient $A$ satisfies $A^{2-\gamma}=2/\beta^2$, and $r_\eta\in(0,R_\eta)$ is chosen such that 
\begin{equation}
\label{EqnMatchingOfValueRadial}
V=\eta\URad=U \text{ on }\partial B_{R_\eta}.
\end{equation}
Or equivalently, we have
\begin{equation}
\label{EqnMatchingOfValueRadialII}
A(R_\eta-r_\eta)^\beta=\eta c_{\mathrm{rad}}R_\eta^\beta,
\end{equation} 
where $c_{\mathrm{rad}}$ is the coefficient in \eqref{EqnURad}.

This function $V$ gives a subsolution as in Definition \ref{DefSolution}:
\begin{lem}
\label{LemInnerRegionRadial}
For $\gamma\in(0,1)$ and $0<\eta\le\frac14$, we have
$$
V\in\SPGa(B_{R_\eta}),
\text{ and }
V\le\URad \text{ in }B_{R_\eta}.
$$
\end{lem} 
\begin{proof}
For our choice of $A$, we have, in $\{V>0\}$, 
$$
\Delta V=A\beta(\beta-1)[|x|-r_\eta]^{\beta-2}+\frac{d-1}{|x|}A\beta[|x|-r_\eta]^{\beta-1}\ge A\beta(\beta-1)[|x|-r_\eta]^{\beta-2}=\gamma V^{\gamma-1}.
$$
A direct computation gives $|\nabla V|^2/V^\gamma=2$ along $\partial B_{r_\eta}=\partial\{V>0\}.$ 

As a result, we have $V\in\SPGa(B_{R_\eta}).$
\vem

It remains to show the ordering between $V$ and $\URad$. 

By its monotonicity in the radial variable, we have
\begin{equation*}
V\le V(R_\eta)=\eta\cdot\URad(R_\eta) \text{ in }B_{R_\eta}.
\end{equation*}
Meanwhile, the homogeneity of $\URad$ gives
$$
\URad(r)\ge\URad(R_\eta)\cdot(\frac{r_\eta}{R_\eta})^\beta \text{ for }r_\eta\le r\le R_\eta.
$$
Condition \eqref{EqnMatchingOfValueRadialII} gives
$$
\frac{r_\eta}{R_\eta}=1-(\frac{\CRad\eta}{A})^{\frac{1}{\beta}}\ge 1-\eta^\frac{1}{\beta}
$$
since $\CRad<A$.

Combining all these, we have
\begin{equation}
\label{EqnByItsMonotonicityInTheRadialVariable}
\URad-V\ge\URad(R_\eta)\cdot[(1-\eta^\frac{1}{\beta})^\beta-\eta].
\end{equation}
With $\gamma\in(0,1)$, we have $\beta<2$ (see \eqref{EqnScalingParameter}). Under our assumption on $\eta$, we have
$
\eta^\frac{1}{\beta}<\eta^\frac{1}{2}\le\frac12.
$
As the result, the right-hand side of \eqref{EqnByItsMonotonicityInTheRadialVariable} is positive, and the desired ordering follows. 
\end{proof} 
\vem

We define an \textit{lower leaf} as
\begin{equation}
\label{EqnLowerLeafRadialV2}
\Phi_1(x):=\begin{cases}
V(x) &\text{ for }x\in\overline{B_{R_\eta}},\\
U(x) &\text{ for }x\in\OEta,
\end{cases}
\end{equation} 
where $\OEta$ and $R_\eta$ are given in \eqref{EqnOEtaRadial} and in  \eqref{EqnRegionNearTheOriginV2} respectively. The functions $U$ and $V$ are given in \eqref{EqnTheProfileAwayFromTheFreeBoundaryV2} and \eqref{EqnTheProfileNearTheFreeBoundaryV2} respectively. 

This function gives a subsolution to the Alt-Phillips problem when $\gamma$ is close to $1$:
\begin{lem}
\label{LemGluedSubsolutionRadial}
For $d\ge3$ and $\gamma\in(1-\frac{(d-2)^2}{64d^2},1)$, if we pick $\eta=(2d)^{-\frac32}$ in the construction of $\Phi_1$, then we have
$$
\Phi_1\in\SPGa(\R^d)
$$ 
and
$$
0\le\Phi_1\le\URad \text{ in }\R^d.
$$
\end{lem}

\begin{proof}
For our choice of $\eta$ and $\gamma$,  Lemma \ref{LemOuterRegionRadial} implies
$
\Phi_1\in\SPGa(\OEta).
$
By Lemma \ref{LemInnerRegionRadial}, the same holds in $B_{R_\eta}$. Below we verify the subsolution property along $\partial\OEta=\partial B_{R_\eta}$.

With the definition of $\OEta$ and the homogeneities of the functions, we have
$$
\PR V(R_\eta)=\frac{\beta}{R_\eta-r_\eta}V(R_\eta)=\frac{\beta}{R_\eta-r_\eta}\eta\cdot\URad(R_\eta)
$$
and 
$$
\PR U(R_\eta)=\frac{\beta}{R_\eta}\URad(R_\eta)+\frac{d-2}{2R_\eta}v(R_\eta)=[\frac{\beta}{R_\eta}+\frac{d-2}{2R_\eta}(1-\eta)]\cdot\URad(R_\eta).
$$
Together with \eqref{EqnMatchingOfValueRadialII}, we have
$$
\frac{\PR V}{\PR U}(R_\eta)=\frac{\beta}{\beta+(1-\eta)(d-2)/2}\cdot\eta\frac{R_\eta}{R_\eta-r_\eta}\le[\frac{A}{c_{\mathrm{rad}}}]^\frac{1}{\beta}\eta^{1-\frac{1}{\beta}}=[\frac{d+\beta-2}{\beta-1}]^\frac{1}{2}\eta^{1-\frac{1}{\beta}}.
$$
With our assumption on $\gamma$, the scaling parameter $\beta$ lies in $(1,\frac32)$. This allows us to continue as
\begin{equation}
\label{EqnOrderingOfDerivativesOfUAndV}
\frac{\PR V}{\PR U}(R_\eta)\le (2d)^{\frac12}\eta^{\frac13}=1
\end{equation}
by our choice of $\eta$.

As a result, along $\partial B_{R_\eta}$, the function $\Phi_1$ makes a convex angle. This implies the desired subsolution property. 

In summary, we have established
$$
\Phi_1\in\SPGa(\R^d).
$$

\vem

The ordering between $\Phi_1$ and $\URad$ holds in $\OEta$ by definition of $U$ in \eqref{EqnTheProfileAwayFromTheFreeBoundaryV2}. The ordering in the remaining region, $B_{R_\eta}$, follows from the ordering between $V$ and $\URad$ in Lemma \ref{LemInnerRegionRadial}.
\end{proof}

With this, we give the proof of the second statement in Proposition \ref{PropMinimalityOfURad}:
\begin{proof}[Proof of minimality of $\URad$ when $\gamma>1-\frac{(d-2)^2}{64d^2}$]
With Proposition \ref{PropSolutionsAreMinimizers}, it suffices to consider the range $\gamma\in(1-\frac{(d-2)^2}{64d^2},1)$. 

For exponents within this range, it is elementary to verify that the discriminant in \eqref{EqnDiscriminant} satisfies 
$
\Delta(d,\gamma)\ge0.
$
In particular, we have already established the one-sided minimality from above in Subsection \ref{SubsectionUpperFoliationRadial} (See Statement (1) in Proposition \ref{PropMinimalityOfURad}). It remains to show the one-sided minimality from below, which amounts to the construction of a lower foliation of $\URad$ (See Proposition \ref{PropMinimalityFromBelow} and Theorem \ref{ThmMinimality}).

\vem

With the profile $\Phi_1$  in Lemma \ref{LemGluedSubsolutionRadial}, we define, for $t>0$, 
$$
\Phi_t(x):=t^\beta\Phi_1(x/t) \text{ in }\R^d.
$$
We show that $\{\Phi_t\}$ give is a lower foliation of $\URad$ as in Definition \ref{DefLowerFoliation}.

It follows from Lemma \ref{LemGluedSubsolutionRadial} and Proposition \ref{PropSymmetryOfEquation} that $\Phi_t\in\SPGa(\R^d)$. The ordering between $\Phi_1$ and $\URad$ implies
$$
\Phi_t(x)-\URad(x)=t^\beta[\Phi_1(x/t)-\URad(x/t)]\le 0 \text{ in }\R^d.
$$

The continuity of $(t,x)\mapsto\Phi_t(x)$ follows from definition. The positive sets of these functions satisfy 
\begin{equation*}
\overline{\{\Phi_t>0\}}=t\cdot\overline{\{\Phi_1>0\}}=\R^d\backslash B_{tr_\eta}.
\end{equation*}
As a result,  the map $t\mapsto\overline{\{\Phi_t>0\}}$ is locally uniformly continuous.  Moreover, given any compact set $K\subset\R^d$, we can find $t_K$ such that $K\subset B_{t_Kr_\eta}$. For $t>t_K$, we have $\Phi_t=0$ on $K$.

For each $0<r<R<+\infty$, by the definition of $\Phi_1$, we have $\Phi_1=U$ in $t^{-1}\cdot(B_R\backslash B_r)$ for $t$ small. This implies that 
$$
\Phi_t(x)=\URad(x)-t^\beta v(x/t)=\URad(x)-t^{\frac{d-2}{2}+\beta}|x|^{-\frac{d-2}{2}}.
$$
As a result, for $t\to0$, we have 
$$
\Phi_t\to\URad \text{ locally uniformly in }\{\URad>0\}.
$$
Therefore, the family $\{\Phi_t\}_{t\in(0,+\infty)}$ is a lower foliation of $\URad.$ 
\end{proof}


\section{Construction of the axially symmetric cone}
\label{SectionCAS}
Starting from this section, we turn our attention to cones with axial symmetry in the $\gamma$-Alt-Phillips problem \eqref{EqnAP}. According to Savin-Yu \cite{SY}, they represent the most typical behavior of minimizing cones as $\gamma$ tends to $1$. 

The radial cone $\URad$ from \eqref{EqnURad}, when trivially extended to higher dimensions, gives examples of axially symmetric cones.  The contact sets of these cones have zero Lebesgue measure. When $\gamma$ is less than $1$ but close to $1$,  we will  construct singular minimizing cones whose contact sets have positive measure, as in Theorem \ref{ThmMainAS}. To our knowledge, this is the first free boundary problem that allows both these two behaviors. 

In this section, we construct these cones as solutions to the Euler-Lagrange equation in Definition \ref{DefSolution}.  In Section \ref{SectionLowerFoliationAS} and Section \ref{SectionUFAS}, we show that these cones are minimizers as in Definition \ref{DefMinimality}.

\vem

Starting from this section, we work \textit{exclusively in $\R^{d+1}$ with $d\ge3$}. We decompose the space as
\begin{equation}
\label{DefSpaceDecomposition}
\R^{d+1}=\{(x,y):\hem x\in\R^d,\hem y\in\R\}.
\end{equation}
From this section, we always assume 
$$
\gamma\in(0,1) \text{ is close to }1.
$$
The \textbf{main result} of this section is
\begin{prop}
\label{PropConstructionOfUAsFirstMention}
For $d\ge3$, there is a  dimensional constant $\gamma_d^3\in(0,1)$ such that for 
$
\gamma_d^3<\gamma<1,
$
there is an axially symmetric,  $\beta$-homogeneous function 
$$
\UAS\in\SGa(\R^{d+1})
$$  
with
$$
|\{\UAS=0\}\cap B_1|\sim_d(1-\gamma)^{\frac{d}{d-2}}.
$$
\end{prop} 
Recall the space of solutions $\SGa$ from Definition \ref{DefSolution}. Recall the similarity notation $\sim_d$ from Remark \ref{RemSimAndO}.

\vem

We briefly explain the \textbf{strategy of the construction}. 

In Subsection \ref{SubsectionEquationForAxiallySymmetricFunctions}, we perform a transformation of our solution (see \eqref{EqnTransformedSolution}). In the remaining part of this work, we work with the transformed solution. In particular, instead of Proposition \ref{PropConstructionOfUAsFirstMention}, we show the equivalent version for the transformed solution in Proposition \ref{PropConstructionOfUAs}. With axial symmetry, this amounts to constructing a solution to an ordinary differential equation \eqref{EqnNonlinearODE}.

The transformation in \eqref{EqnTransformedSolution} is designed so that the quadratic polynomial $p$ in \eqref{EqnTheP} is a solution to our problem. In Subsection \ref{SubsectionEqnNearEquator}, we construct a family of solution  as perturbations of $p$ (see \eqref{EqnNearEquator}).  We get precise expansions of this family of solutions by studying $L_p$ the linearized operator around $p$ (see \eqref{EqnLP}).

Unfortunately, the operator $L_p$ becomes singular near the south pole. Near this point, we construct a family of solutions in Subsection \ref{SubsectionEqnNearSouthPole} (see \eqref{EqnEqnNearSouthPole}), modeled on the radial solution to the obstacle problem in \eqref{EqnPLambda}.

Finally in Subsection \ref{SubsectionGluingOfTheSolutionsUAS}, we show that we can glue the solutions constructed in the previous two subsections to form a solution on the entire sphere. Here we rely on the precise expansions of the solutions in Lemma \ref{CorFinerExpansion} and Lemma \ref{LemExpansionNearTheSouthPole}. 

Technical tools behind these expansions are postponed to Appendix \ref{AppendixODE}.

\subsection{Transformed equation for axially symmetric functions}
\label{SubsectionEquationForAxiallySymmetricFunctions}

To take advantage of the smallness of $(1-\gamma)$, we apply the  transformation
\begin{equation}
\label{EqnTransformedSolution}
v:=\frac{\beta(d+\beta-2)}{2d\gamma}u^{\frac{2}{\beta}},
\end{equation} 
where $\beta$ is the scaling parameter from \eqref{EqnScalingParameter}.

When $u$ is $\beta$-homogeneous, the transformed function $v$ is $2$-homogeneous.

It is elementary to verify
\begin{prop}
\label{PropEquivalenceOfEquations}
For $\gamma\in(0,1)$,  let $u$ and $v$ be non-negative functions related by \eqref{EqnTransformedSolution}. 

Then $u\in\SGa(\Omega)$  if and only if $v$ satisfies
\begin{equation}
\label{EqnTransformedEquation}
\Delta v+\frac{\beta-2}{2}\frac{|\nabla v|^2}{v}=(1+\frac{\beta-2}{d})\chi_{\{v>0\}} \hem\text{ in }\Omega.
\end{equation}

Similar properties hold for subsolutions and supersolutions.
\end{prop} 

\begin{rem}
\label{RemAbusingNotations}
Recall the space of solutions $\SGa$, the space of subsolutions $\SPGa$ and the space of supersolutions $\SMGa$ from Definition \ref{DefSolution}. Below we use the same notations for the transformed equation \eqref{EqnTransformedEquation}. 
\end{rem} 

 The specific form of the transformation \eqref{EqnTransformedSolution} is chosen so that the following quadratic function is always a solution, that is, 
 \begin{equation}
\label{EqnTheP}
p:=\frac{1}{2d}|x|^2=\frac{1}{2d}r^2\sin^2(\theta)\in\SGa(\R^{d+1})\hem \text{ for all }\gamma\in(0,1).
\end{equation}
We note that $p$ satisfies
\begin{equation}
\label{EqnHForP}
|\nabla p|^2/p\equiv2/d.
\end{equation} 

Working with the transformed variables, Proposition \ref{PropConstructionOfUAsFirstMention} is equivalent to the following, which is the focus for the remainder of this section. 
\begin{prop}
\label{PropConstructionOfUAs}
For $d\ge3$, there is a  dimensional constant $\gamma_d^3\in(0,1)$ such that for 
$
\gamma_d^3<\gamma<1,
$
there is an axially symmetric,  $2$-homogeneous function 
$$
\UAS\in\SGa(\R^{d+1})
$$  with
$$
|\{\UAS=0\}\cap B_1|\sim_d(1-\gamma)^{\frac{d}{d-2}}.
$$
\end{prop} 
See Remark \ref{RemAbusingNotations} for the class $\SGa$.

\vem

Recall our decomposition of $\R^{d+1}$ in \eqref{DefSpaceDecomposition}.
For functions with axial symmetry around the $y$-axis, we use the coordinate system 
$
\{(r,\theta): \hem r\ge 0,\hem\theta\in[0,\pi]\}
$ with
\begin{equation}
\label{EqnCoordinate}
|x|=r\sin(\theta),\hem y=-r\cos(\theta).
\end{equation}
In this system, the south pole on the unit sphere has coordinate $(1,0)$, and the equator corresponds to $(1,\pi/2).$

\vem

For a $2$-homogeneous axially symmetric function  in $\R^{d+1}$, namely, a function of the form
$$
u(r,\theta)=r^2\bar{u}(\theta),
$$
the transformed equation \eqref{EqnTransformedEquation} reduces to an ODE:
$$
\bar{u}''+(d-1)\cot(\theta) \BU'+2(d+\beta-1)\BU+\frac{\beta-2}{2}\frac{(\BU')^2}{\BU}=(1+\frac{\beta-2}{d})\chi_{\{\BU>0\}} \text{ on }(0,\pi).
$$

For brevity, we introduce a parameter
\begin{equation}
\label{EqnKappa}
\kappa:=4(1-\beta/2)=\frac{4(1-\gamma)}{2-\gamma}\in(0,2).
\end{equation} 
 In particular, when $\gamma$ is close to $1$, $\kappa$ is close to $0$.

With this parameter, the previous ODE takes the form
\begin{equation}
\label{EqnNonlinearODE}
\BU''+(d-1)\cot(\theta)\BU'+2(d+1)\BU=\frac{\kappa}{4}\frac{(\BU')^2}{\BU}+\kappa\BU+(1-\frac{\kappa}{2d})\chi_{\{\BU>0\}} \text{ on }(0,\pi).
\end{equation} 
\subsection{Equation near the equator}
\label{SubsectionEqnNearEquator}

We begin our construction of a solution $u$ to \eqref{EqnNonlinearODE}. In this section, we focus on the region near the equator on the unit sphere $\partial B_1$, that is, $\theta=\frac{\pi}{2}$ in our system \eqref{EqnCoordinate}. 

To be precise, with a small parameter $\mu>0$ to be chosen, the small parameter $\kappa$ in \eqref{EqnKappa}, and the profile $p$ from \eqref{EqnTheP}, we study solutions to the following
\begin{equation}
\label{EqnNearEquator}
\begin{cases}
u''+(d-1)\cot(\theta)u'+2(d+1)u=\frac{\kappa}{4}\frac{(u')^2}{u}+\kappa u+(1-\frac{\kappa}{2d})\chi_{\{u>0\}} &\text{ on }(0,\pi/2),\\
u'=0, \text{ and }u=p+\mu &\text{ at }\pi/2.
\end{cases}
\end{equation}

The normalized solution 
\begin{equation}
\label{EqnNormalizedSolution}
v:=\frac{u-p}{\mu}
\end{equation} 
satisfies the terminal condition 
\begin{equation}
\label{EqnTerminalCondition}
v'(\pi/2)=0, \text{ and }v(\pi/2)=1.
\end{equation} 
Recall that $p$  is a solution to \eqref{EqnNonlinearODE},
the ODE for $v$ is
\begin{equation}
\label{EqnEqnForV}
\LP v=-\frac{\kappa}{\sin^2(\theta)}v+f \text{ on }\PosS\cap(0,\pi/2),
\end{equation} 
where the linear operator $\LP$ is defined as
\begin{equation}
\label{EqnLP}
\LP v:=v''+(d-1-\kappa)\cot(\theta)v'+2(d+1-\kappa)v,
\end{equation} 
and the nonlinearity $f$ is given by
\begin{equation}
\label{EqnTheF}
f=\frac{\kappa}{4\mu}[\frac{(u')^2}{u}-\frac{(p')^2}{p}-\frac{2p'}{p}(u-p)'+(\frac{p'}{p})^2(u-p)].
\end{equation}

We note that the following function 
$$
v_0=\frac{1}{2d}\sin^2(\theta)-\frac{1-\kappa/d}{2(d+1-\kappa)}
$$
solves the homogeneous  equation 
$
\LP v_0=0 \text{ on }(0,\pi/2).
$
If we normalize
\begin{equation}
\label{EqnV0}
V_0:=\frac{v_0}{v_0(\pi/2)},
\end{equation} 
then it is elementary to verify the following properties:
\begin{lem}
\label{LemPropertiesOfV0}
The function $V_0$ in \eqref{EqnV0} satisfies
\begin{equation*}
\label{EqnV0SolvesEquation}
\LP V_0=0 \text{ on }(0,\pi/2),\hem
V_0'(\pi/2)=0, \text{ and } V_0(\pi/2)=1.
\end{equation*} 

When $d\ge 3$ and $0<\kappa<\frac14$, we have
$$
0\le V_0'\le 6d,\hem  |V_0|\le 2d \text{ on }[0,\pi/2]
$$
and
$$
V_0\le-1/4 \text{ on }[0,\pi/4].
$$
Moreover, we have
$$
|V_0-[(d+1)\sin^2(\theta)-d]|\le \BO(\kappa) \text{ and }|V_0'|\le \BO(\theta) \text{ on }[0,\pi/2].
$$
\end{lem} 
Recall the big-O notation from Remark \ref{RemSimAndO}. 

We give a first estimate on  $u$ in \eqref{EqnNearEquator} and $v$ in \eqref{EqnNormalizedSolution}:
\begin{lem}
\label{LemAnalysisOnV}
For $d\ge3$, there is a large dimensional constant $A_d$  such that   
\begin{equation}
\label{EqnLemAnalysisOnV}
u>0,\hem  |v|< 4d,\text{ and }\hem |v'|< 12d\cdot\theta^{-1}\hem \text{ on }[A_d(\mu^{\frac12}+\kappa^{\frac{1}{d-2}}),\pi/2]
\end{equation}
for dimensionally small $\mu$ and $\kappa$.
\end{lem} 
Recall the small parameters $\mu$ from \eqref{EqnNearEquator} and $\kappa$ from \eqref{EqnKappa}.

\begin{proof}
\textit{Step 1: Preparations.}

 The bounds  hold in a neighborhood of $\pi/2$.  Suppose that $\theta_0$ is the first time that \eqref{EqnLemAnalysisOnV} fails,\footnote{If the bounds never fail, we take $\theta_0=0$.} that is, 
\begin{equation}
\label{EqnExitingTime}
u>0, \hem |v|< 4d,\text{ and } |v'|< 12d\cdot\theta^{-1} \text{ on }(\theta_0,\pi/2]
\end{equation} 
and one of these  fails at $\theta_0$. 

We need to bound $\theta_0$ by $A_d(\mu^{\frac{1}{2}}+\kappa^{\frac{1}{d-2}})$ for some large $A_d$ to be chosen. Suppose that this fails, then 
\begin{equation}
\label{EqnExitingTimeTooShort}
\theta_0>A_d(\mu^{\frac{1}{2}}+\kappa^{\frac{1}{d-2}}).
\end{equation}

Define the difference between $v$ and the normalized solution $V_0$ in \eqref{EqnV0}
\begin{equation}
\label{EqnDifferenceVV0}
w:=v-V_0.
\end{equation} 
With \eqref{EqnTerminalCondition}, \eqref{EqnEqnForV}, Lemma \ref{LemPropertiesOfV0}, this function satisfies
\begin{equation}
\label{EqnODEInPreparations}
\LP w=\rho \text{ on }(\theta_0,\pi/2),  \text{ and }w'(\pi/2)=w(\pi/2)=0,
\end{equation} 
where the right-hand side function $\rho$ is given by
\begin{equation}
\label{EqnRhoForDifference}
\rho=-\frac{\kappa}{\sin^2(\theta)}v+f.
\end{equation}
This equation has the form of  \eqref{EqnAppendixODE} in Appendix \ref{AppendixODE1}.

\vem

\textit{Step 2: Size of the right-hand side.}

Note that the nonlinearity $f$ in \eqref{EqnTheF} is of the form 
\begin{equation}
\label{EqnTheFHasTheForm}
f=\frac{\kappa}{4\mu}[h(X+Z)-h(X)-\nabla h(X)\cdot Z],
\end{equation}
where $h:\R^2\to\R$ is defined as $h(X)=X_1^2/X_2$ with $X=(p',p)$ and $Z=(\mu v',\mu v)$.
With Lemma \ref{LemAppendixLinearExpansion} and \eqref{EqnHForP},  we have
\begin{equation}
\label{EqnTheFHasTheFormTheForm}
|h(X+Z)-h(X)-\nabla h(X)\cdot Z|\le (\frac{\mu v}{p})^2[\frac{p}{p+\mu v}+(\frac{p}{p+\mu v})^2]+\frac{(\mu v')^2}{p}[1+\frac{p}{p+\mu v}].
\end{equation}

With \eqref{EqnExitingTime} and \eqref{EqnExitingTimeTooShort}, we have, on $(\theta_0,\pi/2),$
$$
(\frac{\mu v}{p})^2+ \frac{(\mu v')^2}{p}\le\BO(\mu^2\theta^{-4}),
$$
and
\begin{equation}
\label{EqnLastInequality}
\frac{p}{p+\mu v}=\frac{1}{1+\mu v/p}\le \frac{1}{1-C_dA_d^{-2}}\le 2\hem \text{ if $A_d$ is large}.
\end{equation}

Putting these back into \eqref{EqnTheFHasTheForm} and \eqref{EqnTheFHasTheFormTheForm}, we conclude
\begin{equation}
\label{EqnExpressionOfF}
|f|\le C_dA_d^{-2}\kappa\mu\theta^{-4}\le C_dA_d^{-2}\kappa\theta^{-2} \text{ on }(\theta_0,\frac{\pi}{2}],
\end{equation}
and the function in \eqref{EqnRhoForDifference} satisfies
\begin{equation}
\label{EqnStar}
|\rho|\le C_d A_d^{-2}\kappa\theta^{-2}  \text{ on }(\theta_0,\frac{\pi}{2}].
\end{equation}

\vem

\textit{Step 3: Estimate on $\theta_0$.}

With \eqref{EqnODEInPreparations}, we can apply Lemma \ref{LemVariationOfParameters} to $w$ in \eqref{EqnDifferenceVV0} and obtain
\begin{equation}
\label{EqnTheFormOfW}
w(\theta)=\varphi(\theta) \alpha(\theta)+V_0(\theta) \beta(\theta),
\end{equation}
where 
\begin{equation}
\label{EqnAlphaAndBeta}
\alpha(\theta)=-\int_{\theta}^{\pi/2}\rho(\tau)V_0(\tau)\sin^{d-1-\kappa}(\tau)d\tau,
\text{ and }
\beta(\theta)=\int_{\theta}^{\pi/2}\rho(\tau)\varphi(\tau)\sin^{d-1-\kappa}(\tau)d\tau.
\end{equation}
 Lemma \ref{LemPropertiesOfV0},  Lemma \ref{LemPropertiesOfPhi} and \eqref{EqnStar} imply
\begin{equation}
\label{EqnBoundsOnAlphaBeta}
|\alpha|\le C_dA_d^{-2}\kappa, \text{ and }|\beta|\le C_dA_d^{-2}\kappa(1+|\log(\theta)|) \text{ on }(\theta_0,\pi/2].
\end{equation}
Putting these back into \eqref{EqnTheFormOfW}, we conclude
\begin{equation}
\label{EqnEstimateOnW}
|w|\le C_dA_d^{-2}\kappa\theta^{-d+2+\kappa}\le C_dA_d^{-2}  \text{ on }(\theta_0,\pi/2].
\end{equation}
For the last inequality, we used \eqref{EqnExitingTimeTooShort}.

\vem

With $v=w+V_0$ (see \eqref{EqnDifferenceVV0}) and the bound on $V_0$ in Lemma \ref{LemPropertiesOfV0}, we have
$$
|v|\le |V_0|+|w|\le 2d+C_dA_d^{-2}<3d \text{ on }(\theta_0,\pi/2]
$$
if  $A_d$ is large. Thus the bound on $|v|$ in \eqref{EqnExitingTime}  holds in a neighborhood of $\theta_0$.

By \eqref{EqnNormalizedSolution} and  \eqref{EqnExitingTimeTooShort}, we have 
$$
u\ge p-\mu|v|\ge c_d\theta_0^2-3d\mu\ge (c_dA_d^2-3d)\mu>0  \text{ on }(\theta_0,\pi/2]
$$
if $A_d$ is large.  
Thus the bound on $u$ in \eqref{EqnExitingTime} holds in a neighborhood of $\theta_0$.

With similar ideas for \eqref{EqnEstimateOnW}, we get 
$$
|w'|\le C_dA_d^{-2}\kappa\theta^{-d+1+\kappa} \text{ on }(\theta_0,\pi/2].
$$
Together with the bound on $V_0'$ in Lemma \ref{LemPropertiesOfV0}, this implies
$$
|v'|\le |V_0'|+|w'|\le 6d+C_dA_d^{-2}\theta^{-1}<10d\theta^{-1}  \text{ on }(\theta_0,\pi/2]
$$
if $A_d$ is large.  
Thus the bound on $|v'|$ in \eqref{EqnExitingTime} holds in a neighborhood of $\theta_0$.

\vem

In summary, if $A_d$ is large, then  \eqref{EqnExitingTime} holds in a neighborhood of $\theta_0$, contradicting our definition of $\theta_0$ in \textit{Step 1}. As a result, we conclude
$$
\theta_0\le A_d(\mu^{1/2}+\kappa^{\frac{1}{d-2}}).
$$
As a result, the bounds in \eqref{EqnLemAnalysisOnV} stay true on $[A_d(\mu^{1/2}+\kappa^{\frac{1}{d-2}}),\pi/2]$.
\end{proof} 

As a quick corollary, we have the following expansion for the coefficient $\alpha$ in \eqref{EqnAlphaAndBeta}:
\begin{cor}
\label{CorFinerExpansionOfAlpha}
Under the same assumptions in Lemma \ref{LemAnalysisOnV}, the function $\alpha$ in \eqref{EqnAlphaAndBeta} has the following expansion on  $[A_d(\mu^{\frac12}+\kappa^{\frac{1}{d-2}}),\pi/2],$
$$
\alpha(\theta)=\alpha_d\kappa+\kappa[\BO(\theta^{1-\kappa}+\mu(1+\theta^{d-4-\kappa}))]+\kappa^2 \BO(1+|\log(\theta)|).
$$
Here $A_d$ is the dimensional constant in Lemma \ref{LemAnalysisOnV}, and $\alpha_d$ is defined as
\begin{equation}
\label{EqnDefOfAlphaD}
\alpha_d:=\int_0^{\pi/2}[(d+1)\sin^2(\tau)-d]^2\sin^{d-3}(\tau)d\tau.
\end{equation}
\end{cor}
\begin{proof}
To estimate $\alpha$ in \eqref{EqnAlphaAndBeta}, we decompose the function $\rho$ from \eqref{EqnRhoForDifference} as $\rho=\rho_1+\rho_2+f$, where
$$
\rho_1:=-\frac{\kappa}{\sin^2(\theta)}V_0(\theta),
\text{ and } \rho_2=-\frac{\kappa}{\sin^2(\theta)}w
$$
by \eqref{EqnDifferenceVV0}.
This gives the corresponding decomposition of $\alpha$ as 
 $$\alpha=\alpha_1+\alpha_2+\alpha_f,$$
 each piece satisfying the relation as in \eqref{EqnAlphaAndBeta}.

With Lemma \ref{LemPropertiesOfV0} and  \eqref{EqnEstimateOnW}, we have
$$
|\alpha_2|=\kappa^2\BO(1+|\log(\theta)|) \text{ on }[A_d(\mu^{\frac12}+\kappa^{\frac{1}{d-2}}),\pi/2].
$$ 
With the bound on $f$ from \eqref{EqnExpressionOfF}, we have
$$
|\alpha_f|=\kappa\mu \BO(1+\theta^{d-4-\kappa}) \text{ on }[A_d(\mu^{\frac12}+\kappa^{\frac{1}{d-2}}),\pi/2].
$$

For the term $\alpha_1$, we have, on the same interval, 
\begin{align*}
\alpha_1(\theta)&=\kappa\{\int_0^{\pi/2}V_0^2(\tau)\sin^{d-3-\kappa}(\tau)d\tau-\int^{\theta}_0V_0^2(\tau)\sin^{d-3-\kappa}(\tau)d\tau\}\\
&=\kappa\int_0^{\pi/2}V_0^2(\tau)\sin^{d-3-\kappa}(\tau)d\tau+\kappa \BO(\theta^{d-2-\kappa}).
\end{align*}
With $|\sin^{-\kappa}(\tau)-1|=\kappa \BO(\tau^{-\kappa}|\log(\tau)|)$  and Lemma \ref{LemPropertiesOfV0}, we continue as
\begin{align*}
\alpha_1(\theta)=\alpha_d\kappa+\kappa^2 \BO(1)+\kappa \BO(\theta^{d-2-\kappa}),
\end{align*}
where 
$\alpha_d$ is from \eqref{EqnDefOfAlphaD}.

Summarize,  the coefficient $\alpha$ satisfies,   on  $[A_d(\mu^{\frac12}+\kappa^{\frac{1}{d-2}}),\pi/2],$
$$
\alpha(\theta)=\alpha_d\kappa+\kappa[\BO(\theta^{1-\kappa}+\mu(1+\theta^{d-4-\kappa}))]+\kappa^2 \BO(1+|\log(\theta)|).
$$
\end{proof}

We get finer expansions of the solution to \eqref{EqnNearEquator}:
\begin{lem}
\label{CorFinerExpansion}
For $d\ge3$ and $\mu,\kappa$ dimensionally small, the solution $u$ to \eqref{EqnNearEquator} satisfies, on $[A_d(\mu^{\frac12}+\kappa^{\frac{1}{d-2}}),\pi/2]$,
\begin{equation*}
u(\theta)=\frac{1}{2d}\sin^2(\theta)-d\mu+\mu \BO(\kappa\theta^{-d+2}+\theta^2)
\end{equation*} 
and 
\begin{equation*}
u'(\theta)=\frac{1}{d}\sin(\theta)\cos(\theta)-\frac{\alpha_d}{d}\kappa\mu\sin^{-d+1}(\theta)+\mu \BO(\theta)+\kappa\mu \BO(\theta^{-3/2}+\mu\theta^{-d+1}+\mu\theta^{-7/2}).
\end{equation*}

Here $A_d$ is the constant from Lemma \ref{LemAnalysisOnV}, and $\alpha_d$ is defined in Corollary \ref{CorFinerExpansionOfAlpha}.
\end{lem}
Recall the small parameters $\mu$ from \eqref{EqnNearEquator} and $\kappa$ from \eqref{EqnKappa}.

\begin{proof}
In terms of the normalized solution $v$ in   \eqref{EqnNormalizedSolution}, the desired expansion is equivalent to the following on $[A_d(\mu^{\frac12}+\kappa^{\frac{1}{d-1}}),\pi/2]$
\begin{equation}
\label{EqnExpansionForV}
v(\theta)=-d+\BO(\kappa\theta^{-d+2}+\theta^2)
\end{equation} 
and
\begin{equation}
\label{EqnExpansionForV'}
v'(\theta)=-\frac{\alpha_d}{d}\kappa\sin^{-d+1}(\theta)+\BO(\theta)+\kappa \BO(\theta^{-3/2}+\mu\theta^{-d+1}+\mu\theta^{-7/2}).
\end{equation} 

\vem

\textit{Step 1: The expansion of $v$.}

With \eqref{EqnLemAnalysisOnV}, the difference $w=v-V_0$ in \eqref{EqnDifferenceVV0}  satisfies \eqref{EqnEstimateOnW}. As a result, we have
\begin{align*}
|v(\theta)-V_0(0)|&\le|v-V_0|(\theta)+|V_0(\theta)-V_0(0)|\\
&\le C_d\kappa\theta^{-d+2}+C_d\theta^2 \hem\text{ on }[A_d(\mu^{1/2}+\kappa^{\frac{1}{d-1}}),\pi/2].
\end{align*}
The expansion of $V_0(0)$ in  Lemma \ref{LemPropertiesOfV0} gives \eqref{EqnExpansionForV}.

\vem

\textit{Step 2: Expansion of $v'$.}

Recall that $w$ in \eqref{EqnDifferenceVV0} solves \eqref{EqnODEInPreparations}. Lemma \ref{LemVariationOfParameters} gives
\begin{equation}
\label{EqnAgainVariationOfParameters}
w'(\theta)=\varphi'(\theta)\alpha(\theta)+V_0'(\theta)\beta(\theta),
\end{equation}
where the coefficients $\alpha$ and $\beta$ are given in \eqref{EqnAlphaAndBeta}.

With the expansion of $\alpha$ from Corollary \ref{CorFinerExpansionOfAlpha},  Lemma \ref{LemPropertiesOfPhi} gives,  for $\kappa>0$ small and $\theta\ge A_d(\mu^{\frac12}+\kappa^{\frac{1}{d-2}})$, 
\begin{align*}
\varphi'(\theta)\alpha(\theta)=-\frac{\alpha_d}{d}\kappa\sin^{-d+1}(\theta)+\kappa \BO(\theta^{-3/2}+\mu\theta^{-d+1}+\mu\theta^{-7/2}).
\end{align*}
Meanwhile,  Lemma \ref{LemPropertiesOfV0} and  \eqref{EqnBoundsOnAlphaBeta} imply
$$
V_0'\beta=\kappa \BO(\theta+\theta|\log(\theta)|).
$$

Combining these with \eqref{EqnAgainVariationOfParameters}, we get
$$
w'=-\frac{\alpha_d}{d}\kappa\sin^{-d+1}(\theta)+\kappa \BO(\theta^{-3/2}+\mu\theta^{-d+1}+\mu\theta^{-7/2}).
$$
With $v=w+V_0$ and Lemma \ref{LemPropertiesOfV0}, we have
$$
v'(\theta)=-\frac{\alpha_d}{d}\kappa\sin^{-d+1}(\theta)+\BO(\theta)+\kappa \BO(\theta^{-3/2}+\mu\theta^{-d+1}+\mu\theta^{-7/2}).
$$
This is the desired expansion on $v'$ as in \eqref{EqnExpansionForV'}.
\end{proof}

\subsection{Equation near the south pole}
\label{SubsectionEqnNearSouthPole}
We turn to the region near the south pole, that is, $\theta=0$ in our system \eqref{EqnCoordinate}. Here our family of solutions to \eqref{EqnNonlinearODE} is parametrized by the size of the contact set, $\lambda>0$, to be chosen in the next subsection. 

To be precise, we study solutions to 
\begin{equation}
\label{EqnEqnNearSouthPole}
\begin{cases}
u''+(d-1)\cot(\theta)u'+2(d+1)u=\frac{\kappa}{4}\frac{(u')^2}{u}+\kappa u+(1-\frac{\kappa}{2d}) &\text{ on }(\lambda,\pi/2),\\
u'=0, \text{ and }u=0 &\text{ at }\lambda.
\end{cases}
\end{equation} 
For the derivation of the ODE, see \eqref{EqnNonlinearODE}. Recall the small parameter $\kappa$ from \eqref{EqnKappa}.

The analysis of this problem is summarized into the following
\begin{lem}
\label{LemEqnNearSouthPole}
For $d\ge3$ and $\kappa,\lambda$ dimensionally small, there is a solution to \eqref{EqnEqnNearSouthPole} on $[\lambda,B_d]$ satisfying
$$
u(\theta)\sim_{d}(\theta-\lambda)^2, \hem u'(\theta)\sim_d(\theta-\lambda), \text{ and }|u''|=\BO(1).
$$
Here $B_d>0$ is a dimensional constant. 
\end{lem} 
Recall the notations $O$ and $\sim_d$ in Remark \ref{RemSimAndO}.  The small parameter $\lambda$ is from \eqref{EqnEqnNearSouthPole}. The small parameter $\kappa$ is from \eqref{EqnKappa}.

\begin{proof}
\textit{Step 1: Preparations.}

We study the perturbed system for $\eps>0$
\begin{equation}
\label{EqnEqnNearSouthPolePerturbed}
\begin{cases}
u_{\eps}''+(d-1)\cot(\theta)u_{\eps}'+2(d+1)u_{\eps}=\frac{\kappa}{4}\frac{(u_{\eps}')^2}{u_{\eps}+\eps}+\kappa u_{\eps}+(1-\frac{\kappa}{2d}) &\text{ on }(\lambda,\pi/2),\\
u_{\eps}'=0, \text{ and }u_{\eps}=0 &\text{ at }\lambda.
\end{cases}
\end{equation}

The ODE is equivalent to 
\begin{equation}
\label{EqnEquivalentPerturbedSystem}
\sin^{1-d}(\theta)[\sin^{d-1}(\theta)u_{\eps}']'=1-2(d+1)u_{\eps}+\frac{\kappa}{4}\frac{(u_{\eps}')^2}{u_{\eps}+\eps}+\kappa u_{\eps}-\frac{\kappa}{2d}.
\end{equation} 
Initially, the right-hand side is positive. Hence $u_{\eps}'$ and $u_{\eps}$ turn positive. 
For small $\kappa$, we have
\begin{equation}
\label{EqnAConsequenceOfEquivalentFOrm}
\sin^{1-d}(\theta)[\sin^{d-1}(\theta)u_{\eps}']'\ge 1-2(d+1)u_{\eps}-\frac{\kappa}{2d}>\frac34-2(d+1)u_{\eps}.
\end{equation}
This implies $u_{\eps}'>0$ as long as $u_{\eps}$ is small. In particular, the solution stays positive.

Let $t_0>\lambda$ denote the first instance\footnote{If  $u_{\eps}$ always stays below $\frac{1}{8(d+1)}$, we take $t_0$ to be $\pi/2$.} when $u_{\eps}$ reaches $\frac{1}{8(d+1)}$, that is,
\begin{equation}
\label{EqnExitingTimeForEqnNearSouthPole}
0\le u_{\eps}\le\frac{1}{8(d+1)} \text{ on }[\lambda,t_0], \text{ and } u_{\eps}(t_0)=\frac{1}{8(d+1)}.
\end{equation} 
Then it follows from \eqref{EqnAConsequenceOfEquivalentFOrm} that 
\begin{equation*}
\label{EqnUEps>0}
u_{\eps}'>0 \text{ on }(\lambda,t_0].
\end{equation*} 
\vem

\textit{Step 2: Estimates on $[\lambda,t_0]$.}

We start with an estimate on $\frac{(u_{\eps}')^2}{u_{\eps}+\eps}$. 

Note that this quantity vanishes at $\theta=\lambda$. Thus at its maximum point $\bar{t}$ on $[\lambda,t_0]$, we have
$
\DTheta|_{\theta=\bar{t}}\frac{(u_{\eps}')^2}{u_{\eps}+\eps}\ge0.
$
With the non-negativity of $u_{\eps}'$ and $u_{\eps}$, this gives
$$
[2u_{\eps}''-\frac{(u_{\eps}')^2}{u_{\eps}+\eps}]|_{\theta=\bar{t}}\ge0.
$$
Using the ODE in \eqref{EqnEqnNearSouthPolePerturbed} and the bound on $u_\eps$ in  \eqref{EqnExitingTimeForEqnNearSouthPole}, we get 
$$
2(1+\frac{\kappa}{8(d+1)})\ge(1-\kappa/2)\frac{(u_{\eps}')^2}{u_{\eps}+\eps}\hem \text{ at }\bar{t}.
$$
For $\kappa$ dimensionally small, this gives
\begin{equation}
\label{EqnBoundOnRatioTemp}
\frac{(u_{\eps}')^2}{u_{\eps}+\eps}\le 4 \text{ on }[\lambda,t_0].
\end{equation} 

\vem

Combined with \eqref{EqnEquivalentPerturbedSystem}  and \eqref{EqnExitingTimeForEqnNearSouthPole}, we have
$$
\sin^{1-d}(\theta)[\sin^{d-1}(\theta)u_{\eps}']'\sim_d 1 \text{ on }[\lambda,t_0].
$$
Using $u_{\eps}'(\lambda)=0$, we can integrate this relation to conclude
\begin{equation}
\label{EqnU'Sim}
u_{\eps}'\sim_d(\theta-\lambda) \text{ on }[\lambda,t_0].
\end{equation} 
With $u_{\eps}(\lambda)=0$, this implies
\begin{equation}
\label{EqnUSim}
u_{\eps}\sim_d(\theta-\lambda)^2 \text{ on }[\lambda,t_0].
\end{equation}
With \eqref{EqnU'Sim}, we have $\cot(\theta)u_{\eps}'=\BO(1)$. Putting this back into  \eqref{EqnEqnNearSouthPolePerturbed}, we conclude
$$
|u_{\eps}''|=\BO(1)  \text{ on }[\lambda,t_0].
$$

In summary, we have established all desired estimates for $u_\eps$ on $[\lambda,t_0]$.

\vem

\textit{Step 3: The conclusion.}

With \eqref{EqnExitingTimeForEqnNearSouthPole} and   \eqref{EqnUSim}, we see that
$$
(t_0-\lambda)\sim_d 1
$$
for dimensionally small $\lambda$. As a result, the estimates in the previous step hold on $[\lambda,B_d]$ for a dimensional constant $B_d$.

These estimates give enough compactness on the family $\{u_\eps\}$. We extract the limit $u$ of a subsequence as $\eps\to 0$. This limit satisfies estimates in the lemma. 
\end{proof}

The model for  the solution $u$ in Lemma \ref{LemEqnNearSouthPole} is the radial solution to the obstacle problem $P_\lambda$ in \eqref{EqnPLambda}. To be precise, given $\lambda,$ we set
$
\tilde{\lambda}=\sin(\lambda)
$ 
and define
\begin{equation}
\label{EqnModelPLambda}
p_\lambda(\theta):=P_{\tilde{\lambda}}(\sin(\theta)).
\end{equation}
This gives $|\tilde{\lambda}-\lambda|=\BO(\lambda^2)$. Together with the expansion of $P_\lambda$ in  \eqref{EqnPLambda},  we have
\begin{equation}
\label{EqnModelPLambdaApprox}
p_\lambda(\theta)=\frac{1}{2d}\sin^2(\theta)-\frac{\lambda^2}{2(d-2)}+\frac{\lambda^d}{d(d-2)}\sin^{-d+2}(\theta)+\BO(\lambda^3) \hem\text{ for }\theta\ge\lambda.
\end{equation} 

\begin{lem}
\label{LemExpansionNearTheSouthPole}
Under the same assumptions as in Lemma \ref{LemEqnNearSouthPole}, the solution $u$ to \eqref{EqnEqnNearSouthPole} satisfies, on $[\lambda,B_d]$, 
$$
u(\theta)=\frac{1}{2d}\sin^2(\theta)-\frac{\lambda^2}{2(d-2)}+\frac{\lambda^d}{d(d-2)}\sin^{-d+2}(\theta)+\BO(\lambda^3)+(\theta-\lambda)\theta \BO(\kappa+\theta^2)
$$
and
$$
u'(\theta)=\frac{1}{d}\sin(\theta)\cos(\theta)-\frac{\lambda^d}{d}\sin^{-d+1}(\theta)\cos(\theta)+\BO(\lambda^2)+(\theta-\lambda) \BO(\kappa+\theta^2).
$$
Here $B_d$ is the dimensional constant in Lemma \ref{LemEqnNearSouthPole}.
\end{lem} 
\begin{proof}
Since $P_{\tilde{\lambda}}$ solves 
\eqref{EqnClassicalObstacleProblem}, the profile in \eqref{EqnModelPLambda} solves
$$
\begin{cases}
p_\lambda''+(d-1)\cot(\theta)p_\lambda'=\cos^2(\theta)-P_{\tilde{\lambda}}'(\sin(\theta))\sin(\theta) &\text{ on }(\lambda,\pi/2),\\
p_\lambda'=p_\lambda=0 &\text{ at }\lambda.
\end{cases}
$$
In particular, the difference
\begin{equation}
\label{EqnDifferenceUPLambda}
w:=u-p_\lambda
\end{equation}
solves \eqref{EqnAppendixODE2} on $(\lambda,\pi/2)$ with right-hand side 
$$
\rho:=1-2(d+1)u+\frac{\kappa}{4}\frac{(u')^2}{u}+\kappa u-\frac{\kappa}{2d}-\cos^2(\theta)+P'_{\tilde{\lambda}}(\sin(\theta))\sin(\theta).
$$
With Lemma \ref{LemEqnNearSouthPole}, we have
\begin{equation}
\label{EqnBigOForRho}
|\rho|\le\BO(\kappa+\theta^2) \text{ on }[\lambda,B_d].
\end{equation}

\vem

Lemma \ref{LemVariationOfParameters2} gives
$$
w(\theta)=\varphi(\theta) \alpha(\theta)+\beta(\theta), \text{ and }w'(\theta)=\varphi'(\theta)\alpha(\theta),
$$
where 
$$
\alpha(\theta)=-\int_{\lambda}^{\theta}\rho(\tau)\sin^{d-1}(\tau)d\tau,
\text{ and }
\beta(\theta)=\int_{\lambda}^{\theta}\rho(\tau)\varphi(\tau)\sin^{d-1}(\tau)d\tau.
$$

With \eqref{EqnBigOForRho} and \eqref{EqnFundamentalSolutionODE2}, these can be estimate as
$$
|\alpha|(\theta)=(\theta-\lambda)\theta^{d-1}\BO(\kappa+\theta^2), \text{ and }|\beta|(\theta)=(\theta-\lambda)\BO(\theta),
$$
which gives
$$
|w(\theta)|=(\theta-\lambda)\theta \BO(\kappa+\theta^2), \text{ and }|w'|(\theta)=(\theta-\lambda)\BO(\kappa+\theta^2) \text{ on }[\lambda,B_d].
$$

The conclusion follows from \eqref{EqnModelPLambdaApprox} and  \eqref{EqnDifferenceUPLambda}.
\end{proof} 

\subsection{Gluing of the solutions}
\label{SubsectionGluingOfTheSolutionsUAS}
In Subsection \ref{SubsectionEqnNearEquator}, we constructed a family of solutions near the equator. To be precise, we have a family of solutions to \eqref{EqnNearEquator} on $[A_d(\mu^{\frac12}+\kappa^{\frac{1}{d-1}}),\pi/2]$ (see Lemma \ref{CorFinerExpansion}). This family is parametrized by $\mu$ in the terminal condition of \eqref{EqnNearEquator}. 

In Subsection \ref{SubsectionEqnNearSouthPole}, we constructed a family of solutions near the south pole, namely, solutions to \eqref{EqnEqnNearSouthPole} on $[\lambda,B_d]$ (see Lemma \ref{LemExpansionNearTheSouthPole}). This family is parametrized by  $\lambda$, the size of the contact set in \eqref{EqnEqnNearSouthPole}.

To get a solution on the entire sphere, we need to glue these two families. That is, by adjusting the parameters $\mu$ and $\lambda$, we show that the two solutions have matching values and derivatives at a point in their common domain. Here we rely heavily on the expansions of the solutions and their derivatives in Lemma \ref{CorFinerExpansion} and Lemma \ref{LemExpansionNearTheSouthPole}.

To be precise, we have
\begin{lem}
\label{LemGlueAS}
For $d\ge3$ and $\kappa>0$ dimensionally small, there is a solution $u\ge 0$ to \eqref{EqnNonlinearODE} on $(0,\pi/2)$ satisfying $u'(\pi/2)=0$.

Moreover, there are constants $\theta_0<\theta_m$ with 
$$
\theta_0\sim_d\kappa^{\frac{1}{d-2}},\text{ and }\hem\theta_m/\theta_0=\mathcal{T}_d
$$
for a large dimensional constant $\mathcal{T}_d$,
such that 
$$
u(\theta)\sim_d[(\theta-\theta_0)_+]^2,  \hem  0\le u'(\theta)\le \BO((\theta-\theta_0)_+) \text{ on }[0,\pi/2],
$$
and
$$
u'(\theta)\sim_d (\theta-\theta_0)_+ \text{ on }[0,\theta_m].
$$
\end{lem} 
Recall the small parameter $\kappa$ from \eqref{EqnKappa}. In the statement of this lemma, the angle $\theta_0$ corresponds to the opening of the contact set of $u$. The angle $\theta_m$ corresponds to the point where we glue the solutions from Subsection \ref{SubsectionEqnNearEquator} and Subsection \ref{SubsectionEqnNearSouthPole}.
\begin{proof}
\textit{Step 1: Preparations.}

Denote by $u_e$ the solution to \eqref{EqnNearEquator} with parameter $\mu$. Denote by $u_o$ the solution to \eqref{EqnEqnNearSouthPole} with parameter  $\lambda$. 

For some constants $M,L$ and $T$ to be chosen,  bounded away from $0$ and infinity, we take 
$$\mu=M\kappa^{\frac{2}{d-2}}, \hem \lambda=L\kappa^{\frac{1}{d-2}}, \text{ and } \theta_m=T\kappa^{\frac{1}{d-2}}.$$
With a proper choice of $M, L$ and $T$, we will show that 
$$
u_e(\theta_m)=u_o(\theta_m), \text{ and }u_e'(\theta_m)=u'_o(\theta_m).
$$

\vem

\textit{Step 2: Matching data.}

With Lemma \ref{CorFinerExpansion} and Lemma \ref{LemExpansionNearTheSouthPole}, we have
\begin{equation}
\label{EqnMatchingData1}
\frac{u_e-u_o}{\kappa^{\frac{2}{d-2}}}(\theta_m)=-dM+\frac{L^2}{2(d-2)}+\BO(\frac{M+L^d}{T})+\BO_{d,M,L,T}(\kappa^{\frac{1}{d-2}}),
\end{equation}
and
\begin{equation}
\label{EqnMatchingData2}
\frac{(u_e'-u_o')(\theta_m)}{\kappa^{\frac{d}{d-2}}\sin^{-d+1}(\theta_m)}=-\frac{\alpha_d}{d}M+\frac{L^d}{d}+\BO_{d,M,L,T}(\kappa^{\frac{d-5/2}{d-2}}+\kappa^{\frac{1}{d-2}}).
\end{equation}

Define a map $F:\R^2\to\R^2$ as
$$
F(M,L):=(-dM+\frac{L^2}{2(d-2)},-\frac{\alpha_d}{d}M+\frac{L^d}{d}),
$$
then there is a unique pair of positive numbers  $(M_0,L_0)$ solving 
$$
F(M_0,L_0)=0.
$$
Moreover, the differential of $F$ at this point, $DF(M_0,L_0)$, is invertible. 

By the inverse function theorem, we get $0<\delta_F,\delta_F'<1$ such that $F$ is invertible from a $\delta_F$-neighborhood of $(M_0,L_0)$ to a $\delta_F'$-neighborhood of $0$. By fixing $T\gg (M_0+L_0^d)$ and then $\kappa$ small, we see the error  in \eqref{EqnMatchingData1} and \eqref{EqnMatchingData2} can be made arbitrarily small, uniformly in this neighborhood of $(M_0,L_0)$. 

As a result, for this choice of $T$ and small $\kappa>0$, we can find $(M,L)$ such that 
$$
u_e(\theta_m)=u_o(\theta_m), \text{ and }u_e'(\theta_m)=u_o'(\theta_m).
$$

\vem

\textit{Step 3: Properties of the solution.}

Define the function $u$ as 
$$
u:=u_o \text{ on }[0,\theta_m], \text{ and }u:=u_e \text{ on }(\theta_m,\pi/2].
$$
Then $u$ solves \eqref{EqnNonlinearODE} on $(0,\pi/2)$ since $u_e$ and $u_o$ both solve this equation in their domains of definition. Recall that $u_e'(\pi/2)=0$, the same holds for $u'$.

Define $\theta_0=\lambda=L\kappa^{\frac{1}{d-2}}$, then $\{u=0\}=\{u_o=0\}=[0,\theta_0]$.

Estimates on $[0,\theta_m]$ have been established in Lemma \ref{LemEqnNearSouthPole}. Below we focus on the estimates on $(\theta_m,\pi/2]$.
\vem

From Lemma \ref{LemAnalysisOnV}, we have
$$
\frac{|u'(\theta)|}{|\theta-\theta_0|}=\BO(\frac{\theta}{\theta-\theta_0})+\mu \BO(\theta^{-1}(\theta-\theta_0)^{-1}).
$$
Note that on $(\theta_m,\pi/2)$, we have $|\theta_0/\theta|\le L/T<\frac14$ since $T\gg L$ as in the previous step.  We have
$$
|u'|=\BO(\theta-\theta_0) \text{ on }[\theta_m,\pi/2].
$$
The estimate on $u$ follows by integrating this relation. 

\vem

It remains to show that $u'=u_e'\ge0$ on $(\theta_m,\pi/2]$.

With $p(\pi/2)=\frac{1}{2d}$, we find a dimensional constant $\theta_d\in(\pi/4,\pi/2)$ such that 
$$
p>\frac{1}{2d}-\frac{1}{100d^2} \text{ on }(\theta_d,\pi/2].
$$
With Lemma \ref{LemAnalysisOnV}, we get
$$
u\ge\frac{1}{2d}-\frac{1}{100d^2}-4d\mu>\frac{1}{2d}-\frac{1}{10d^2} \text{ on }(\theta_d,\pi/2]
$$
if $\kappa$ is dimensionally small. 

Putting this into \eqref{EqnNearEquator}, we have
$$
\sin^{1-d}(\theta)[\sin^{d-1}(\theta)u']'=u''+(d-1)\cot(\theta)u'=1-2(d+1)u+\BO(\kappa)\le 0 \text{ on }(\theta_d,\pi/2).
$$
That is, the function $\sin^{d-1}(\theta)u'$ is decreasing on this interval. Combined with the terminal condition in \eqref{EqnNearEquator}, we see that
$$
u'\ge0 \text{ on }(\theta_d,\pi/2].
$$

On $[\theta_m,\theta_d]$, Lemma \ref{LemAnalysisOnV} implies
$$
u'\ge\frac{1}{d}\sin(\theta)\cos(\theta)-\mu|v'|(\theta)\ge c_d(\theta-\mu\theta^{-1})\ge \frac{c_d}{\theta}(\theta_m^2-\mu).
$$
Recall $\theta_m=T\kappa^{\frac{1}{d-2}}$ and $\mu=M\kappa^{\frac{2}{d-2}}$ with $T\gg M$, the right-hand side is non-negative. 
\end{proof} 

With this, we give the proof of the main result of this section.

\begin{proof}[Proof of Proposition \ref{PropConstructionOfUAs}]
By choosing the constant $\gamma_d^3$ close to $1$, we make the parameter $\kappa$ from \eqref{EqnKappa} small so that Lemma \ref{LemGlueAS} holds.

Let $u$ be the solution in Lemma \ref{LemGlueAS}. Note that the vanishing condition on $u'(\pi/2)$ allows us to evenly reflect $u$ to $(\pi/2,\pi]$ and get a solution on $(0,\pi)$.

Using the coordinates in \eqref{EqnCoordinate}, we define $\UAS$ as
$$
\UAS(r,\theta)=r^2u(\theta).
$$
Since $u$ solves \eqref{EqnNonlinearODE}, this $2$-homogeneous function $\UAS$ solves \eqref{EqnTransformedEquation}. As a consequence of Proposition \ref{PropEquivalenceOfEquations}, we have $\UAS\in\SGa(\R^d)$.

By the homogeneity of $\UAS$, the estimate on $|\{\UAS=0\}\cap B_1|$ follows from the estimate on $|\{u=0\}|$ in Lemma \ref{LemGlueAS}.
\end{proof}

\section{Lower foliation of the axially symmetric cone}
\label{SectionLowerFoliationAS}
In the previous section, we constructed an axially symmetric cone to our problem, namely,  the profile $\UAS$ in Proposition \ref{PropConstructionOfUAs}, as a solution to the Euler-Lagrange equation \eqref{EqnTransformedEquation}. 

Starting from this section, we address its minimality. Based on Theorem \ref{ThmMinimality}, this amounts to the construction of a foliation of $\UAS$.
In this section, we focus on the lower foliation (see Definition \ref{DefLowerFoliation}). In Section \ref{SectionUFAS}, we construct the upper foliation. 

To be precise, the \textbf{main result} of this section is

\begin{prop}
\label{PropLowerFoliationAS}
For $d\ge3$, there is a dimensional constant $\gamma_d^3\in(0,1)$ such that for 
$$
\gamma_d^3<\gamma<1,
$$
there is a lower foliation $\{\Phi_t\}_{t\in(0,+\infty)}$ of the cone $\UAS$ in Proposition \ref{PropConstructionOfUAs}.
\end{prop} 
Recall we work in $\R^{d+1}$, decomposed as in \eqref{DefSpaceDecomposition}. 
Like in Section \ref{SectionCAS}, we use the transformed equation \eqref{EqnTransformedEquation}. According to Proposition \ref{PropEquivalenceOfEquations}, this is equivalent to the original Euler-Lagrange equation in Definition \ref{DefSolution}.

\vem

We briefly explain the \textbf{strategy of the construction}. 

The leaves in this lower foliation arise as rescalings of a single leaf in \eqref{EqnSingleLowerLeaf}. To build this leaf, the main ingredient is a positive supersolution to the linearized operator around the cone (see \eqref{EqnLU}). 
The study of this linearized operator is done in Subsection \ref{SubsectionLEAU} with the positive supersolution, $v$, given in \eqref{EqnV} and \eqref{EqnBV}.

We perturb our cone $\UAS$ by the positive supersolution $v$, leading to the profile $U$ in \eqref{EqnLowerLeafFar}. The properties of $U$ are studied in Subsection \ref{SubsectionLowerAwayFromFreeBoundary}. If we stay away from the free boundary, namely, in the region $\Omega$ from \eqref{EqnAwayFromFBLower}, the linearized equation approximates  well the full equation, and the profile $U$ is indeed a subsolution (see Lemma \ref{LemLowerLeafFar}). 

Unfortunately, this is no longer true once we leave the domain $\Omega$ in \eqref{EqnAwayFromFBLower}. To go beyond $\Omega$, we extend $U$ in each horizontal slice by an explicit quadratic profile $V$ in \eqref{EqnHorizontalFillingLower}. In Subsection \ref{SubsectionLowerNearFreeBoundary}, we show that this profile $V$ has the desired properties (see Lemma \ref{LemEquationInOLower}, Lemma \ref{LemAngleAlongGlue} and Lemma \ref{LemOrderingLower}). This is the most technical portion of this section. Here we rely on information of the trace of $U$ on $\partial\Omega$ (see Lemma \ref{LemTechnical1} and Lemma \ref{LemTechnical2}).

With all these preparations, we construct the lower foliation in Subsection \ref{SubsectionLowerFoliationOfTheASCone} and give the proof of Proposition \ref{PropLowerFoliationAS}.

\subsection{Linearized operator around the cone}
\label{SubsectionLEAU}

We work with $\gamma\in(0,1)$ close to $1$. The space  $\R^{d+1}$  is decomposed as in \eqref{DefSpaceDecomposition}, equipped with the coordinates $(r,\theta)$ in \eqref{EqnCoordinate}.

For brevity, we \textit{use $u$ to denote the  cone $\UAS$} from Proposition \ref{PropConstructionOfUAs}. Its \textit{trace on the sphere} is denoted by $\BU$, that is,
\begin{equation}
\label{EqnTraceOfUAS}
u(r,\theta):=\UAS(r,\theta)=:r^2\BU(\theta).
\end{equation}

Basic properties of $\BU$ are listed in Lemma \ref{LemGlueAS}. In particular,  we have
\begin{equation}
\label{EqnPositiveSets}
\{\BU>0\}=(\theta_0,\pi-\theta_0), \text{ and } \PosS=\{r>0,\hem \theta\in(\theta_0,\pi-\theta_0)\},
\end{equation}
where $\theta_0$ is the opening of the contact set in Lemma \ref{LemGlueAS}. It  satisfies
$$
\theta_0\sim_d\kappa^{\frac{1}{d-2}},
$$
where $\kappa$ is the small parameter in \eqref{EqnKappa}.

As part of the main result from Section \ref{SectionCAS}, this cone $u$ solves the transformed Euler-Lagrange equation \eqref{EqnTransformedEquation}. The \textit{linearized operator around }$u$ is given by
\begin{equation}
\label{EqnLU}
\LU(v):=\Delta v+(\frac{\beta}{2}-1)[\frac{2\nabla u}{u}\cdot\nabla v-\frac{|\nabla u|^2}{u^2}v].
\end{equation} 
For a function of the form $v(r,\theta)=f(r)\BV(\theta)$, this operator splits into two pieces
\begin{equation}
\label{EqnSplittingLU}
\LU(v)=\frac{f}{r^2}\cdot\LBU(\BV)+[f''+\frac{d-\kappa}{r}f'+\kappa\frac{f}{r^2}]\cdot\BV,
\end{equation} 
where $\kappa$ is from \eqref{EqnKappa}, and $\LBU$ is the \textit{linearized operator around} $\BU$, defined as
\begin{equation}
\label{EqnLBU}
\LBU(\BV):=\BV''+[(d-1)\cot(\theta)-\frac{\kappa}{2}\frac{\BU'}{\BU}]\BV'+\frac{\kappa}{4}(\frac{\BU'}{\BU})^2\BV.
\end{equation} 

\vem


We give a \textit{positive supersolution} $v$ to the linearized operator $\LU$  in $ \PosS=\{r>0,\hem \theta\in(\theta_0,\pi-\theta_0)\}$ (see \eqref{EqnPositiveSets}). This supersolution is of the form 
\begin{equation}
\label{EqnV}
v:=r^{-\frac12}\BV(\theta)
\end{equation} 
in our coordinate system \eqref{EqnCoordinate}.

The function $\BV$ is constructed piecewisely on $(\theta_0,\theta_1)$, $(\theta_1,\pi-\theta_1)$ and $(\pi-\theta_1,\pi-\theta_0)$, where
\begin{equation}
\label{EqnTheta1}
\theta_1:=(1+\kappa^\frac18)\theta_0.
\end{equation} 
We define $\BV$ as
\begin{equation}
\label{EqnBV}
\BV:=
\begin{cases}
\sigma(\theta-\theta_0)^\alpha &\text{ on }(\theta_0,\theta_1],\\
\sin^{-\frac12}(\theta) &\text{  on }(\theta_1,\pi-\theta_1],\\
\sigma(\pi-\theta_0-\theta)^\alpha &\text{ on }(\pi-\theta_1,\pi-\theta_0),
\end{cases}
\end{equation} 
where $\alpha>0$ is a small parameter to be chosen (see Remark \ref{RemParametersFixing}), and $\sigma$ satisfies $$\sigma(\theta_1-\theta_0)^\alpha=\sin^{-\frac12}(\theta_1).$$

\begin{rem}
\label{RemProblemAtTheta1}
The function $\BV$ is Lipschitz in compact subsets of $(\theta_0,\pi-\theta_0)$.
Once a  uniform estimate  on $\BV'$ has been established on $(\theta_0,\pi-\theta_0)\backslash\{\theta_1,\pi-\theta_1\}$, it is understood  in the weak sense over the entire interval. 

The second derivative $\BV''$, however, involves a singular measure at $\theta_1$ and $\pi-\theta_1$.  It is crucial that this measure is negative. 
\end{rem} 

\begin{rem}
\label{RemParametersFixing}
We need  several small parameters, in particular, the parameter $\alpha$ in the definition of $\BV$ (see \eqref{EqnBV}), and the small parameter $\kappa$ in \eqref{EqnKappa}. 

The parameter $\alpha$ is chosen in Lemma \ref{LemCommutator} to get the lower bound for $(\BU'/\BU-\BV'/\BV)$ (see \eqref{EqnFirstEquationInProof}).  Given this choice of $\alpha$, the parameter $\kappa$ is chosen small in the remaining part of this section. The link between $\kappa$ and $\gamma$ in \eqref{EqnKappa} gives $\gamma_d^3$ in Proposition \ref{PropLowerFoliationAS}.
\end{rem} 

We begin with some basic estimates on $\BV$ in \eqref{EqnBV}:

\begin{lem}
\label{LemCommutator}
For $d\ge3$ and dimensionally small $\alpha$ and $\kappa$,  we have 
$$
|\BV'/\BV|\le\BO((\theta-\theta_0)^{-1}\cos(\theta)),\hem (\frac{\BU'}{\BU}-\frac{\BV'}{\BV})\sim_d(\theta-\theta_0)^{-1}\cos(\theta) \text{ on }(\theta_0,\pi/2],
$$
and
$$
 |(\BV'/\BV)'|\le\BO((\theta-\theta_0)^{-2})  \text{ on }(\theta_0,\theta_1)\cup(\theta_1,\pi/2].
 $$
\end{lem} 
The big-O notation is from Remark \ref{RemSimAndO}. The angle $\theta_0$ is from Lemma \ref{LemGlueAS}. The angle $\theta_1$ is from the definition of $\BV$ in \eqref{EqnTheta1} and  \eqref{EqnBV}. 

\begin{rem}
\label{RemHalfIntervalAndViscosity}
We often state results for the interval $(0,\pi/2]$. By even symmetry, similar properties hold on the entire $(0,\pi)$.
\end{rem} 

\begin{proof}
It is elementary to verify the bounds on $|\BV'/\BV|$ and $|(\BV'/\BV)'|$. The upper bound for $(\frac{\BU'}{\BU}-\frac{\BV'}{\BV})$ follows from the definition of $\BV$ in \eqref{EqnBV} and Lemma \ref{LemGlueAS}. Below we focus on the lower bound for $(\frac{\BU'}{\BU}-\frac{\BV'}{\BV})$.

\vem

We first address the lower bound on $(\theta_0,\theta_m]$, where $\theta_0$ and $\theta_m$  are the angles in Lemma \ref{LemGlueAS}, satisfying
\begin{equation}
\label{EqnTheta0And1}
\theta_0\sim_d\kappa^{\frac{1}{d-2}} \text{ and }\theta_m/\kappa^{\frac{1}{d-2}}\gg 1.
\end{equation}

On this interval, Lemma \ref{LemGlueAS} gives a dimensional constant $c_d>0$ such that 
$$
\BU'(\theta)/\BU(\theta)\ge c_d(\theta-\theta_0)^{-1} \text{ on }(\theta_0,\theta_m].
$$

With \eqref{EqnTheta1} and \eqref{EqnTheta0And1}, we see that $\theta_1<\theta_m$.
The definition of $\BV$ in \eqref{EqnBV} gives
$
\BV'(\theta)/\BV(\theta)=\alpha(\theta-\theta_0)^{-1}
$
on $(\theta_0,\theta_1)$.
Thus
\begin{equation}
\label{EqnFirstEquationInProof}
(\frac{\BU'}{\BU}-\frac{\BV'}{\BV})(\theta)\ge(c_d-\alpha)(\theta-\theta_0)^{-1}\ge \frac12 c_d(\theta-\theta_0)^{-1} \text{ on }(\theta_0,\theta_1)
\end{equation}
if $\alpha$ is  small. 

On $(\theta_1,\theta_m]$ we have $\BV'\le 0$ (see \eqref{EqnBV}). This implies
 \begin{equation}\label{EqnSecondEquationInProof}
 (\frac{\BU'}{\BU}-\frac{\BV'}{\BV})(\theta)\ge(\frac{\BU'}{\BU})(\theta)\ge c_d(\theta-\theta_0)^{-1}.
 \end{equation}

 On $(\theta_m,\pi/2)$, Lemma \ref{LemGlueAS} gives $\BU'\ge 0$. Using \eqref{EqnBV}, we have
 \begin{equation}\label{EqnThirdEquationInProof}
 (\frac{\BU'}{\BU}-\frac{\BV'}{\BV})\ge-\frac{\BV'}{\BV}=\frac12\sin^{-1}(\theta)\cos(\theta)\ge c_d(\theta-\theta_0)^{-1}\cos(\theta).
 \end{equation}
 For the last inequality, we used $\theta_m=\mathcal{T}_d\theta_0$ with $\mathcal{T}_d\gg1$ as in Lemma \ref{LemGlueAS}.

Combining \eqref{EqnFirstEquationInProof}, \eqref{EqnSecondEquationInProof} and \eqref{EqnThirdEquationInProof}, we get the lower bound on $(\theta_0,\theta_1)\cup(\theta_1,\pi/2]$. In the weak sense, the bound holds over the entire interval. See Remark \ref{RemProblemAtTheta1}.
\end{proof} 

The supersolution property of $v$  is given below. To absorb errors from the nonlinearity, we add an extra term to the linearized operator.
\begin{lem}
\label{LemSupersolutionLU}
Under the same assumptions as in Lemma \ref{LemCommutator}, we have
$$
\LU(v)+\kappa^{\frac12}v/u\le 0 \text{ in }\PosS.
$$
\end{lem} 
Here $v$ is the function in \eqref{EqnV}, and the operator $\LU$ is given in \eqref{EqnLU}. 

\begin{proof}
From \eqref{EqnPositiveSets}, the positive set $\PosS$ is given by
$$ \PosS=\{r>0,\hem \theta\in(\theta_0,\pi-\theta_0)\}.$$
Due to the piecewise nature of $\BV$ in \eqref{EqnBV}, we divide the proof into several regions, firstly on on $(\theta_1,\pi/2)$, then on $(\theta_0,\theta_1)$, and lastly at $\theta_1$.

\vem

\textit{Step 1: Equation for $v$ in $\{r>0,\theta\in(\theta_1,\pi/2)\}$.}

On $(\theta_1,\pi/2)$, the definition of $\BV$ in \eqref{EqnBV} gives
\begin{equation}
\label{Eqn712}
\LBU(\BV)/\BV=[-\frac{d-1}{2}+\frac34]\sin^{-2}(\theta)\cos^2(\theta)+\frac12+\frac{\kappa}{4}[\frac{\BU'}{\BU}\sin^{-1}(\theta)\cos(\theta)+(\frac{\BU'}{\BU})^2],
\end{equation} 
where $\LBU$ is the linearized operator around $\BU$ from \eqref{EqnLBU}.

With  Lemma \ref{LemGlueAS} and \eqref{EqnTheta1}, we have
\begin{align*}
\frac{\BU'}{\BU}\sin^{-1}(\theta)\cos(\theta)&=\sin^{-2}(\theta)\cos^2(\theta)\BO(\kappa^{-\frac{1}{8}})+\BO(\kappa^{-\frac{1}{8}}) \hem\text{ on }(\theta_1,\pi/2),
\end{align*}
and
$$
(\frac{\BU'}{\BU})^2=\sin^{-2}(\theta)\cos^2(\theta)\BO(\kappa^{-\frac{1}{4}})+\BO(\kappa^{-\frac{1}{4}}) \hem\text{ on }(\theta_1,\pi/2).
$$

Putting these into \eqref{Eqn712}, we have
$$
\LBU(\BV)/\BV\le[-\frac14+\BO(\kappa^{1-\frac{1}{4}})]\sin^{-2}(\theta)\cos^2(\theta)+\frac12+\BO(\kappa^{1-\frac{1}{4}}) \text{ on }(\theta_1,\pi/2).
$$

With Lemma \ref{LemGlueAS}, we have
$$
\BU^{-1}=\sin^{-2}(\theta)\cos^2(\theta)\BO(\kappa^{-\frac{1}{4}})+\BO(\kappa^{-\frac{1}{4}}) \text{ on }(\theta_1,\pi/2).
$$
As a result, we have
\begin{align}
\label{TheFirstIntervalTrace}
\LBU(\BV)/\BV+\kappa^{\frac12}/\BU&\le [-\frac14+\BO(\kappa^{\frac12-\frac{1}{4}})]\sin^{-2}(\theta)\cos^2(\theta)+\frac12+\BO(\kappa^{\frac12-\frac{1}{4}})\\   
&\le \frac12+\BO(\kappa^{\frac{1}{4}})  \text{ on }(\theta_1,\pi/2). \nonumber
\end{align} 

With \eqref{EqnSplittingLU}, \eqref{EqnV} and \eqref{TheFirstIntervalTrace}, we have,  in $\{r>0,\hem\theta\in(\theta_1,\pi/2)\}$,
$$
\LU(v)+\kappa^{\frac12}\frac{v}{u}\le r^{-\frac52}\BV\cdot\{-\frac14+\BO(\kappa^{\frac14})\}\le 0,
$$
for $\kappa$ small.

\vem

\textit{Step 2: Equation for $v$ in $\{r>0,\hem\theta\in(\theta_0,\theta_1)\}$.}

With similar ideas as in Step 1, we have
$$
\LBU(\BV)/\BV+\kappa^{\frac12}/\BU\le(\theta-\theta_0)^{-2}[\alpha(\alpha-1+\BO(\kappa^{\frac14}))+\BO(\kappa^{\frac12})] \text{ on }(\theta_0,\theta_1).
$$
As a result, we have, in $\{r>0,\hem\theta\in(\theta_0,\theta_1)\}$,
$$
\LU(v)+\kappa^{\frac12}\frac{v}{u}\le r^{-\frac52}\BV\cdot\{(\theta-\theta_0)^{-2}[-\frac14\alpha+\BO(\kappa^{\frac14})]-\frac34+\kappa/2\}\le 0
$$
 if $\kappa$ is small.

\vem

\textit{Step 3: The conclusion.}

We have established the desired inequality in $(\theta_0,\theta_1)\cup(\theta_1,\pi/2)$. At $\theta=\theta_1$, the trace $\BV$ forms a concave corner. The desired inequality holds in the viscosity sense. 

Similar results can be extended to $(\pi/2,\pi-\theta_0)$ by symmetry as in Remark \ref{RemHalfIntervalAndViscosity}. 
\end{proof}

\subsection{Region away from the free boundary}
\label{SubsectionLowerAwayFromFreeBoundary}
We begin the construction of a leaf in the lower foliation. In this subsection, we focus on the \textit{region away from the free boundary}, defined as
\begin{equation}
\label{EqnAwayFromFBLower}
\Omega:=\{u-v\ge\kappa^{\frac{1}{8}}u\}.
\end{equation} 
The functions $u$ and $v$ are from \eqref{EqnTraceOfUAS} and \eqref{EqnV}.
In this region $\Omega$, the leaf is taken as
\begin{equation}
\label{EqnLowerLeafFar}
U:=u-v.
\end{equation}

For a non-negative function $\varphi$, we define its error in solving \eqref{EqnTransformedEquation} as $E(\varphi)$, that is, 
\begin{equation}
\label{EqnTheError}
E(\varphi):=\Delta\varphi+(\frac{\beta}{2}-1)\frac{|\nabla\varphi|^2}{\varphi}-\frac{d+\beta-2}{d}\chi_{\{\varphi>0\}}.
\end{equation} 

The profile in \eqref{EqnLowerLeafFar} has the desired subsolution property in $\Omega$:
\begin{lem}
\label{LemLowerLeafFar}
For $d\ge3$ and $\alpha,\kappa$ dimensionally small, we have
$$
E(U)\ge 0 \text{ in }\Omega.
$$
\end{lem} 
Recall the parameter $\alpha$ in the definition of $\BV$ \eqref{EqnBV}. This parameter has been fixed in Lemma \ref{LemCommutator}. The parameter $\kappa$ is from \eqref{EqnKappa}.
\begin{proof}
By definition of $U$ in  \eqref{EqnLowerLeafFar}, we have
\begin{equation}
\label{EqnWeHaveEU}
E(U)=-\LU(v)+(\frac{\beta}{2}-1)[h(X+Z)-h(X)-\nabla h(X)\cdot Z],
\end{equation}
where $h:\R^{d+2}\to\R$ is given by $h(X',x_{d+2}):=|X'|^2/x_{d+2}$ when $x_{d+2}>0$, and $X=(\nabla u,u)$, $Z=-(\nabla v,v)$.

 The homogeneity of $u$ in \eqref{EqnTraceOfUAS} and Lemma \ref{LemGlueAS} lead to
$$
|\nabla u|^2/u=4\BU+(\BU')^2/\BU\le\BO(1). 
$$
Lemma \ref{LemAppendixLinearExpansion} gives
\begin{equation}
\label{EqnThePreviousEstimate}
|h(X+Z)-h(X)-\nabla h(X)\cdot Z|\le C_d(\frac{v}{u})^2[\frac{u}{u-v}+(\frac{u}{u-v})^2]+\frac{|\nabla v|^2}{u}(1+\frac{u}{u-v}).
\end{equation}

With \eqref{EqnAwayFromFBLower}, the first term is  bounded by $\frac{v}{u}\cdot \BO(\kappa^{-\frac{1}{4}})$ in $\Omega$.

For the second term, we use the homogeneity of $v$ and Lemma \ref{LemCommutator} to get
$$
\frac{|\nabla v|^2}{v^2}=\frac14r^{-2}+r^{-2}(\frac{\BV'}{\BV})^2\le r^{-2}\BO((\theta-\theta_0)^{-2}).
$$
Together with Lemma \ref{LemGlueAS}, this implies
$$
\frac{|\nabla v|^2/u}{(v/u)^2}=\frac{|\nabla v|^2}{v^2}\cdot u\le\BU\BO((\theta-\theta_0)^{-2})\le\BO(1).
$$
Thus the second term in \eqref{EqnThePreviousEstimate} is bounded by the first term. 

Consequently, we have
$$
|h(X+Z)-h(X)-\nabla h(X)\cdot Z|\le \frac{v}{u}\BO(\kappa^{-\frac{1}{4}}) \text{ in }\Omega.
$$
Putting this into \eqref{EqnWeHaveEU}, we get
$$
E(U)\ge-\LU(v)-\frac{v}{u}\cdot \BO(\kappa^{\frac{3}{4}})\ge-[\LU(v)+\kappa^{1/2}\frac{v}{u}]\ge 0 \text{ in }\Omega
$$
if $\kappa$ is small. The last inequality follows from Lemma \ref{LemSupersolutionLU}.
\end{proof}

\vem

In Subsection \ref{SubsectionLowerNearFreeBoundary}, we extend the profile $U$ in \eqref{EqnLowerLeafFar} to the complement of $\Omega$ in \eqref{EqnAwayFromFBLower}. To get  a subsolution, we need information about the trace of $U$ on the boundary of $\Omega$,  denoted as
\begin{equation}
\label{EqnGammaLower}
\Gamma:=\partial\Omega=\{U=\kappa^{\frac{1}{8}}u\}.
\end{equation} 
On $\Gamma$, we have $(1-\kappa^\frac{1}{8})u=v$, that is,
$$
(1-\kappa^\frac{1}{8})r^2\BU(\theta)=r^{-1/2}\BV(\theta)
$$
using \eqref{EqnTraceOfUAS} and \eqref{EqnBV}.

As a result, the boundary $\Gamma$ is described by a graph $\theta\in(\theta_0,\pi-\theta_0)\mapsto R(\theta)$, satisfying the relation
\begin{equation}
\label{EqnRadialGraphLower}
R^{5/2}(\theta)=\frac{1}{1-\kappa^\frac{1}{8}}\BV(\theta)/\BU(\theta).
\end{equation} 
Equivalently, the boundary $\Gamma$ can  be described by its distance to the $y$-axis at each height $y$, namely,   $y\in(-\infty,+\infty)\mapsto L(y)$, satisfying
\begin{equation}
\label{EqnYGraphLower}
y=-R(\theta)\cos(\theta), \text{ and }L(y)=R(\theta)\sin(\theta).
\end{equation}

The change of variable between these two descriptions  satisfies the following bounds. Due to its technical nature, the proof is postponed to Appendix \ref{AppendixQuantitiesOnGamma1}.
\begin{lem}
\label{LemTechnical1}
The quantities in \eqref{EqnRadialGraphLower} and \eqref{EqnYGraphLower} satisfy 
$$
0\le\frac{d\theta}{dy}\le \BO(1/R), \hem |\frac{dL}{dy}|\le \BO(1) \text{ on }(\theta_0,\pi/2) ,
$$
and
$$
 |\frac{d^2\theta}{dy^2}|\le (\frac{d\theta}{dy})^2\BO((\theta-\theta_0)^{-1}), \hem |\frac{d^2L}{dy^2}|\le R(\frac{d\theta}{dy})^2 \BO(\theta(\theta-\theta_0)^{-2}) \text{ on }(\theta_0,\theta_1)\cup(\theta_1,\pi/2).
$$

Moreover, for $\eta>0$ dimensionally small, we have
$$
|\frac{dL}{dy}|\ge c_d\eta^{-1}R\frac{d\theta}{dy}\cdot\theta(\theta-\theta_0)^{-1} \text{ on }(\theta_0,(1+\eta)\theta_0)
$$
for a dimensional constant $c_d>0$.
\end{lem} 
Recall $\theta_0$ from \eqref{EqnPositiveSets} and $\theta_1$ from \eqref{EqnTheta1}.

Along $\Gamma$, our profile $U$ in \eqref{EqnLowerLeafFar} satisfies the following lemma. The proof is postponed to Appendix \ref{AppendixQuantitiesOnGamma1}.
\begin{lem}
\label{LemTechnical2}
The function $y\mapsto\sqrt{U}(L(y),y)$ satisfies
$$
|\frac{d\sqrt{U}}{dy}|\le R\frac{d\theta}{dy}\BO(\kappa^\frac{1}{16}) \text{ on }(\theta_0,\pi/2), 
$$
and
$$
|\frac{d^2\sqrt{U}}{dy^2}|\le R(\frac{d\theta}{dy})^2\BO(\kappa^\frac{1}{16}(\theta-\theta_0)^{-1}) \text{ on }(\theta_0,\theta_1)\cup(\theta_1,\pi/2].
$$
\end{lem} 

\subsection{Region near the free boundary}
\label{SubsectionLowerNearFreeBoundary}
In this subsection, we extend the profile $U$ in \eqref{EqnLowerLeafFar} to the complement of $\Omega$ in \eqref{EqnAwayFromFBLower}. For each horizontal slice with fixed $y\in\R$, define
\begin{equation}
\label{EqnHorizontalFillingLower}
V(x,y):=[(|x|-L(y)+\sqrt{U(L(y),y)})_+]^2 \hem\text{ for }|x|\le L(y),
\end{equation} 
where $(L(y),y)$ is the description of $\Gamma=\partial\Omega$  in \eqref{EqnYGraphLower}. 

Note that by this definition, we have
\begin{equation}
\label{EqnContinuousFilling}
V=U \text{ on }\Gamma.
\end{equation}

We first verify the desired subsolution property of $V$ in the interior of $\R^{d+1}\backslash\Omega$, denoted as
\begin{equation}
\label{EqnInteriorOfComplementLower}
\CO:=\mathrm{Int}(\R^{d+1}\backslash\Omega)=\{(x,y):\hem y\in\R, \text{ and } |x|<L(y)\}.
\end{equation} 

With the notation in \eqref{EqnTheError}, we have
\begin{lem}
\label{LemEquationInOLower}
For $d\ge3$ and dimensionally small $\alpha$ and $\kappa$, we have
$$
E(V)\ge 0 \text{ in }\CO.
$$
\end{lem} 
Recall the parameter $\alpha$ in the definition of $\BV$ \eqref{EqnBV}. The parameter $\kappa$ is from \eqref{EqnKappa}.

\begin{proof}
\textit{Step 1: Preparations.}

By definition, we have $|\nabla V|=0$ on $\partial\{V>0\}$. As a result, it suffices to verify the inequality in $$
\tilde{\CO}:=\{V>0\}\cap\CO.
$$

Inside $\tilde\CO$, the definition of $V$ in \eqref{EqnHorizontalFillingLower} gives
\begin{equation}
\label{EqnXDerivatives1}
\frac{\partial}{\partial|x|}V=2(|x|-L+\sqrt{U})\ge0, \text{ and }\frac{\partial^2}{\partial|x|^2}V=2.
\end{equation}
As a result, we have
$$
\Delta V=2+\frac{d-1}{|x|}\frac{\partial}{\partial|x|}V+\frac{\partial^2}{\partial y^2}V\ge2+\frac{\partial^2}{\partial y^2}V \text{ in }\tilde\CO.
$$
This implies
\begin{equation}
\label{EqnTwoNonConstantTerms}
E(V)> 1+\frac{\partial^2}{\partial y^2}V +(\frac{\beta}{2}-1)|\nabla V|^2/V \text{ in }\tilde\CO,
\end{equation}
where we used $\beta<2$ for $\gamma<1$ (see \eqref{EqnScalingParameter}).

Below we estimate the two non-constant terms on the right-hand side. 

\vem

\textit{Step 2: Estimate on $|\nabla V|^2/V$.}

Inside $\tilde\CO$, we see from \eqref{EqnXDerivatives1} that 
$
(\frac{\partial}{\partial|x|}V)^2/V=4.
$

For the derivative in the $y$-variable, we have
\begin{equation}
\label{EqnDDY1}
\frac{\partial}{\partial y}V=2\sqrt{V}\frac{d}{dy}(-L+\sqrt{U}).
\end{equation} 
With Lemma \ref{LemTechnical1} and Lemma \ref{LemTechnical2}, we see that $\frac{d}{dy}(-L+\sqrt{U})$ is  bounded by a dimensional constant. As a result, we have
$$
(\frac{\partial}{\partial y}V)^2/V\le\BO(1).
$$

Summarize, we have
\begin{equation}
\label{EqnStep1OnRatio}
|\nabla V|^2/V\le\BO(1) \text{ in }\tilde\CO.
\end{equation} 

\vem

\textit{Step 3: Estimate of $\frac{\partial^2}{\partial y^2}V$ on $[(1+\eta)\theta_0,\pi/2]$.}

For a small $\eta>0$ as in Lemma \ref{LemTechnical1}, in this step we give the estimate on $\frac{\partial^2}{\partial y^2}V$ on $[(1+\eta)\theta_0,\pi/2]$.

Direct computation gives
\begin{equation}
\label{DDYV}
\frac{\partial^2}{\partial y^2}V=2[\frac{d}{dy}(-L+\sqrt{U})]^2+2\sqrt{V}\frac{d^2}{dy^2}(-L+\sqrt{U}) \text{ in }\tilde\CO.
\end{equation} 

With $V\le U(L(y),y)=\kappa^\frac{1}{8}R^2\BU$ in $\tilde\CO$ (see \eqref{EqnGammaLower}), Lemma \ref{LemGlueAS} and Lemma \ref{LemTechnical1} imply, on $(\theta_0,\theta_1)\cup(\theta_1,\pi/2)$,
$$
\sqrt{V}|\frac{d^2}{dy^2}L|\le\sqrt{U}|\frac{d^2}{dy^2}L|=\kappa^{1/16}R\BU^{1/2}\cdot R(\frac{d\theta}{dy})^2\BO(\theta(\theta-\theta_0)^{-2})\le\kappa^{1/16}\BO(\theta(\theta-\theta_0)^{-1}).
$$
In particular, on $[(1+\eta)\theta_0,\pi/2)$, we have
$$
\sqrt{V}|\frac{d^2}{dy^2}L|\le\BO(\kappa^{1/16}).
$$
Similarly, we have
$$
\sqrt{V}|\frac{d^2}{dy^2}\sqrt{U}|\le \BO(\kappa^\frac18) \text{ on }[(1+\eta)\theta_0,\pi/2). 
$$

Combining these with \eqref{DDYV}, we conclude
\begin{equation*}
\label{Step3}
\frac{\partial^2}{\partial y^2}V\ge-\BO(\kappa^{1/16}) \text{ on }[(1+\eta)\theta_0,\pi/2).
\end{equation*} 

\vem

\textit{Step 4: Estimate of $\frac{\partial^2}{\partial y^2}V$ on $(\theta_0,\theta_1)\cup(\theta_1,(1+\eta)\theta).$}

Recall the angle $\theta_1$ from the definition of $\BV$ in \eqref{EqnTheta1} and \eqref{EqnBV}.

Similar as in the previous step, using Lemma \ref{LemTechnical1} and Lemma \ref{LemTechnical2}, we have
$$
\sqrt{V}|\frac{d^2}{dy^2}L+\frac{d^2}{dy^2}\sqrt{U}|\le\kappa^{1/16}R^2(\frac{d\theta}{dy})^2\BO(\theta(\theta-\theta_0)^{-1}) \text{ on }(\theta_0,\theta_1)\cup(\theta_1,(1+\eta)\theta_0).
$$

On the other hand, with Lemma \ref{LemTechnical1} and Lemma \ref{LemTechnical2}, on $(\theta_0,(1+\eta)\theta_0)$, we have
\begin{equation}
\label{SomethingWeStillHave}
|\frac{d}{dy}(L-\sqrt{U})|\ge|\frac{d}{dy}L|-|\frac{d}{dy}\sqrt{U}|\ge R\frac{d\theta}{dy}[c_d\eta^{-1}\theta(\theta-\theta_0)-\kappa^{1/16}]\ge R\frac{d\theta}{dy}\theta(\theta-\theta_0)^{-1}.
\end{equation}
for $\eta$ and $\kappa$ dimensionally small. 

Combining these with \eqref{DDYV}, we have
\begin{equation*}
\label{Step4}
\frac{\partial^2}{\partial y^2}V\ge2(R\frac{d\theta}{dy})^2\cdot\{\theta^2(\theta-\theta_0)^{-2}-\kappa^{1/16}\BO(\theta(\theta-\theta_0)^{-1})\}\ge0
\end{equation*}
on $(\theta_0,\theta_1)\cup(\theta_1,(1+\eta)\theta_0)$ if $\kappa$ is dimensionally small. 

\vem

\textit{Step 5: Estimate of $\frac{\partial^2}{\partial y^2}V$ at $\theta=\theta_1.$}

With similar ideas as in previous steps, we have
$$
|\frac{d}{d\theta}(-L+\sqrt{U})|\le R\BO(\theta(\theta-\theta_0)^{-1}),
$$
and
\begin{equation}
\label{Delicate}
\frac{d^2}{d\theta^2}(-L+\sqrt{U})\ge\frac25R(\frac{\BV'}{\BV}-\frac{\BU'}{\BU})'(-\sin(\theta)+\kappa^{1/16}\BU^{1/2})-R\BO(\theta(\theta-\theta_0)^{-2}).
\end{equation}

By  \eqref{EqnBV}, the term $(\frac{\BV'}{\BV}-\frac{\BU'}{\BU})'$ gives a negative measure at $\theta=\theta_1$. Meanwhile, Lemma \ref{LemGlueAS} gives
\begin{equation}
\label{EqnSineDominatesAtTheta1}
-\sin(\theta_1)+\kappa^{1/16}\BU^{1/2}(\theta_1)\le-c_d\theta_1<0.
\end{equation}
Thus the first term in \eqref{Delicate} is non-negative. As a result, we have
\begin{equation*}
\frac{d^2}{d\theta^2}(-L+\sqrt{U})\ge-R\BO(\theta(\theta-\theta_0)^{-2}).
\end{equation*}

Together with Lemma \ref{LemTechnical1}, we have
\begin{align*}
\frac{d^2}{dy^2}(-L+\sqrt{U})&=\frac{d^2}{d\theta^2}(-L+\sqrt{U})(\frac{d\theta}{dy})^2+\frac{d}{d\theta}(-L+\sqrt{U})\frac{d^2\theta}{dy^2}\\
&\ge-R(\frac{d\theta}{dy})^2\BO(\theta(\theta-\theta_0)^{-2}).
\end{align*}
This implies
$$
\sqrt{V}\frac{d^2}{dy^2}(-L+\sqrt{U})\ge-\kappa^{1/16}R^2(\frac{d\theta}{dy})^2\BO(\theta(\theta-\theta_0)^{-1}).
$$

On the other hand, at $\theta=\theta_1$, we still have \eqref{SomethingWeStillHave}. As a consequence of \eqref{DDYV}, we have
\begin{equation*}
\label{Step5}
\frac{\partial^2}{\partial y^2}V\ge R^2(\frac{d\theta}{dy})^2(c_d-\kappa^{1/16})\theta(\theta-\theta_0)^{-1}\ge0
\end{equation*} 
at $\theta=\theta_1$ if $\kappa$ is small. 

Combining this with Step 3 and Step 4, we have established
\begin{equation}
\label{EqnStep345}
\frac{\partial^2}{\partial y^2}V\ge-\BO(\kappa^{1/16}) \text{ in }\tilde\CO.
\end{equation}

\vem

\textit{Step 6: Conclusion.}

Putting \eqref{EqnStep1OnRatio} and \eqref{EqnStep345} into \eqref{EqnTwoNonConstantTerms}, we get
$$
E(V)\ge 1-\BO(\kappa^{1/16})\ge 0 \text{ in }\tilde\CO
$$
if $\kappa$ is small. 
\end{proof} 

Recall from \eqref{EqnTwoNonConstantTerms} that $V$ in \eqref{EqnHorizontalFillingLower} and $U$ in \eqref{EqnLowerLeafFar} meet continuously along the interface $\Gamma$ in \eqref{EqnGammaLower}. To get a subsolution when gluing them, it is crucial that the derivatives are ordered.

\begin{lem}
\label{LemAngleAlongGlue}
Under the same assumptions as in Lemma \ref{LemEquationInOLower}, we have
$$
\frac{\partial}{\partial|x|}V<\frac{\partial}{\partial|x|}U \text{ along }\Gamma.
$$
\end{lem} 
\begin{proof}
With the definition of $U$ in \eqref{EqnLowerLeafFar}, direction computation with the coordinate system \eqref{EqnCoordinate} gives
$$
\frac{\partial}{\partial|x|}U=(2r\BU+\frac12 r^{-\frac32}\BV)\sin(\theta)+r^{-1}(r^2\BU'-r^{-1/2}\BV')\cos(\theta).
$$
Along $\Gamma$, the relation \eqref{EqnRadialGraphLower} implies
\begin{equation*}
\frac{\partial}{\partial|x|}U|_\Gamma\ge 2R\BU\sin(\theta)+R^{-1}(1-\kappa^\frac18)u(\frac{\BU'}{\BU}-\frac{\BV'}{\BV})\cos(\theta).
\end{equation*} 
With Lemma \ref{LemGlueAS} and Lemma \ref{LemCommutator}, we have
$$
\frac{\partial}{\partial|x|}U|_\Gamma\ge c_dR[\theta(\theta-\theta_0)^2+(\theta-\theta_0)\cos(\theta)].
$$

Meanwhile, from \eqref{EqnXDerivatives1} and Lemma \ref{LemGlueAS}, we get
$$
\frac{\partial}{\partial|x|}V|_\Gamma=2\sqrt{U}(R(\theta),\theta)=2\kappa^{1/16}R\BU^{1/2}\le\kappa^{1/16}R\BO(\theta-\theta_0).
$$

The desired comparison follows from the previous two estimates when $\kappa$ is small. 
\end{proof} 

We also have the desired ordering between  the solution $u$ in \eqref{EqnTraceOfUAS} and the extension $V$ in \eqref{EqnHorizontalFillingLower}.
\begin{lem}
\label{LemOrderingLower}
Under the same assumptions as in Lemma \ref{LemEquationInOLower}, we have
$$
V\le u \text{ in }\CO.
$$
\end{lem} 
Recall the domain $\CO$ from \eqref{EqnInteriorOfComplementLower}.
\begin{proof}
We verify this ordering for a horizontal slice with a fixed $y$-coorditate, say $\bar{y}$. Suppose that the point $(L(\bar{y}),\bar{y})$ on the boundary $\Gamma$ corresponds to $(R(\bar{\theta}),\bar{\theta})$. See \eqref{EqnRadialGraphLower} and \eqref{EqnYGraphLower}. Without loss of generality, we assume $\bar{\theta}\in(\theta_0,\pi/2]$.

Under this assumption,  points in $\{(x,\bar{y}): |x|\le L(\bar{y})\}$ satisfy 
$$
r\le R(\bar{\theta}), \text{ and }\theta\le\bar{\theta}
$$
in the coordinate system \eqref{EqnCoordinate}.
As a result, Lemma \ref{LemGlueAS} gives
\begin{equation}\label{OscilationOfSolution}
\frac{\partial u}{\partial|x|}=2r\BU\sin(\theta)+r\BU'\cos(\theta)\le\BO(R(\bar{\theta})(\bar{\theta}-\theta_0))\le\BO(u^{1/2}(L(\bar{y}),\bar{y}))
\end{equation}
in $\{(x,\bar{y}): |x|\le L(\bar{y})\}$.
\vem

In this slice  $\{(x,\bar{y}): |x|\le L(\bar{y})\}$, the extension $V$ vanishes for $|x|\le L(\bar{y})-\sqrt{U}(L(\bar{y}),\bar{y})$. As a result, it suffices to establish the ordering for $L(\bar y)-\sqrt{U}(L(\bar y),\bar y)\le|x|\le L(\bar y)$. Within this range, we use \eqref{OscilationOfSolution} to get
$$
u\ge u(L(\bar{y}),\bar{y})-\sqrt{U}(L(\bar{y}),\bar{y})\BO(u^{1/2}(L(\bar{y}),\bar{y}))\ge u(L(\bar{y}),\bar{y})-\kappa^{1/16}\BO(u(L(\bar{y}),\bar{y})).
$$
Note that we used $U=\kappa^{\frac18}u$ along $\Gamma$ (see \eqref{EqnGammaLower}).

That is, in $\{(x,\bar{y}):L(\bar{y})-\sqrt{U}(L(\bar{y}),\bar{y})\le |x|\le L(\bar{y})\}$ we have
$$
u\ge (1-C_d\kappa^{1/16})u(L(\bar{y}),\bar{y}).
$$

On the other hand, in the same region, we have
$$
V\le U(L(\bar{y}),\bar{y})=\kappa^{\frac18}u(L(\bar{y}),\bar{y}).
$$
The desired ordering holds if $\kappa$ is small. 
\end{proof}

\subsection{Lower foliation of the axially symmetric cone}
\label{SubsectionLowerFoliationOfTheASCone}
By gluing the profiles $U$ from \eqref{EqnLowerLeafFar} and $V$ from \eqref{EqnHorizontalFillingLower}, we get one leaf in the lower foliation, that is,
\begin{equation}
\label{EqnSingleLowerLeaf}
\Phi_1:=\begin{cases}
U &\text{ in }\Omega,\\
V &\text{ in }\R^{d+1}\backslash\Omega.
\end{cases}
\end{equation} 

This leaf has the desired properties to generate a lower foliation.
\begin{lem}
\label{LemGenerateLowerFoliation}
For $d\ge3$ and dimensionally small $\alpha$ and $\kappa$, we have
$$\Phi_1\in\SPGa(\R^{d+1}), \hem \Phi_1\le u \text{ in }\R^{d+1},$$
and 
$$
\{\Phi_1=0\}\supset B_{\eta} \text{ for some }\eta>0.
$$
\end{lem} 
Recall the space $\SPGa$ from Definition \ref{DefSolution}. See also Remark \ref{RemAbusingNotations}. 

Recall the parameter $\alpha$ in the definition of $\BV$ \eqref{EqnBV}. This parameter has been fixed in Lemma \ref{LemCommutator}. The parameter $\kappa$ is from \eqref{EqnKappa}.
\begin{proof}
The subsolution property follows from Lemma \ref{LemLowerLeafFar}, Lemma \ref{LemEquationInOLower} and Lemma \ref{LemAngleAlongGlue}. The ordering between $\Phi_1$ and $u$ follows from the definition of $U$ in \eqref{EqnLowerLeafFar} and Lemma \ref{LemOrderingLower}.

For the slice at height $y$, with \eqref{EqnYGraphLower} and \eqref{EqnHorizontalFillingLower}, we have
$$
V(x,y)=0 \text{ for }|x|\le R\sin(\theta)-\kappa^{1/16}R\BU^{\frac12}(\theta).
$$
With Lemma \ref{LemGlueAS}, we have $R\sin(\theta)-\kappa^{1/16}R\BU^{\frac12}(\theta)\ge c_dR\theta.$ Meanwhile, with \eqref{EqnRadialGraphLower}, we see 
$$
R\sim_d(\BV/\BU)^{\frac25}\ge (\BV/\BU)^{\frac25}(\pi/2).
$$
As a result, for each $\theta>\theta_0$, $\{\Phi_1=0\}$ contains a ball of radius $c_d\theta_0.$

Since $V\le u$ in $\R^{d+1}$, the entire cone $\{\theta\le\theta_0\}$ is contained in $\{\Phi_1=0\}$. 

Combining these with the even symmetry of $\Phi_1$, we get the $\{\Phi_1=0\}\supset B_\eta$ for a positive $\eta.$ 
\end{proof} 

Starting with $\Phi_1$ in \eqref{EqnSingleLowerLeaf}, we define
$$
\Phi_t(x,y):=t^2\Phi_1(x/t,y/t)
$$
for $t>0$. 

Below we show that this gives a lower foliation, and hence give the proof of Proposition \ref{PropLowerFoliationAS}.

\begin{proof}[Proof of Proposition \ref{PropLowerFoliationAS}]
With Lemma \ref{LemGenerateLowerFoliation} and the scaling symmetry of \eqref{EqnTheError}, we have
$
\Phi_t\in\SPGa(\R^{d+1}) \text{ for all }t>0.
$
Continuity of $(x,y,t)\mapsto\Phi_t(x,y)$ and continuity of $t\mapsto\partial\{\Phi_t>0\}$ follow from definition. 

\vem

With Lemma \ref{LemGenerateLowerFoliation}, we have
$$
\Phi_t(x,y)=t^2\Phi_1(x/t,y/t)\le t^2 u(x/t,y/t)=u \text{ in }\R^{d+1}.
$$
Moreover, we have
$$
\{\Phi_t=0\}\supset B_{t\eta},
$$
which engulfs any compact set when $t\to+\infty.$

Note that at each $(x,y)\in\{u>0\}$, we have
$$
(1-\kappa^\frac18)u(x/t,y/t)-v(x/t,y/t)=(1-\kappa^\frac18)t^{-2}u(x,y)-t^{1/2}v(x,y)>0
$$
for $t>0$ small. As a result, if $K$ is a compact subset of $\PosS$, then we can find $t_K$ such that 
$$
t\cdot\Omega\supset K
$$if $0<t<t_K.$
Here $\Omega$ is the domain away from the free boundary in \eqref{EqnAwayFromFBLower}.

As a result, on $K$ for $t<t_K$ we have
$$
\Phi_t(x,y)=t^2U(x/t,y/t)=u(x,y)-t^{5/2}v(x,y)\to u(x,y)
$$
uniformly as $t\to0.$

\vem

In summary, the family $\{\Phi_t\}$ gives a lower foliation of $u$.
\end{proof}

\section{Upper foliation of the axially symmetric cone}
\label{SectionUFAS}

In this section, we construct an upper foliation of the axially symmetric cone $\UAS$ in Proposition \ref{PropConstructionOfUAs}.  Together with the lower foliation from Proposition \ref{PropLowerFoliationAS}, this gives the minimality of $\UAS$, following Theorem \ref{ThmMinimality}.

The \textbf{main result} in this section is:
\begin{prop}
\label{PropUpperFoliationAS}
For $d\ge3$, there is a dimensional constant $\gamma_d^3\in(0,1)$ such that for 
$$
\gamma_d^3<\gamma<1,
$$
there is an upper foliation $\{\Psi_t\}_{t\in(0,+\infty)}$ of the cone $\UAS$ from Proposition  \ref{PropConstructionOfUAs} in $\R^{d+1}$.
\end{prop} 
Recall the definition of an upper foliation in Definition \ref{DefUpperFoliation}. As in Section \ref{SectionCAS} and Section \ref{SectionLowerFoliationAS}, we work with the transformed equation \eqref{EqnTransformedEquation}. See Proposition \ref{PropEquivalenceOfEquations} and Remark \ref{RemAbusingNotations}.

\vem
We briefly explain the \textbf{strategy of the construction}.

Similar to Section \ref{SectionLowerFoliationAS}, leaves in this upper foliation are rescalings of the single leaf in \eqref{EqnUpperLeaf}. Most of this section is devoted to the construction of this leaf. In contrast to the lower leaves, which stay below the cone $\UAS$, the upper leaves are larger than $\UAS$. This requires a topological change in the positive set near the origin. 

We first study the linearized equation in Subsection \ref{SubsectionLEACUpper}, and construct a positive supersolution $v$ in \eqref{EqnV2}. Compared with its counterpart in Subsection \ref{SubsectionLEAU}, here the construction is more involved for reasons we outline in Remark \ref{RemWhyDifferentTheta1}.

With this supersolution $v$, we begin the construction of a leaf in the upper foliation. 

In Subsection \ref{SubsectionRegionNearInfinityUpper}, we start with the region near infinity in \eqref{EqnRegionNearInfinity}. Here the strategy is similar to the construction in Section \ref{SectionLowerFoliationAS}. Away from the free boundary (see \eqref{EqnAwayFromFBUpper}), the profile $U$ is given by a perturbation of the cone by $v$ (see \eqref{EqnPerturbedSolution}). Its supersolution property is given in Lemma \ref{LemSupersolutionNonlinear}. Then we extend $U$ to the region near the free boundary in horizontal slices as $V$ in  \eqref{EqnHorizontalFillingUpper}. Compared to Subsection \ref{SubsectionLowerNearFreeBoundary}, 
 the preferred sign has been reversed in several key estimates, leading to new technical challenges (see Lemma \ref{LemNonlinearEquationInInterior}, Lemma \ref{LemEqnAlongGammaUpper}, and Lemma \ref{LemOuterRegionOrdering}). 

Gluing $U$ and $V$, we get the profile near infinity $\UOut$ in \eqref{EqnProfileNearInfinityUpper}.

In Subsection \ref{SubsectionRegionNearTheOriginUpper}, we turn to the region near the origin in \eqref{EqnRegionNearTheOrigin}. Here the topological change of the positive set happens. In this region, the profile, $\UIn$, is given in \eqref{EqnProfileNearOriginUpper}. 

Finally, in Subsection \ref{SubsectionUpperFoliationOfASC}, we glue $\UOut$ and $\UIn$ by a radial harmonic function to get the glued profile $W$ in \eqref{EqnW}, which inherits the supersolution property of its components (see Lemma \ref{LemW}). To get a bounded function, we perform a truncation, leading to the leaf $\Psi_1$ in \eqref{EqnUpperLeaf}. We then give the proof of Proposition \ref{PropUpperFoliationAS} and Theorem \ref{ThmMainAS}.

\vem

Recall that $\gamma\in(0,1)$ is close to one. We work in $\R^{d+1}$ with $d\ge3$ as in \eqref{DefSpaceDecomposition}, and use the coordinate system in  \eqref{EqnCoordinate}. The cone $\UAS$ is denoted by $u$ and its trace by $\BU$ as in  \eqref{EqnTraceOfUAS} and \eqref{EqnPositiveSets}.

\subsection{Linearized equation around the cone} 
\label{SubsectionLEACUpper}
We need a positive supersolution to the linearized operator $\LU$ in \eqref{EqnLU}. Similar to the situation in Subsection \ref{SubsectionLEAU}, this supersolution takes the form 
\begin{equation}
\label{EqnV2}
v(r,\theta):=r^{-1/2}\BV(\theta)
\end{equation}
in our coordinate system \eqref{EqnCoordinate}.

With the angle $\theta_0$ from Lemma \ref{LemGlueAS}, 
we take 
\begin{equation}
\label{EqnTheta12}
\theta_1:=(1+2\delta)\theta_0
\end{equation} 
for a small $\delta>0$ to be chosen, depending only on the dimension. 
The function  $\BV$ is defined as\footnote{We focus on the half-interval $(0,\pi/2]$. The situation in $[\pi/2,\pi)$ follows by even symmetry. See Remark \ref{RemHalfIntervalAndViscosity}.}
\begin{equation}
\label{EqnBV2}
\BV:=\begin{cases}
\sigma(\theta-\theta_0)^\delta &\text{ on }(\theta_0,\theta_1],\\
\sin^{-1/2}(\theta) &\text{ on }(\theta_1,\pi/2],
\end{cases}
\end{equation} 
where the coefficient $\sigma$ satisfies
$$
\sigma(\theta_1-\theta_0)^\delta=\sin^{-1/2}(\theta_1).
$$

\begin{rem}
\label{RemWhyDifferentTheta1}
In Subsection \ref{SubsectionLEAU}, the function  $\BV$ in \eqref{EqnBV} makes a concave angle at $(1+\kappa^\frac18)\theta_0$. At this point, the function  $\BU^{1/2}$ is dominated by $\sin(\theta)$ (see \eqref{EqnSineDominatesAtTheta1}), which leads to the preferable sign in \eqref{Delicate}. In this section, the opposite sign is preferred. As a result, we  need to delay the sharp angle in $\BV$ to $\theta_1$ in \eqref{EqnTheta12}. See, for instance, \eqref{EqnAvoidCorner} and \eqref{EqnAbsValueDDYV}.

The exponent in \eqref{EqnBV2} is also modified to $\delta$, the same parameter in \eqref{EqnTheta12}. This is to achieve a different ordering between the derivatives of the profiles (compare Lemma \ref{LemAngleAlongGlue} and Lemma \ref{LemEqnAlongGammaUpper}, and see in particular \eqref{EqnUsingTheExponent}).
\end{rem}

With the same proof as for Lemma \ref{LemCommutator}, we still have
\begin{lem}
\label{LemCommutatorUpper}
For  $d\ge3$ and dimensionally small $\delta$ and $\kappa$,  we have 
$$
|\BV'/\BV|\le\BO((\theta-\theta_0)^{-1}\cos(\theta)),\hem (\frac{\BU'}{\BU}-\frac{\BV'}{\BV})\sim_d(\theta-\theta_0)^{-1}\cos(\theta) \text{ on }(\theta_0,\pi/2].
$$
In particular, we have
$$
(\BU/\BV)'\ge 0  \text{ on }(\theta_0,\pi/2].
$$
Moreover, we have
$$
 |(\BV'/\BV)'|\le\BO((\theta-\theta_0)^{-2})  \text{ on }(\theta_0,\theta_1)\cup(\theta_1,\pi/2).
 $$
\end{lem} 
Recall the small parameter $\kappa$ from \eqref{EqnKappa}. The parameter $\delta$ appears in \eqref{EqnTheta12} and \eqref{EqnBV2}.

Similar to Lemma \ref{LemSupersolutionLU}, the function in \eqref{EqnV2} is a supersolution to the linearized operator $\LU$ in \eqref{EqnLU}:
\begin{lem}
\label{LemLinearizedOperatorUpper}
Under the same assumptions as in Lemma \ref{LemCommutatorUpper}, we have
$$
\LU(v)+\kappa^{\frac{1}{2}}v/u\le 0 \text{ in }\PosS.
$$ 
\end{lem} 

\begin{proof}
Recall the linearized operator $\LBU$  from \eqref{EqnLBU}. With similar ideas as in the proof of Lemma \ref{LemSupersolutionLU}, we have
$$
\LBU(\BV)/\BV+\kappa^{\frac{1}{2}}\BU^{-1} \le(\theta-\theta_0)^{-2}\{\delta(\delta-1+\BO(\delta))+\BO(\kappa^{\frac{1}{2}})\} \text{ on }(\theta_0,\theta_1).
$$
For small $\delta$, the right-hand side is negative for small $\kappa>0$.

Similarly, on $(\theta_1,\pi/2)$ we have
$$
\LBU(\BV)/\BV+\kappa^{\frac{1}{2}}\BU^{-1} \le\frac12+[-\frac{d-1}{2}+\frac34+\BO_{d,\delta}(\kappa^{\frac{1}{2}})]\sin^{-2}(\theta)\cos^2(\theta).
$$
For small $\kappa$, the right-hand side is less than $\frac12$.

With the concave angle of $\BV$ at $\theta=\theta_1$, we have the desired supersolution property in the viscosity sense. Combining these, we get
$$
\LBU(\BV)/\BV+\kappa^{\frac{1}{2}}\BU^{-1} \le\frac12 \text{ on }(\theta_0,\pi/2).
$$

With the splitting of the operator $\LU$ in \eqref{EqnSplittingLU} and the definition of $v$ in \eqref{EqnV2}, we have
$$
\LU(v)+\kappa^{\frac{1}{2}}v/u\le r^{-5/2}\BV\cdot\{-\frac{d}{2}+\frac54+\frac32\kappa\}\le0
$$
for small $\kappa$.
\end{proof} 

Recall the transformed Euler-Lagrange equation in \eqref{EqnTransformedEquation}, a natural candidate for a supersolution is
\begin{equation}
\label{EqnPerturbedSolution}
U:=u+v,
\end{equation}
where $u$ denotes the cone in Proposition \ref{PropConstructionOfUAs}, and $v$ is the profile in \eqref{EqnV2}.

With the error $E$ defined in \eqref{EqnTheError}, we have
\begin{lem}
\label{LemSupersolutionNonlinear}
Under the same assumptions as in Lemma \ref{LemCommutatorUpper}, we have
$$
E(U)\le 0 \text{ in }\PosS.
$$
\end{lem} 
\begin{proof}
We have
\begin{equation}
\label{EqnAnotherEstimateForNonlinear}
E(U)=\LU(v)+(\frac{\beta}{2}-1)[\frac{|\nabla(u+v)|^2}{u+v}-\frac{|\nabla u|^2}{u}-\frac{2\nabla u}{u}\cdot\nabla v+\frac{|\nabla u|^2}{u^2}v].
\end{equation}
We estimate the nonlinearity in two regions, namely,  
$$\Omega_1:=\{u\ge v\} \text{ and }\Omega_2:=\{0<u<v\}.$$

In $\Omega_1$,  the same idea in the proof of Lemma \ref{LemLowerLeafFar} gives
$$
|\frac{|\nabla(u+v)|^2}{u+v}-\frac{|\nabla u|^2}{u}-\frac{2\nabla u}{u}\cdot\nabla v+\frac{|\nabla u|^2}{u^2}v|\le \BO(v/u).
$$

In $\Omega_2$, we have $u<v$, leading to  $$|\nabla u|^2/u<|\nabla u|^2v/u^2.$$ As a result, the negative part of the nonlinearity satisfies
$$
[\frac{|\nabla(u+v)|^2}{u+v}-\frac{|\nabla u|^2}{u}-\frac{2\nabla u}{u}\cdot\nabla v+\frac{|\nabla u|^2}{u^2}v]_-\le [\frac{2\nabla u}{u}\cdot\nabla v]_+ \text{ in }\Omega_2.
$$
Since the radial derivatives of $u$ and $v$ have opposite signs, the right-hand side can be bounded by $|2\nabla_\tau u\cdot\nabla_\tau v/u|=|2r^{-5/2}\frac{\BU'\BV'}{\BU}|$. 

Note that $|2r^{-5/2}\frac{\BU'\BV'}{\BU}|\cdot \frac{u}{v}=2\BU'|\frac{\BV'}{\BV}|$. With the bound on $\BU'$ from Lemma \ref{LemGlueAS} and the bound on $\BV'/\BV$ in Lemma \ref{LemCommutatorUpper}, this quantity is bounded by a dimensional constant. As a result, we have
$$
[\frac{|\nabla(u+v)|^2}{u+v}-\frac{|\nabla u|^2}{u}-\frac{2\nabla u}{u}\cdot\nabla v+\frac{|\nabla u|^2}{u^2}v]_-\le\BO(v/u) \text{ in }\Omega_2.
$$

Combining these with $\frac{\beta}{2}-1<0$ for $\gamma<1$, the nonlinearity in \eqref{EqnAnotherEstimateForNonlinear} satisfies
$$
(\frac{\beta}{2}-1)[\frac{|\nabla(u+v)|^2}{u+v}-\frac{|\nabla u|^2}{u}-\frac{2\nabla u}{u}\cdot\nabla v+\frac{|\nabla u|^2}{u^2}v]\le \kappa\BO(v/u) \text{ in }\PosS.
$$
With Lemma \ref{LemLinearizedOperatorUpper}, we conclude
$$
E(U)\le\frac{v}{u}[-\kappa^{\frac{1}{2}}+\BO(\kappa)]\le0 \text{ in }\PosS
$$
for $\kappa$ small. 
\end{proof}

\subsection{Region near infinity}
\label{SubsectionRegionNearInfinityUpper}
We begin our construction of an upper leaf. In this subsection, we focus on the \textit{region near infinity}, namely,  
\begin{equation}
\label{EqnRegionNearInfinity}
\Omega_{\mathrm{out}}:=\R^{d+1}\backslash B_{R_{\mathrm{out}}},
\end{equation}
where $\ROut$ is a large constant,  chosen in \eqref{EqnROut}. This choice helps avoid complications related to the sharp angle in $\BV$. See Remark \ref{RemWhyDifferentTheta1} and \eqref{EqnAvoidCorner}.

In $\Omega_{\mathrm{out}}$, the construction follows the same strategy  in Section \ref{SectionLowerFoliationAS}. We use the profile $U$ from \eqref{EqnPerturbedSolution} in the region away from the free boundary, namely, the region in \eqref{EqnAwayFromFBUpper}. We then extend  it as $V$ in \eqref{EqnHorizontalFillingUpper} to the region near the free boundary. Compared with Section \ref{SectionLowerFoliationAS}, the main differences lie  in Lemma \ref{LemNonlinearEquationInInterior} and Lemma \ref{LemEqnAlongGammaUpper}, where the preferred sign has been reversed (see, in particular, \eqref{EqnAbsValueDDYV} and \eqref{EqnUsingTheExponent}).

\vem

To be precise, the \textit{region away from the free boundary} is defined as
\begin{equation}
\label{EqnAwayFromFBUpper}
\Omega:=\{u>\delta^{3/2}v\}\backslash \overline{B}_{\ROut}.
\end{equation} 
Recall the parameter $\delta$ from \eqref{EqnTheta12} and \eqref{EqnBV2}. 

\begin{rem}
The lower bound of $u/v$ in $\Omega$ is chosen to satisfy two competing requirements. 

On the one hand,  for our profile to stay above $u$, we need this bound to be small (see Lemma \ref{LemOuterRegionOrdering}). On the other hand, to have the desired supersolution property, we need $(u/v)$ to be much larger than $\delta^2$ (see, for instance,  Lemma \ref{LemNonlinearEquationInInterior} and Lemma \ref{LemEqnAlongGammaUpper}).
\end{rem}

In this region, the leaf is taken as $U$ in \eqref{EqnPerturbedSolution}. The desired supersolution property follows from Lemma \ref{LemSupersolutionNonlinear}.

\vem

To extend the profile $U$ beyond $\Omega$, we need information on its trace along the boundary of this region, namely,
\begin{equation}
\label{EqnGammaUpper}
\Gamma:=\partial\Omega\backslash \overline{B}_{\ROut}.
\end{equation} 
Along $\Gamma$, we have 
\begin{equation}
\label{EqnConditionAlongGammaUpper}
u=\delta^{3/2}v, \text{ and }U=(1+\delta^{-3/2})u.
\end{equation}
In particular, when $\delta$ is small, $U/u\gg 1$ on $\Gamma$. This is the reason behind the ordering in Lemma \ref{LemOuterRegionOrdering}.

Similar to Subsection \ref{SubsectionLowerAwayFromFreeBoundary}, using the coordinate in \eqref{EqnCoordinate}, this interface $\Gamma$ is described by a graph $\theta\mapsto R(\theta)$, satisfying the relation
\begin{equation}
\label{EqnRadialGraphUpper}
R^{5/2}(\theta)=\delta^{3/2}\frac{\BV}{\BU}(\theta).
\end{equation} 
If we choose $\ROut$ as
\begin{equation}
\label{EqnROut}
\ROut^{5/2}:=2\delta^{3/2}\frac{\BV}{\BU}((1+\delta)\theta_0),
\end{equation} 
then the monotonicity of $\BV/\BU$ from Lemma \ref{LemCommutatorUpper} implies
$$
\delta^{3/2}\frac{\BV}{\BU}(\theta)\le \delta^{3/2}\frac{\BV}{\BU}((1+\delta)\theta_0)<\ROut^{5/2}\hem \text{ for }\theta\in[(1+\delta)\theta_0,\pi/2].
$$
With \eqref{EqnRadialGraphUpper}, we see that 
\begin{equation}
\label{EqnAvoidCorner}
\Gamma\subset\{\theta\in(\theta_0,(1+\delta)\theta_0)\}.
\end{equation}
In particular, when doing analysis on $\Gamma$, we do not see the sharp angle in the graph of $\BV$ in \eqref{EqnBV2}.

We remark that our choice of $\ROut$ in \eqref{EqnROut} gives
\begin{equation}
\label{EqnROutSim}
\ROut^{5/2}\sim_d\delta^{-1/2}\theta_0^{-5/2}.
\end{equation}

Similar to Subsection \ref{SubsectionLowerAwayFromFreeBoundary}, the interface is also described by $y\mapsto L(y)$ where 
\begin{equation}
\label{EqnYGraphUpper}
y=-R(\theta)\cos(\theta), \text{ and }L(y)=R(\theta)\sin(\theta).
\end{equation} 

\vem

Similar to Lemma \ref{LemTechnical1}, the two descriptions in \eqref{EqnRadialGraphUpper} and \eqref{EqnYGraphUpper} satisfy the following bounds. With \eqref{EqnAvoidCorner}, we only need information on $(\theta_0,(1+\delta)\theta_0)$. For this range, the proof becomes  elementary and  is omitted.

\begin{lem}
\label{LemTechnical3}
The quantities in \eqref{EqnRadialGraphUpper} and \eqref{EqnYGraphUpper} satisfy, on $(\theta_0,(1+\delta)\theta_0)$,  
$$
0\le R\frac{d\theta}{dy}\le \BO((\theta-\theta_0)), \hem |\frac{dL}{dy}|\sim_d R\frac{d\theta}{dy}\cdot \theta(\theta-\theta_0)^{-1},
$$
and
$$
 |\frac{d^2\theta}{dy^2}|\le (\frac{d\theta}{dy})^2\BO((\theta-\theta_0)^{-1}), \hem |\frac{d^2L}{dy^2}|\le R(\frac{d\theta}{dy})^2 \BO(\theta(\theta-\theta_0)^{-2}).
$$
\end{lem} 

Similar to Lemma \ref{LemTechnical2}, we have the following bound on  $U|_\Gamma$. In the range $(\theta_0,(1+\delta)\theta_0)$, the proof is elementary and thus omitted:
\begin{lem}
\label{LemTechnical4}
The function $y\mapsto\sqrt{U}(L(y),y)$ satisfies, on $(\theta_0,(1+\delta)\theta_0)$, 
$$
|\frac{d\sqrt{U}}{dy}|\le \delta^{-3/4}R\frac{d\theta}{dy}\BO(1),
\text{ and }
|\frac{d^2\sqrt{U}}{dy^2}|=\delta^{-3/4}R(\frac{d\theta}{dy})^2\BO((\theta-\theta_0)^{-1}).
$$
\end{lem} 

\vem

With these preparations, we extend the profile $U$ beyond $\Omega$. For the horizontal slice at height $y$, the extended profile is defined  as
\begin{equation}
\label{EqnHorizontalFillingUpper}
V(x,y):=\frac{1}{20}[(|x|-L(y)+\sqrt{20}\sqrt{U}(L(y),y))_+]^2 \text{ for }|x|\le L(y).
\end{equation} 
Recall the map $y\mapsto L(y)$ in \eqref{EqnYGraphUpper}.

We first verify the supersolution property in the interior of the complement of $\Omega$, namely
\begin{align}
\label{EqnCOCO}
\CO:&=\Omega_{\mathrm{out}}\backslash\Omega\\
&=\{(x,y):|x|<L(y)\}\backslash \overline{B}_{\ROut}, \nonumber
\end{align}
where $\Omega_{\mathrm{out}}$ is given in \eqref{EqnRegionNearInfinity}, and $\Omega$ is from \eqref{EqnAwayFromFBUpper}.

\begin{lem}
\label{LemNonlinearEquationInInterior}
For $d\ge3$ and dimensionally small $\delta,\kappa$, we have
$$
\Delta V\le 1/5 \text{ in }\CO.
$$
\end{lem} 
Recall the small parameter $\kappa$ in \eqref{EqnKappa}, and the parameter $\delta$ appearing in \eqref{EqnTheta12}, \eqref{EqnBV2} and \eqref{EqnAwayFromFBUpper}.
\begin{proof}
By definition, we have $\nabla V=0$ along $\partial\{V>0\}$.  Thus it suffices to verify the inequality  in 
\begin{align*}
\tilde\CO:&=\CO\cap\{V>0\}\\
&\subset\{(x,y): L(y)-\sqrt{U}(L(y),y)<|x|<L(y)\}.
\end{align*}
Recall the map $y\mapsto L(y)$ in \eqref{EqnYGraphUpper}.

In $\tilde\CO$, direct computation gives
\begin{align}
\label{EqnDeltaV}
\Delta V&\le \frac{1}{10}+\frac{d-1}{10}\frac{|x|-L+\sqrt{20}\sqrt{U}}{|x|}+\frac{\partial^2}{\partial y^2}V\\
&\le\frac{1}{10}+\frac{d-1}{10}\frac{\sqrt{20}\sqrt{U}}{L-\sqrt{20}\sqrt{U}}+\frac{\partial^2}{\partial y^2}V. \nonumber
\end{align}

Recall from \eqref{EqnConditionAlongGammaUpper},  \eqref{EqnYGraphUpper} and \eqref{EqnAvoidCorner}, we have $\sqrt{U}=\sqrt{1+\delta^{-3/2}}R\BU^{1/2}$, $L=R\sin(\theta)$ and $\theta\in(\theta_0,(1+\delta)\theta_0)$. With the bound on $\BU$ from Lemma \ref{LemGlueAS}, these imply
$$
\sqrt{U}/L\le\delta^{-3/4}\BO(\frac{\theta-\theta_0}{\theta})\le\BO(\delta^{1/4}).
$$
Thus we can continue the estimate in \eqref{EqnDeltaV} as
\begin{equation}
\label{EqnDeltaV2}
\Delta V\le\frac{1}{10}+\frac{\partial^2}{\partial y^2}V+\BO(\delta^{1/4}) \hem\text{ in }\tilde\CO.
\end{equation} 

It remains to bound $\frac{\partial^2}{\partial y^2}V$.

With the definition of $V$ in \eqref{EqnHorizontalFillingUpper}, the chain rule gives
\begin{equation*}
|\frac{\partial^2}{\partial y^2}V|\le \frac{1}{10}|\frac{d}{dy}(-L+\sqrt{U})|^2+\frac{1}{10}|\sqrt{U}||\frac{d^2}{dy^2}(-L+\sqrt{U})|.
\end{equation*}
With Lemma \ref{LemTechnical3} and Lemma \ref{LemTechnical4}, both terms on the right-hand side can be bounded by $\BO(\theta_0^2)$. As a result, we have
\begin{equation}
\label{EqnAbsValueDDYV}
|\frac{\partial^2}{\partial y^2}V|\le\BO(\kappa^{\frac{2}{d-2}})
\end{equation}
with the bound on $\theta_0$ from Lemma \ref{LemGlueAS}.

Combining this with \eqref{EqnDeltaV2}, we get the desired conclusion for  small $\delta$ and $\kappa$. 
\end{proof} 

Similar to Lemma \ref{LemAngleAlongGlue}, it is crucial to have the following ordering between the derivatives of $U$ in \eqref{EqnPerturbedSolution} and $V$ in \eqref{EqnHorizontalFillingUpper}:
\begin{lem}
\label{LemEqnAlongGammaUpper}
Under the same assumptions as in Lemma \ref{LemNonlinearEquationInInterior}, we have
$$
\frac{\partial}{\partial |x|}V>\frac{\partial}{\partial |x|}U \text{ along }\Gamma.
$$
\end{lem} 
\begin{proof}
The definition of $V$ in \eqref{EqnHorizontalFillingUpper} gives
\begin{equation}
\label{FirstEquationInAProof}
\frac{\partial}{\partial |x|}V=\frac{\sqrt{20}\sqrt{U}}{10} \text{ along }\Gamma.
\end{equation}

On the other hand,  we have
\begin{align*}
\frac{\partial}{\partial |x|}U&=R^{-1}(2u-\frac12v)\sin(\theta)+(R\BU'+R^{-3/2}\BV')\cos(\theta)\\
& =R^{-1}u(2-\delta^{-3/2}/2)\sin(\theta)+(R\BU'+R^{-3/2}\BV')\cos(\theta)\\
&\le (R\BU'+R^{-3/2}\BV') \hem\text{ along }\Gamma. 
\end{align*}
Note that we also used \eqref{EqnRadialGraphUpper}.

With the definition of $R(\theta)$ in  \eqref{EqnRadialGraphUpper} and the bound on $\BU$ in  Lemma \ref{LemGlueAS}, we further estimate
\begin{align*}
\frac{\partial U/\partial |x|}{\sqrt{U}}&\le \frac{\BU'}{\sqrt{1+\delta^{-3/2}}\BU^{1/2}}+\frac{\delta^{-3/2}\BU\BV'}{\sqrt{1+\delta^{-3/2}}\BU^{1/2}\BV}\\
&\le\BO(\delta^{3/4})+\frac{\BV'}{\BV}(\theta-\theta_0)\BO(\delta^{-3/4}).
\end{align*}
With the definition of $\BV$ in \eqref{EqnBV2} for the range of $\theta$  in \eqref{EqnAvoidCorner}, we conclude
\begin{equation}
\label{EqnUsingTheExponent}
\frac{\partial U/\partial |x|}{\sqrt{U}}\le\BO(\delta^{3/4})+\delta\BO(\delta^{-3/4}) \text{ along }\Gamma.
\end{equation}

Comparing with \eqref{FirstEquationInAProof}, the desired comparison follows for small $\delta>0$.
\end{proof} 

The extension $V$ in \eqref{EqnHorizontalFillingUpper} stays on top of the cone:
\begin{lem}
\label{LemOuterRegionOrdering}
Under the same assumptions as in Lemma \ref{LemNonlinearEquationInInterior}, we have
$$
V\ge u \text{ in }\CO.
$$
\end{lem} 
Recall the domain $\CO$ from \eqref{EqnCOCO}.
\begin{proof}
We show the comparison in the horizontal slice for a fixed $\bar{y}$, namely, 
$$
\{(x,\bar{y}): L(\bar y)-\sqrt{U}(L(\bar y),\bar y)<|x|<L(\bar{y})\}.
$$
Recall the map $y\mapsto L(y)$ in \eqref{EqnYGraphUpper}. In our coordinate system \eqref{EqnCoordinate}, we have
$$
L(\bar y)=\bar r\sin(\bar\theta)
$$
form some $\bar r>0$ and $\bar\theta\in(0,\pi/2].$

For simplicity, denote by $L_0>0$ the constant such that 
\begin{equation}
\label{EqnL0}
\PosS\cap\{y=\bar y\}=\{|x|>L_0, y=\bar y\}.
\end{equation}
In our coordinate system \eqref{EqnCoordinate}, we have, for some $r_0>0$,
$$
L_0=r_0\sin(\theta_0),
$$
where $\theta_0$ is the angle from Lemma \ref{LemGlueAS}. 

It is elementary to see
\begin{equation}
\label{EqnLInLOut}
L(\bar y)-L_0\le \BO(\bar r(\bar\theta-\theta_0))\le \BO(\bar r\sqrt{\BU}(L(\bar y),\bar y))\le\BO(\sqrt{u}(L(\bar y),\bar y)).
\end{equation} 
Note that we used  \eqref{EqnTraceOfUAS} and Lemma \ref{LemGlueAS}.

\vem

Based on the definition of $V$ in \eqref{EqnHorizontalFillingUpper}, we have
$
\frac{\partial}{\partial |x|}V=\frac{\sqrt{20}}{10}\sqrt{V}.
$
Integrate this relation from $|x|=L_0$ to $|x|=L(\bar y)$, we get
$$
\sqrt{V}(L_0,\bar y)=\sqrt{V}(L(\bar y),\bar y)-\frac{\sqrt{20}}{20}(L(\bar y)-L_0).
$$
With \eqref{EqnHorizontalFillingUpper} and \eqref{EqnLInLOut}, this implies
$$
\sqrt{V}(L_0,\bar y)\ge \sqrt{U}(L(\bar y),\bar y)-\BO(\sqrt{u}(L(\bar y),\bar y))\ge (\delta^{-\frac34}-c_d)\sqrt{u}(L(\bar y),\bar y).
$$
With $\frac{\partial}{\partial |x|}V\ge0$, we get
\begin{equation}
\label{EqnVOneSlice}
V(|x|,\bar y)\ge \frac12\delta^{-\frac32}u(L(\bar y),\bar y)\hem \text{ for }L_0\le |x|\le L(\bar y)
\end{equation}
if $\delta$ is dimensionally small. 

\vem

In this slice, the homogeneity of $u$ and the monotonicity of $\BU$ in Lemma \ref{LemGlueAS} give
$$
u\le u(L(\bar{y}),\bar{y}).
$$
With \eqref{EqnVOneSlice}, we get
$$
V(|x|,\bar y)\ge u(|x|,\bar y)\hem \text{ for }L_0\le |x|\le L(\bar y)
$$
for small $\delta>0$.

With $u(|x|,\bar y)=0$ for $|x|\le L_0$ (see \eqref{EqnL0}), we get the desired ordering in the entire horizontal slice.
\end{proof} 

Combining $U$  from \eqref{EqnPerturbedSolution} and  $V$ from \eqref{EqnHorizontalFillingUpper},
we define the \textit{profile near infinity} as
\begin{equation}
\label{EqnProfileNearInfinityUpper}
\UOut:=\begin{cases}
U &\text{ in }\overline{\Omega},\\
V &\text{ in }\CO
\end{cases}
\end{equation} 
where $\Omega$ is given in \eqref{EqnAwayFromFBUpper} and $\CO$ is from \eqref{EqnCOCO}. 

We have established the following:
\begin{lem}
\label{LemPropertyOfUOut}
For $d\ge3$ and dimensionally small $\delta,\kappa$, we have
$$
\UOut\ge u, \text{ and } E(\UOut)\le 0 \text{ in }\R^{d+1}\backslash\overline{B}_{\ROut}.
$$
\end{lem} 




\subsection{Region near the origin}
\label{SubsectionRegionNearTheOriginUpper}
In this subsection, we construct the leaf in the \textit{region near the origin}, that is, 
\begin{equation}
\label{EqnRegionNearTheOrigin}
\Omega_{\mathrm{in}}:=B_{\RIn}
\end{equation} 
for the large $\RIn$ chosen in Lemma \ref{LemConcaveAngleInConstantExtension}. 

In the coordinate system \eqref{EqnCoordinate}, we start with the profile $U$ in \eqref{EqnPerturbedSolution} and make it constant for small angles. To be precise, we define 
\begin{equation}
\label{EqnProfileNearOriginUpper}
\UIn(r,\theta):=\begin{cases}
U(r,\theta) &\text{ for }\theta\in(\theta_1,\pi/2],\\
U(r,\theta_1) &\text{ for }\theta\in[0,\theta_1].
\end{cases}
\end{equation} 
Recall that $\theta_1=(1+2\delta)\theta_0$ from \eqref{EqnTheta12}.

By choosing the proper $\RIn$, we have
\begin{lem}
\label{LemConcaveAngleInConstantExtension}
For $d\ge3$ and $\delta,\kappa$ dimensionally small, there is a large dimensional constant $M_d$ such that for  
$$
\RIn^{5/2}=M_d\delta^{-1}\theta_0^{-5/2},
$$
we have
$$
\frac{\partial}{\partial\theta}U(r,\theta_1)<0 \text{ for }r\le\RIn.
$$
\end{lem} 
Recall the small parameter $\kappa$ from \eqref{EqnKappa}, and the parameter $\delta$ appearing in \eqref{EqnTheta12}, \eqref{EqnBV2} and \eqref{EqnAwayFromFBUpper}.

\begin{proof}
With the definition of $U$ in \eqref{EqnPerturbedSolution}, the definition of $v$ in \eqref{EqnV2} and the bound on $\BU'$ in Lemma \ref{LemGlueAS}, we have
$$
\frac{\partial}{\partial\theta}U(r,\theta_1)\le C_dr^{-1/2}\delta\theta_0[r^{5/2}-c_d\delta^{-1}\theta_0^{-5/2}]
$$
for dimensional constants $C_d$ and $c_d$. 
\end{proof} 

For $\theta<\theta_1$ from \eqref{EqnTheta12}, the supersolution property is given by the following lemma:
\begin{lem}
\label{LemSuperSolutionInInteriorInner}
Under the same assumptions as in Lemma \ref{LemConcaveAngleInConstantExtension}, we have
$$
\Delta\UIn\le 0 \text{ in }B_{\RIn}\cap\{0\le\theta<\theta_1\}.
$$
\end{lem} 
\begin{proof}
Direct computation gives
$$
\Delta\UIn=-(\frac{d}{2}-\frac34)r^{-5/2}\BV(\theta_1)+2(d+1)\BU(\theta_1)\le -c_d[r^{-5/2}\theta_0^{-1/2}-C_d(\delta\theta_0)^2],
$$
where we used the definition of $\BV$ in  \eqref{EqnBV2} and the bound on $\BU$ in  Lemma \ref{LemGlueAS}.

For the choice of $\RIn$  in Lemma \ref{LemConcaveAngleInConstantExtension},  we can further estimate 
$$
\Delta\UIn\le -c_d[\delta\theta_0^{2}-C_d\delta^2\theta_0^2]\le 0 \text{ in }B_{\RIn}
$$
if $\delta$ is small. 
\end{proof} 

With the monotonicity of $\BU$  in Lemma \ref{LemGlueAS}, we see that
\begin{lem}
\label{LemOrderingInnerRegionUpper}
Under the same assumption as in Lemma \ref{LemSuperSolutionInInteriorInner}, we have
$$
\UIn\ge u \text{ in }B_{\RIn}.
$$
\end{lem}

\subsection{Upper foliation of the axially symmetric cone}
\label{SubsectionUpperFoliationOfASC}
We have constructed the profile $\UOut$  from \eqref{EqnProfileNearInfinityUpper} in the region near infinity $\Omega_{\mathrm{out}}$ in \eqref{EqnRegionNearInfinity}, and a profile $\UIn$   from \eqref{EqnProfileNearOriginUpper} in the region near the origin $\Omega_{\mathrm{in}}$ in \eqref{EqnRegionNearTheOrigin}. In this subsection, we glue them to form a leaf in the lower foliation. 

From our choices of $\ROut$ in  \eqref{EqnROutSim} and $\RIn$ in Lemma \ref{LemConcaveAngleInConstantExtension}, we have
\begin{equation}
\label{EqnRatioInOut}
(\ROut/\RIn)^{5/2}\sim_d\delta^{1/2}, \text{ and }\ROut,\RIn\sim_{d,\delta}\theta_0^{-1}.
\end{equation}
For $\delta>0$ small, the two regions, $\R^{d+1}\backslash\overline{B}_{\ROut}$ and $B_{\RIn}$, overlap in an annulus  whose diameter is comparable to $\theta_0^{-1}$, namely,
\begin{equation}
\label{EqnGluingAnnulus}
\mathcal{A}:=B_{\RIn}\backslash B_{\ROut}.
\end{equation} 
The gluing takes place in this annulus. 

To be precise, let $\zeta$ denote the continuous radial function of the form 
$$
\zeta(r,\theta)=Ar^{1-d}+B \text{ for }\ROut<r<\RIn,
$$
and
$$
\zeta(r,\theta)=1 \text{ for }r\le\ROut,\hem\text{ and }\hem \zeta(r,\theta)=0 \text{ for }r\ge\RIn.
$$
With \eqref{EqnRatioInOut}, the coefficient $A$ satisfies
\begin{equation}
A\ROut^{1-d}\sim_d 1.
\end{equation} 

\vem

With $\UIn$ from \eqref{EqnProfileNearOriginUpper} and $\UOut$ from \eqref{EqnProfileNearInfinityUpper}, the \textit{glued profile} is defined as
\begin{equation}
\label{EqnW}
W:=\zeta\UIn+(1-\zeta)\UOut.
\end{equation} 

This gives a supersolution to the transformed Euler-Lagrange equation \eqref{EqnTransformedEquation}. See also Proposition \ref{PropEquivalenceOfEquations}.
\begin{lem}
\label{LemW}
For $d\ge3$ and $\delta,\kappa$ dimensionally small, we have
$$
E(W)\le 0 \text{ in }\R^{d+1}\backslash\{0\}.
$$
\end{lem} 
Recall the error $E$ defined in \eqref{EqnTheError}.

\begin{proof}
With \eqref{EqnAvoidCorner}, \eqref{EqnProfileNearInfinityUpper} and \eqref{EqnProfileNearOriginUpper}, we see that
$$
\UOut=\UIn=U \text{ for }\theta\in(\theta_1,\pi/2].
$$
For this range of $\theta$, the desired conclusion follows from Lemma \ref{LemSupersolutionNonlinear}.

For $\theta=\theta_1$, the desired conclusion follows from Lemma \ref{LemConcaveAngleInConstantExtension} for $\zeta>0$, and follows from Lemma \ref{LemPropertyOfUOut} when $\zeta=0.$

As a result, it suffices to consider $\theta\in[0,\theta_1)$. For this range, when $\zeta=1$, we have $W=\UIn$. The desired conclusion follows from Lemma \ref{LemSuperSolutionInInteriorInner}. When $\zeta=0$, we have $W=\UOut$. The desired conclusion follows from Lemma \ref{LemPropertyOfUOut}.

It remains to consider the region
$$
\tilde{\mathcal{A}}:=\mathcal{A}\cap\{\theta\in[0,\theta_1)\},
$$
where $\mathcal{A}$ is the annulus in \eqref{EqnGluingAnnulus}.

\vem

In $\tilde{\mathcal{A}}$, we first estimate $\nabla \zeta\cdot\nabla\UIn$ as
\begin{align*}
|\nabla\zeta\cdot\nabla\UIn|&\le C_dAr^{-d}\cdot|r\BU-\frac12r^{-3/2}\BV|\\
&\le C_dAr^{-d+1}(|\BU|+r^{-5/2}|\BV|)\\
&\le C_d\theta_0^2.
\end{align*}
The last inequality follows from the bound on $\BU$ in Lemma \ref{LemGlueAS}, the definition of $\BV$ in  \eqref{EqnBV2} and the bound on $\RIn$ in \eqref{EqnRatioInOut}.

Similarly, in $\tilde{\mathcal{A}}$, the term $\nabla\zeta\cdot\nabla\UOut$ can be estimated as
$
|\nabla\zeta\cdot\nabla\UOut|\le C_d\theta_0^2.
$

As a result, in $\tilde{\mathcal{A}}$, we have
\begin{equation}
\label{EqnEstimateOnMixedTerm}
|\nabla\zeta\cdot\nabla\UIn|+|\nabla\zeta\cdot\nabla\UOut|\le\BO(\theta_0^2).
\end{equation} 

\vem

In $\tilde{\mathcal{A}}\backslash\Omega$, where $\Omega$ is the domain in \eqref{EqnAwayFromFBUpper}, we have $\UOut=V$ in \eqref{EqnHorizontalFillingUpper} and $\UIn=U(r,\theta_1)$. With Lemma \ref{LemNonlinearEquationInInterior} and Lemma \ref{LemSuperSolutionInInteriorInner}, we have
$$
\Delta W=\zeta\Delta\UIn+(1-\zeta)\Delta\UOut+\nabla\zeta\cdot\nabla\UIn-\nabla\zeta\cdot\nabla\UOut\le\frac15+\BO(\theta_0^2).
$$
Recall that $\beta<2$ is close to $2$ (see \eqref{EqnScalingParameter}), we have
$$
E(W)\le\frac15+\BO(\theta_0^2)-\frac{d+\beta-2}{d}<0 \text{ in }\tilde{\mathcal{A}}\backslash\Omega
$$
if $\kappa$ is small. 

\vem

It remains to consider the region $\tilde{\mathcal{A}}\cap\Omega$, where $\UOut=U$ and $\UIn=U(r,\theta_1)$.

By definition of $U$ in \eqref{EqnPerturbedSolution}, we have
$$
\Delta\UOut=\Delta u+\Delta v \text{ in }\tilde{\mathcal{A}}\cap\Omega.
$$
Using the definition of $\BV$ in \eqref{EqnBV2}, we have
$$
\Delta v\le r^{-5/2}\sigma(\theta-\theta_0)^{\delta-2}\cdot\{\delta(\delta-\frac12)\}.
$$
Recall we have $\ROut<r<\RIn$ with  \eqref{EqnRatioInOut}, this can be further estimated as
$$
\Delta v\le -c_\delta \text{ for }\theta\in(\theta_0,\theta_1).
$$
As a result, we have
$$
\Delta\UOut\le\Delta u-c_\delta  \text{ in }\tilde{\mathcal{A}}\cap\Omega.
$$

Together with Lemma \ref{LemSuperSolutionInInteriorInner}, we have
$$
\Delta W\le(1-\zeta)[\Delta u-c_\delta]+\BO(\theta_0^2)   \text{ in }\tilde{\mathcal{A}}\cap\Omega.
$$
Since $u$ is a solution to \eqref{EqnTransformedEquation}, we have
$$
|\Delta u-\frac{d+\beta-2}{d}|\sim_d\kappa.
$$ 
Thus 
$$
E(W)\le -c_\delta+\BO(\theta_0)+\BO(\kappa)\le0    \text{ in }\tilde{\mathcal{A}}\cap\Omega
$$
if $\kappa$ is small. 
\end{proof} 

With these preparations, an \textit{upper leaf} is given as a truncated version of $W$ in \eqref{EqnW}. This truncation is to remove the singularity at $0$. 

To be precise, we cut-off $W$ in \eqref{EqnW} as
\begin{equation}
\label{EqnUpperLeaf}
\Psi_1:=\begin{cases}
\min\{W,2(\sup_{\mathcal{A}}W+\sup_{B_{\RIn}}u)\} &\text{ in }\Omega_{\mathrm{in}},\\
W &\text{ in }\Omega_{\mathrm{out}},
\end{cases}
\end{equation} 
where the annulus $\mathcal{A}$ is defined in \eqref{EqnGluingAnnulus}, $\Omega_{\mathrm{in}}$ is from \eqref{EqnRegionNearTheOrigin}, and $\Omega_{\mathrm{out}}$ is from \eqref{EqnRegionNearInfinity}.

This profile has the desired properties:
\begin{lem}
\label{LemPropertiesOfPsi1}
For $d\ge3$ and $\delta,\kappa$ dimensionally small, we have
$$
\Psi_1>u \text{ in }\overline{\PosS}, \text{ and }E(\Psi_1)\le 0 \text{ in }\R^{d+1}.
$$
\end{lem} 
\begin{proof}
The ordering between $\Psi_1$ and $u$ follows from Lemma \ref{LemPropertyOfUOut}, Lemma \ref{LemOrderingInnerRegionUpper}, \eqref{EqnPerturbedSolution},  and \eqref{EqnW}.

Note that positive constants are strict supersolutions to \eqref{EqnTransformedEquation}, Lemma \ref{LemW} gives $E(\Psi_1)\le0$.
\end{proof} 

The leaf in \eqref{EqnUpperLeaf} generates an upper foliation. With this, we give the proof of Proposition \ref{PropUpperFoliationAS}.
\begin{proof}[Proof of Proposition \ref{PropUpperFoliationAS}]
For $t>0$, define
$$
\Psi_t(x,y)=t^2\Psi_1(x/t,y/t).
$$

With similar ideas as in the proof of Proposition \ref{PropLowerFoliationAS}, we have $\Psi_t>u $ in $\overline{\PosS}$ and $E(\Psi_t)\le 0$ in $\R^{d+1}$. Continuity of the map $(x,t)\mapsto\Psi_t(x)$ and the map $t\mapsto\partial\{\Psi_t>0\}$ follows from definition. 

With our definition \eqref{EqnUpperLeaf}, there is a constant $\eta>0$ such that $\Psi_1\ge\eta$ in $B_1$. This implies $\Psi_t\ge\eta t^2$ in $B_{t}$. Thus $\Psi_t\to+\infty$ locally uniformly as $t\to+\infty.$

Finally, given a compact subset $K$ of $\PosS$, we have $t_K>0$ such that $\Psi_t=u-t^{5/2}v$ for $t<t_K.$ This gives the convergence of $\Phi_t$ to $u$ as $t\to0.$ 
\end{proof} 

Combining the main results in Section \ref{SectionCAS}, Section \ref{SectionLowerFoliationAS} and this section, we have the proof of Theorem \ref{ThmMainAS}.
\begin{proof}[Proof of Theorem \ref{ThmMainAS}]
The construction of $\UAS$ and the estimate of the measure of its contact set were given in Proposition \ref{PropConstructionOfUAs}. 

With Theorem \ref{ThmMinimality}, its minimality follows from Proposition \ref{PropLowerFoliationAS} and Proposition \ref{PropUpperFoliationAS}.
\end{proof}

\appendix


\section{Expansion for solutions to second order ODEs}
\label{AppendixODE}
In this appendix, we collect some technicalities concerning several ODEs.
\subsection{Equation near the equator for the axially symmetric cone}
\label{AppendixODE1}
For a given function $\rho$ on $(0,\pi/2)$, we study the behavior of solutions to 
\begin{equation}
\label{EqnAppendixODE}
\LP u=\rho \text{ on }(0,\pi/2),  \text{ and }u'(\pi/2)=u(\pi/2)=0,
\end{equation} 
where the operator $\LP$ is given in \eqref{EqnLP}.

\vem

One of the fundamental solutions to the homogeneous ODE is given by $V_0$ from \eqref{EqnV0} with properties listed in Lemma \ref{LemPropertiesOfV0}.

The other fundamental solution, denoted by $\varphi$, solves
\begin{equation}
\label{EqnFundamentalSolution}
\LP\varphi=0 \text{ on }(0,\pi/2), \text{ and }\varphi'(\pi/2)=1, \varphi(\pi/2)=0.
\end{equation}
Liouville's formula gives
\begin{equation}
\label{EqnFirstLiouville}
\varphi'V_0-\varphi V_0'=\sin^{-d+1+\kappa}(\theta) \text{ on }(0,\pi/2),
\end{equation} 
that is,
\begin{equation}
\label{EqnLiouville}
(\varphi/V_0)'=\sin^{-d+1+\kappa}(\theta)/V_0^2 \text{ on }(0,\pi/2).
\end{equation} 

As a consequence, we have the following properties of  $\varphi$:
\begin{lem}
\label{LemPropertiesOfPhi}
For $d\ge3$ and $\kappa\in(0,\frac14)$, we have
$$
|\varphi|(\theta)\le \BO(\theta^{-d+2+\kappa}) \text{ and }\hem|\varphi'|(\theta)\le \BO(\theta^{-d+1+\kappa}) \text{ on }(0,\pi/2).
$$

Moreover, we have the following expansion
$$
\varphi'(\theta)=-\frac{1}{d}\sin^{-d+1}(\theta)(1+\kappa \BO(|\log(\theta)|)+\BO(\theta^2))
 \text{ on }(0,\pi/4).
$$
\end{lem} 
Recall the big-O notation from Remark \ref{RemSimAndO}.

\begin{proof}
\textit{Step 1: Upper bounds on $\varphi$ and $\varphi'$.}

Note that there is no singularity in the operator $\LP$ on $[\pi/4,\pi/2]$, standard analysis gives
$$
|(\varphi',\varphi)|\le C_d \text{ on }[\pi/4,\pi/2].
$$

Between $0$ and $\pi/4$, we integrate \eqref{EqnLiouville} to obtain
$$
\frac{\varphi(\theta)}{V_0(\theta)}=\frac{\varphi(\pi/4)}{V_0(\pi/4)}-\int_\theta^{\pi/4}\sin^{-d+1+\kappa}(\tau)V_0^{-2}(\tau)d\tau.
$$
From Lemma \ref{LemPropertiesOfV0}, we see that $V_0$ is  bounded away from $0$ and infinity on $[0,\pi/4]$, the desired bound on $\varphi$ follows. 

With this, the bound on $\varphi'$ follows from \eqref{EqnFirstLiouville}.

\vem

\textit{Step 2: The expansion of $\varphi'$.}

With the bound on $\varphi$ in \textit{Step 1} and the bound on $V_0'$ in Lemma \ref{LemPropertiesOfV0}, we can use \eqref{EqnFirstLiouville} to get
$$
\varphi'V_0=\sin^{-d+1+\kappa}(\theta)[1+\BO(\theta^2)].
$$
With $|\sin^\kappa(\theta)-1|=\kappa \BO(|\log(\theta)|)$ and the regularity of $V_0$, we conclude
$$
\varphi'(\theta)=-\frac{1}{d}\sin^{-d+1}(\theta)(1+\kappa \BO(|\log(\theta)|)+\BO(\theta^2)).
$$
\end{proof} 

By the method of variation of parameters, we have the following representation:
\begin{lem}
\label{LemVariationOfParameters}
Suppose that $u$ is a solution to \eqref{EqnAppendixODE}, then
$$
u(\theta)=\varphi(\theta) \alpha(\theta)+V_0(\theta) \beta(\theta), \text{ and }u'(\theta)=\varphi'(\theta)\alpha(\theta)+V_0'(\theta)\beta(\theta),
$$
where 
$$
\alpha(\theta)=-\int_{\theta}^{\pi/2}\rho(\tau)V_0(\tau)\sin^{d-1-\kappa}(\tau)d\tau
$$
and
$$
\beta(\theta)=\int_{\theta}^{\pi/2}\rho(\tau)\varphi(\tau)\sin^{d-1-\kappa}(\tau)d\tau.
$$
\end{lem} 

\subsection{Equation near the south pole for the axially symmetric cone} 
Given $\lambda\in(0,\pi/2)$ and a function $\rho$,  we study solutions to 
\begin{equation}
\label{EqnAppendixODE2}
Lu=\rho \text{ on } (0,\pi/2), \hem \text{ and }u'(\lambda)=u(\lambda)=0,
\end{equation} 
where the operator $L$ is given by
$$
Lu:=u''+(d-1)\cot(\theta)u'.
$$

The constant function $1$ is a fundamental solution to the homogeneous equation $Lu=0$. The other fundamental solution can be taken as
\begin{equation}
\label{EqnFundamentalSolutionODE2}
\varphi(\theta):=\int_\theta^{\pi/2}\sin^{-d+1}(\tau)d\tau \hem\text{ with }\hem
\varphi'(\theta)=-\sin^{-d+1}(\theta).
\end{equation}

With the method of variation of parameters, we have
\begin{lem}
\label{LemVariationOfParameters2}
Suppose that $u$ is a solution to \eqref{EqnAppendixODE2}, then
$$
u(\theta)=\varphi(\theta) \alpha(\theta)+\beta(\theta), \text{ and }u'(\theta)=\varphi'(\theta)\alpha(\theta),
$$
where 
$$
\alpha(\theta)=-\int_{\lambda}^{\theta}\rho(\tau)\sin^{d-1}(\tau)d\tau
$$
and
$$
\beta(\theta)=\int_{\lambda}^{\theta}\rho(\tau)\varphi(\tau)\sin^{d-1}(\tau)d\tau.
$$

\end{lem}

\section{Linear expansion of a nonlinearity}
\label{AppendixLinearExpansion}
The nonlinearity we face in several equations is given by $h:\R^{d+1}\to\R$  of the form
$$
h(X)=h(X',x_{d+1})=\frac{|X'|^2}{x_{d+1}} \text{ for }x_{d+1}>0.
$$
Direct computation gives
$$
\nabla h(X)=(2X'/x_{d+1},-|X'|^2/x_{d+1}^2).
$$

We bound the error in the linear expansion:
\begin{lem}
\label{LemAppendixLinearExpansion}
Given a constant $H>0$, for $X\in\{h\le H\}$, we have
\begin{align*}
|h(X+Z)-h(X)-\nabla h(X)\cdot Z|&\le H\cdot (\frac{z_{d+1}}{x_{d+1}})^2[\frac{x_{d+1}}{x_{d+1}+z_{d+1}}+(\frac{x_{d+1}}{x_{d+1}+z_{d+1}})^2]\\
&+\frac{|Z'|^2}{x_{d+1}}(1+\frac{x_{d+1}}{x_{d+1}+z_{d+1}})
\end{align*}
as long as $x_{d+1}>0$ and $x_{d+1}+z_{d+1}>0.$
\end{lem} 

\begin{proof}
Direct expansion gives
\begin{align*}
h(X+Z)-h(X)-\nabla h(X)\cdot Z&=|X'|^2[\frac{1}{x_{d+1}+z_{d+1}}-\frac{1}{x_{d+1}}+\frac{z_{d+1}}{x_{d+1}^2}]\\
+|Z'|^2[\frac{1}{x_{d+1}+z_{d+1}}]&+2X'\cdot Z'[\frac{-z_{d+1}}{x_{d+1}(x_{d+1}+z_{d+1})}].
\end{align*}
We deal with each term on the right-hand side.

For the first term, we note that 
$$
\frac{1}{x_{d+1}+z_{d+1}}-\frac{1}{x_{d+1}}+\frac{z_{d+1}}{x_{d+1}^2}=\frac{z_{d+1}^2}{x_{d+1}^2(x_{d+1}+z_{d+1})}.
$$
Thus the first term can be bounded as
\begin{align*}
||X'|^2[\frac{1}{x_{d+1}+z_{d+1}}-\frac{1}{x_{d+1}}+\frac{z_{d+1}}{x_{d+1}^2}]|&=\frac{|X'|^2}{x_{d+1}}(\frac{z_{d+1}}{x_{d+1}})^2\frac{x_{d+1}}{x_{d+1}+z_{d+1}}\\
&\le H(\frac{z_{d+1}}{x_{d+1}})^2\frac{x_{d+1}}{x_{d+1}+z_{d+1}}
\end{align*}
by our assumption on $X$.

The second term is
$$
|Z'|^2[\frac{1}{x_{d+1}+z_{d+1}}]=\frac{|Z'|^2}{x_{d+1}}\cdot \frac{x_{d+1}}{x_{d+1}+z_{d+1}}.
$$

For the last term, we note 
\begin{align*}
|2X'\cdot Z'[\frac{-z_{d+1}}{x_{d+1}(x_{d+1}+z_{d+1})}]|&=2|\frac{X'}{x_{d+1}^{1/2}}\frac{z_{d+1}}{x_{d+1}}(\frac{x_{d+1}}{x_{d+1}+z_{d+1}})\cdot\frac{Z'}{x_{d+1}^{1/2}}|\\
&\le |\frac{X'}{x_{d+1}^{1/2}}\frac{z_{d+1}}{x_{d+1}}(\frac{x_{d+1}}{x_{d+1}+z_{d+1}})|^2+|\frac{Z'}{x_{d+1}^{1/2}}|^2\\
&\le H(\frac{z_{d+1}}{x_{d+1}})^2(\frac{x_{d+1}}{x_{d+1}+z_{d+1}})^2+\frac{|Z'|^2}{x_{d+1}}.
\end{align*}

Putting all these back into the first equation, we get the desired bound. 
\end{proof}

\section{Quantities on the interface}
\label{AppendixQuantitiesOnGamma1}
In Section \ref{SectionLowerFoliationAS}, we constructed a  lower foliation of the axially symmetric cone $\UAS$. We need to divide the space into two regions, one close to the free boundary and one away from the free boundary. In this appendix, we perform some technical computations on the interfaces separating these regions.  

 To be precise, we focus on the interface $\Gamma$ from \eqref{EqnGammaLower}, described by either \eqref{EqnRadialGraphLower} or \eqref{EqnYGraphLower}. 

We first prove Lemma \ref{LemTechnical1}.
\begin{proof}[Proof of Lemma \ref{LemTechnical1}]
\textit{Step 1: The bound on $\frac{d\theta}{dy}$.}

Differentiating both sides of  \eqref{EqnRadialGraphLower}, we get
\begin{equation}
\label{EqnDRinR}
\frac{dR}{d\theta}=\frac25R[\frac{\BV'}{\BV}-\frac{\BU'}{\BU}].
\end{equation} 
Taking the $y$-derivative of the first equation in \eqref{EqnYGraphLower}, we get
\begin{equation}
\label{EqnDTDY}
1=-[R'\cos(\theta)-R\sin(\theta)]\frac{d\theta}{dy}=R\frac{d\theta}{dy}\cdot\{\frac25[\frac{\BU'}{\BU}-\frac{\BV'}{\BV}]\cos(\theta)+\sin(\theta)\}.
\end{equation} 

With Lemma \ref{LemCommutator}, we see that $[\frac{\BU'}{\BU}-\frac{\BV'}{\BV}]\ge0$. The previous equality implies $\frac{d\theta}{dy}\ge0.$ 

Moreover, Lemma \ref{LemCommutator} implies
$$
\frac25[\frac{\BU'}{\BU}-\frac{\BV'}{\BV}]\cos(\theta)+\sin(\theta)\ge c_d\cos^2(\theta)+\sin(\theta)\ge c_d
$$
for a dimensional constant $c_d>0$. It follows from \eqref{EqnDTDY} that 
\begin{equation}
\label{EqnDThetaDY}
\frac{d\theta}{dy}\le \BO(1/R).
\end{equation}

\vem

\textit{Step 2: The bound on $\frac{d^2\theta}{dy^2}$.}

Differentiating \eqref{EqnDTDY} in the $y$-variable, we get, on $(\theta_0,\theta_1)\cup(\theta_1,\pi/2)$,
$$
0=\frac{d^2\theta}{dy^2}R\cdot \{\frac25[\frac{\BU'}{\BU}-\frac{\BV'}{\BV}]\cos(\theta)+\sin(\theta)\}+(\frac{d\theta}{dy})^2\frac{d}{d\theta}\{R\cdot[\frac25(\frac{\BU'}{\BU}-\frac{\BV'}{\BV})\cos(\theta)+\sin(\theta)]\}.
$$
To simplify our notations, take 
$$
Q:=|R\cdot \{\frac25[\frac{\BU'}{\BU}-\frac{\BV'}{\BV}]\cos(\theta)+\sin(\theta)\}|.
$$

With \eqref{EqnDRinR}, Lemma \ref{LemGlueAS} and Lemma \ref{LemCommutator}, we have, 
\begin{align*}
|R'[\frac25(\frac{\BU'}{\BU}-\frac{\BV'}{\BV})\cos(\theta)+\sin(\theta)]|&\le R\BO((\theta-\theta_0)^{-1})\cdot [(\frac{\BU'}{\BU}-\frac{\BV'}{\BV})\cos(\theta)+\sin(\theta)]\\
&\le Q \BO((\theta-\theta_0)^{-1}).
\end{align*}
With Lemma \ref{LemGlueAS} and Lemma \ref{LemCommutator}, we have, on $(\theta_0,\theta_1)\cup(\theta_1,\pi/2)$,
\begin{align*}
|R\cdot\frac{d}{d\theta}[\frac25(\frac{\BU'}{\BU}-\frac{\BV'}{\BV})\cos(\theta)+\sin(\theta)]|&\le R\cdot[\cos(\theta)\BO((\theta-\theta_0)^{-2})+\sin(\theta)\BO((\theta-\theta_0)^{-1})]\\
&\le Q\BO((\theta-\theta_0)^{-1}).
\end{align*}

Putting these two estimates into the first equation of this step, we get
$$
|\frac{d^2\theta}{dy^2}|Q\le (\frac{d\theta}{dy})^2Q\BO((\theta-\theta_0)^{-1}) \text{ on }(\theta_0,\theta_1)\cup(\theta_1,\pi/2).
$$
This implies
$$
|\frac{d^2\theta}{dy^2}|\le (\frac{d\theta}{dy})^2 \BO((\theta-\theta_0)^{-1})  \text{ on }(\theta_0,\theta_1)\cup(\theta_1,\pi/2).
$$

\vem

\textit{Step 3: The bound on $\frac{dL}{dy}$ and $\frac{d^2L}{dy^2}$.}

Differentiating the second equation in \eqref{EqnYGraphLower} and using \eqref{EqnDRinR}, we get
\begin{equation}
\label{EqnDLDY}
\frac{dL}{dy}=R\frac{d\theta}{dy}\cdot\{\frac25[\frac{\BV'}{\BV}-\frac{\BU'}{\BU}]\sin(\theta)+\cos(\theta)\}.
\end{equation}
With \eqref{EqnDThetaDY} and Lemma \ref{LemCommutator}, we see the right-hand side is bounded by a dimensional constant. 

Differentiating \eqref{EqnDLDY}, we get
$$
\frac{d^2L}{dy^2}=\frac{d^2\theta}{dy^2}\cdot R\{\frac25[\frac{\BV'}{\BV}-\frac{\BU'}{\BU}]\sin(\theta)+\cos(\theta)\}+(\frac{d\theta}{dy})^2\frac{d}{d\theta}\{R[\frac25(\frac{\BV'}{\BV}-\frac{\BU'}{\BU})\sin(\theta)+\cos(\theta)]\}.
$$

With Lemma \ref{LemCommutator}, the first term is  bounded by $R|\frac{d^2\theta}{dy^2}|\BO(\theta(\theta-\theta_0)^{-1})$. This is in turn controlled by $R(\frac{d\theta}{dy})^2\BO(\theta(\theta-\theta_0)^{-2})$ on $(\theta_0,\theta_1)\cup(\theta_1,\pi/2)$ with the estimate from Step 2. 

For the second term, we use \eqref{EqnDRinR} and Lemma \ref{LemCommutator} to bound it by $R(\frac{d\theta}{dy})^2\BO(\theta(\theta-\theta_0)^{-2})$ on $(\theta_0,\theta_1)\cup(\theta_1,\pi/2)$.

Combining these, we get the desired bound on $\frac{d^2L}{dy^2}$ on $(\theta_0,\theta_1)\cup(\theta_1,\pi/2)$.

\vem

\textit{Step 4: The lower bound on $|\frac{dL}{dy}|.$}

With Lemma \ref{LemCommutator}, we see that
$$
|[\frac{\BV'}{\BV}-\frac{\BU'}{\BU}]\sin(\theta)|\ge c_d\theta(\theta-\theta_0)^{-1}\ge c_d(1+\eta^{-1}) \text{ on }(\theta_0,(1+\eta)\theta_0).
$$

For $\eta$ dimensionally small, this gives 
$$
|\frac25[\frac{\BV'}{\BV}-\frac{\BU'}{\BU}]\sin(\theta)+\cos(\theta)|\ge c_d\eta^{-1}\theta(\theta-\theta_0)^{-1} \text{ on }(\theta_0,(1+\eta)\theta_0).
$$
Together with \eqref{EqnDLDY}, this gives the desired lower bound on $|\frac{dL}{dy}|$.
\end{proof}

Now we turn to Lemma \ref{LemTechnical2}.

\begin{proof}[Proof of Lemma \ref{LemTechnical2}]
From \eqref{EqnGammaLower}, along $y\mapsto(L(y),y)$ in \eqref{EqnYGraphLower}, we have
$$
\sqrt{U}=\sqrt{\kappa^\frac18 u}=\kappa^{1/16}R(\theta)\BU^{1/2}(\theta).
$$
Differentiating this in the $\theta$-variable, we use \eqref{EqnDRinR} to get
\begin{equation}
\label{EqnTechnicalTwo}
\frac{d\sqrt{U}}{d\theta}=\kappa^{1/16}\BU^{1/2}R\cdot[\frac25\frac{\BV'}{\BV}+\frac{1}{10}\frac{\BU'}{\BU}]\le R\BO(\kappa^{1/16}).
\end{equation}
Note that we used Lemma \ref{LemGlueAS} and Lemma \ref{LemCommutator}.
The chain rule gives
$$
|\frac{d\sqrt{U}}{dy}|=R\frac{d\theta}{dy}\BO(\kappa^{1/16}).
$$
\vem

Differentiating \eqref{EqnTechnicalTwo}, we get, on $(\theta_0,\theta_1)\cup(\theta_1,\pi/2)$,
\begin{align*}
\frac{d^2\sqrt{U}}{d\theta^2}&=\frac25\kappa^{1/16}R\BU^{1/2}\cdot(\frac{\BV'}{\BV}-\frac{\BU'}{\BU})[\frac25\frac{\BV'}{\BV}+\frac{1}{10}\frac{\BU'}{\BU}]\\
&+\frac12\kappa^{1/16}R\frac{\BU'}{\BU^{1/2}}\cdot[\frac25\frac{\BV'}{\BV}+\frac{1}{10}\frac{\BU'}{\BU}]\\
&+\kappa^{1/16}R\BU^{1/2}[\frac25(\frac{\BV'}{\BV})'+\frac{1}{10}(\frac{\BU'}{\BU})'].
\end{align*}
With Lemma \ref{LemGlueAS} and Lemma \ref{LemCommutator}, each term on the right-hand side is of the order $\kappa^{1/16}R\BO((\theta-\theta_0)^{-1})$. 
As a result, we have
\begin{equation}
\label{EqnTechnical3}
|\frac{d^2\sqrt{U}}{d\theta^2}|\le \kappa^{1/16}R\BO((\theta-\theta_0)^{-1}) \text{ on }(\theta_0,\theta_1)\cup(\theta_1,\pi/2).
\end{equation}

With Lemma \ref{LemTechnical1}, we have, on $(\theta_0,\theta_1)\cup(\theta_1,\pi/2)$,
$$
\frac{d^2\sqrt{U}}{dy^2}=\frac{d^2\sqrt{U}}{d\theta^2}(\frac{d\theta}{dy})^2+\frac{d\sqrt{U}}{d\theta}\frac{d^2\theta}{dy^2}=(\frac{d\theta}{dy})^2\cdot\{\frac{d^2\sqrt{U}}{d\theta^2}+\frac{d\sqrt{U}}{d\theta}\BO((\theta-\theta_0)^{-1}).
$$
The conclusion follows from \eqref{EqnTechnicalTwo} and \eqref{EqnTechnical3}.
\end{proof}


\end{document}